\documentclass[aos]{imsart}

\usepackage[utf8]{inputenc}

\usepackage{amsmath,amsbsy,amssymb,amsthm}
\usepackage{multirow,multicol,tabularx,booktabs}
\usepackage{graphicx,placeins,color,url}
\usepackage{bbm,bm}
\usepackage[ruled]{algorithm2e}
\usepackage[shortlabels]{enumitem}
\usepackage{wrapfig}
\usepackage{mathrsfs}
\usepackage{textcase}

\startlocaldefs
\allowdisplaybreaks

\usepackage{tikz}
\usetikzlibrary{arrows,arrows.meta,positioning,calc}
\usepackage{pgfplots}
\pgfplotsset{compat=1.15}
\usepgfplotslibrary{fillbetween}

\newtheorem{conj}{Conjecture}[section]
\newtheorem{thm}[conj]{\bf Theorem}

\newtheorem{cor}[conj]{\bf Corollary}
\newtheorem{prop}[conj]{\bf Proposition}
\newtheorem{lemma}[conj]{\bf Lemma}

\newtheorem{example}[conj]{\bf Example}

\def\eps{\varepsilon}

\def\bar{\overline}

\def\sign{{\; \text{sign}}}

\def\to{\rightarrow}

\def\simiid{\buildrel {\text{i.i.d.}} \over \sim}

\def\Bias{\operatorname{Bias}}
\def\IB{\operatorname{IBias}^2}
\def\Cov{\operatorname{Cov}}
\def\Var{\operatorname{Var}}
\def\IVar{\operatorname{IVar}}
\def\MSE{\operatorname{MSE}}
\def\MISE{\operatorname{MISE}}

\def\supp{\operatorname{supp}}
\def\JS{\operatorname{JS}}
\def\Tr{\operatorname{Tr}}

\def\TV{\operatorname{TV}}
\def\KL{\operatorname{KL}}
\def\Jac{\operatorname{Jac}}
\def\Det{\operatorname{Det}}
\def\mc{\operatorname{mc}}

\def\om{\omega}

\def\Ac{\mbox{$\mathcal A$}}

\def\Cc{{\mathscr C}}

\def\Fc{{\mathscr F}}
\def\Gc{{\mathscr G}}

\def\Mc{{\mathcal M}}
\def\Nc{{\mathcal N}}

\def\Sc{{\mathcal S}}
\def\Tc{{\mathcal T}}

\def\Xc{{\mathcal X}}

\def\Rb{{\mathbb R}}

\def\v{ {\bf v}}

\def\1{\mathbbm{1}}
\def\floorbeta{{\lfloor \beta \rfloor}}

\def\Mk{{\mathfrak{M}}}

\def\Om{\mathcal{O}_m}

\def\wh{\widehat}
\def\wt{\widetilde}
\def\ol{\overline}

\endlocaldefs

\begin{document}

\begin{frontmatter}

% "Title of the paper"
\title{On lower bounds for the bias-variance trade-off}
\runtitle{On lower bounds for the bias-variance trade-off}

% indicate corresponding author with \corref{}
% \author{\fnms{John} \snm{Smith}\corref{}\ead[label=e1]{smith@foo.com}\thanksref{t1}}
% \thankstext{t1}{Thanks to somebody} 
% \address{line 1\\ line 2\\ printead{e1}}
% \affiliation{Some University}

\author{\fnms{Alexis} \snm{Derumigny}\ead[label=e1]{a.f.f.derumigny@tudelft.nl}
\and 
\fnms{Johannes} \snm{Schmidt-Hieber}\ead[label=e2]{a.j.schmidt-hieber@utwente.nl}
}
\address{Delft University of Technology, \\
Mekelweg 4, \\ 2628 CD Delft \\ The Netherlands.\\ \printead{e1}}
\address{University of Twente,\\ P.O. Box 217, \\ 7500 AE Enschede \\ The Netherlands\\ \printead{e2}}
\affiliation{Delft University of Technology \and University of Twente}

\runauthor{A. Derumigny and J. Schmidt-Hieber}

\begin{abstract}
It is a common phenomenon that for high-dimensional and nonparametric statistical models, rate-optimal estimators balance squared bias and variance. Although this balancing is widely observed, little is known whether methods exist that could avoid the trade-off between bias and variance. We propose a general strategy to obtain lower bounds on the variance of any estimator with bias smaller than a prespecified bound. This shows to which extent the bias-variance trade-off is unavoidable and allows to quantify the loss of performance for methods that do not obey it. The approach is based on a number of abstract lower bounds for the variance involving the change of expectation with respect to different probability measures as well as information measures such as the Kullback-Leibler or $\chi^2$-divergence. Some of these inequalities rely on a new concept of information matrices. In a second part of the article, the abstract lower bounds are applied to several statistical models including the Gaussian white noise model, a boundary estimation problem, the Gaussian sequence model and the high-dimensional linear regression model. For these specific statistical applications, different types of bias-variance trade-offs occur that vary considerably in their strength. For the trade-off between integrated squared bias and integrated variance in the Gaussian white noise model, we propose to combine the general strategy for lower bounds with a reduction technique. This allows us to reduce the original problem to a lower bound on the bias-variance trade-off for estimators with additional symmetry properties in a simpler statistical model. In the Gaussian sequence model, different phase transitions of the bias-variance trade-off occur. Although there is a non-trivial interplay between bias and variance, the rate of the squared bias and the variance do not have to be balanced in order to achieve the minimax estimation rate.
% To highlight possible extensions of the proposed framework, we moreover briefly discuss the trade-off between bias and mean absolute deviation. 
\end{abstract}



\end{frontmatter}

\section{Introduction}
Can the bias-variance trade-off be avoided, for instance by using machine learning methods in the overparametrized regime? This is currently debated in machine learning. While older work on neural networks mention that ``the fundamental limitations resulting from the bias-variance dilemma apply to all nonparametric inference methods, including neural networks'' (\cite{Geman1992}, p.45), the very recent work on overparametrization in machine learning has cast some doubt on the necessity to balance squared bias and variance \cite{belkin2018reconciling, neal2018modern}. While for fixed and moderate growth, the number of parameters in the method (e.g. the number of network parameters in a neural network) can be associated to the bias and the variance of the procedure, resulting in the well-known U-shaped curves for the statistical risk (see e.g. Figure 2.11 in \cite{MR2722294}), such a link cannot be made in the overparametrized regime. But this does not mean that the bias-variance trade-off disappears. In this work we prove that for standard estimation problems in nonparametric and high-dimensional statistics, there are universal bias-variance trade-offs that cannot be circumvented by any method.

Besides the debate about overparametrization, there are many other good reasons why a better understanding of the bias-variance trade-off is relevant for statistical practice. Even in non-adaptive settings, confidence sets in nonparametric statistics require control on the bias of the centering estimator and often use a slight undersmoothing to make the bias negligible compared to the variance. If rate-optimal estimators with negligible bias would exist, such troubles could be overcome. In some instances, small bias is possible. An important example is the rather subtle de-biasing of the LASSO for a class of functionals in the high-dimensional regression model \cite{MR3153940, MR3224285, MR3650395}. This shows that the occurrence of the bias-variance trade-off is a highly non-trivial phenomenon.

Finite-dimensional parametric models do typically not exhibit a bias-variance trade-off and there may exist unbiased estimators with finite variance. On the contrary, our results shows that for high-dimensional and infinite-dimensional statistical models, unbiased estimators with finite variance are in almost all of the considered settings impossible. The fundamental difference lies in the amount of information per parameter: for parametric models of dimension $p$, the sample size $n$ is, by definition, of a larger order than $p$, and the statistician has a budget of $n/p$ observations per parameter; on the contrary, for nonparametric models, we have $p>n$ or even $p=+\infty$ and there is simply not enough data to estimate each parameter well using a $n/p$\,-fraction of the observations. For example, in the Gaussian white noise model, we observe the process $(Y_x)_x$ satisfying $dY_x = f(x) \, dx + n^{-1/2} \, dW_x$ for an unknown function $f$. If the regression function $f$ lies in a nonparametric class, it is impossible to transform the data into the form $f(x_0)+$'noise'. Instead one has to rely here on the similarity of the regression function in a small vicinity around $x_0,$ which leads to an unavoidable bias. 
 
%However, in the general case and for a given point $x_0 \in (0,1)$, estimators of $f(x_0)$ of the form $\int_{x_0-h}^{x_0+h} Y_t dt$ are necessarily biased upwards if $f$ is convex and biased downwards if $f$ is concave. By reducing $h$ we can decrease the bias at the price of an increased variance. Our results show that there does not exists any possibility of escaping this trade-off with the given information: we do need some form of smoothing and this necessarily add some kind of bias.

Only few theoretical articles exist on lower bounds for the interplay between bias and variance. The major contribution is due to Mark Low \cite{low1995bias} proving that the bias-variance trade-off is unavoidable for estimation of functionals in the Gaussian white noise model. The approach relies on a complete characterization of the bias-variance trade-off phenomenon in a parametric Gaussian model via the Cram\'er-Rao lower bound, see also Section \ref{sec:gauss_wn} for a more in-depth discussion. Another related result is \cite{MR1836577}, also considering estimation of functionals but not necessarily in the Gaussian white noise model. It is shown that for any functional $\kappa$, a lower bound on the asymptotic deviation probability $\lim_{u \to 0} \liminf_{n \to \infty} P_0^n \big(c_n |\wh \kappa - \kappa(P_0)| \leq u \big)$ implies an asymptotic lower bound on variance-like measures of the estimator $\wh \kappa_n$ of $\kappa(P_0)$. In this article, we do not consider such deviation probability and establish direct and non-asymptotic trade-offs between bias and variance. \cite{liu1993nonexistence} introduces a notion of singular functional estimation problems and proves that for such singular problems, no unbiased estimators with finite variance exist. In the same spirit, \cite{chen2004notes} shows that the supremum of the variance of an unbiased estimator is infinite if a singular point belongs to the closure of the parameter set. Moreover, it is shown that the difference between biases is lower-bounded if the worst-case variance is upper-bounded.

In this article, we propose a general strategy to derive lower bounds for the bias-variance trade-off. The key ingredient are general inequalities bounding the change of expectation with respect to different distributions by the variance and information measures such as the total variation, Hellinger distance, Kullback-Leibler divergence and the $\chi^2$-divergence. 

%It should also be mentioned that the widely available minimax lower bounds are of little help. Having a lower bound on the mean squared error implies by the well-known bias-variance decomposition a lower bound for the sum of squared bias and variance. This does not exclude of course the possibility that the bias is always small. 

%The aim of this paper is to establish a general approach to prove lower bounds for the bias-variance trade-off. 

%The concept of minimax lower bounds holds for parametric and non-parametric problems alike. In many situations such as regression problems, unbiased estimation is possible for parametric problems and the bias-variance trade-off only kicks-in if either additional constraints on the parameters are imposed or if the parameter space becomes high/infinite-dimensional. The bias-variance trade-off is in this sense a purely nonparametric phenomenon. 

As examples, we consider nonparametric estimation in the Gaussian white noise model as well as sparse recovery in the sequence model and the high-dimensional linear regression model. By applying the lower bounds to different statistical models, it is surprising to see different types of bias-variance trade-offs occurring. The weakest type are worst-case scenarios stating that if the bias is small for all parameters, then there exists a potentially different parameter in the parameter space with a large variance and vice versa. For the pointwise estimation in the Gaussian white noise model, the derived lower bounds imply also a stronger version proving that small bias for all parameters will necessarily inflate the variance for all parameters that are in a suitable sense separated away from the boundary of the parameter space.

We also study lower bounds for the trade-off between integrated squared bias and integrated variance in the Gaussian white noise model. In this case a direct application of the multiple parameter lower bound is rather tricky and we propose instead a two-fold reduction first. The first reduction shows that it is sufficient to prove a lower bound on the bias-variance trade-off in a related sequence model. The second reduction states that it is enough to consider estimators that are constrained by some additional symmetry property. After the reductions, a few lines argument applying the information matrix lower bound is enough to derive a matching lower bound for the trade-off between integrated squared bias and integrated variance.

For function estimation in the Gaussian white noise model, the variance blows up if the estimator is constrained to have a bias decreasing faster than the minimax rate. In the sparse sequence model and the high-dimensional regression model with sparsity $\ll \sqrt{n},$ a different phenomenon occurs. For estimators with bias bounded by constant$\times$minimax rate, the derived lower bounds show that a sufficiently small constant already enforces that the variance must be larger than the minimax rate by a polynomial factor in the sample size. Interestingly, for an estimator achieving the minimax estimation rate, the rate of the variance can be of a smaller order than the rate of the squared bias and therefore, variance and squared bias do not need to be balanced.

Summarizing the results, for all of the considered models a non-trivial bias-variance trade-off could be established. For some estimation problems, the bias-variance trade-off only holds in a worst-case sense and, on subsets of the parameter space, rate-optimal methods with negligible bias exist. It should also be emphasized that for this work only non-adaptive setups are considered. Adaptation to either smoothness or sparsity induces additional bias. 
% 
% 
% 
%As mentioned above, the main motivation for this work is to test whether the new regimes found in modern machine learning could avoid the classical bias-variance trade-off.  
% 
The bias-variance trade-off problem can also be rephrased by asking for the optimal estimation rate if only estimators with, for instance, small bias are allowed. In this sense, the work contributes to the growing literature on optimal estimation rates under constraints on the estimators. So far, major theoretical work has been done for polynomial time computable estimators \cite{MR3127849, pmlr-v30-Berthet13}, lower and upper bounds for estimation under privacy constraints \cite{6686179, 2018arXiv180501422R, MR4227314}, and parallelizable estimators under communication constraints \cite{NIPS2013_4902, MR4134798}.

The paper is organized as follows. In Section \ref{sec.general_low_bds}, we provide a number of new abstract lower bounds, where we distinguish between inequalities bounding the change of expectation for two distributions and inequalities involving an arbitrary number of expectations. The subsequent sections of the article study lower and upper bounds for the bias-variance trade-off based on these inequalities. The considered setups range from pointwise estimation in the Gaussian white noise model (Section \ref{sec:gauss_wn} and Section \ref{sec.reduction}) and a boundary estimation problem (Section \ref{sec.boundary}) to high-dimensional models in Section \ref{sec.highdimensional}.
% Section \ref{sec.mean_deviation} serves as an outlook to generalizations of the bias-variance trade-off. Specifically, we study the mean absolute deviation and derive a lower bound for the trade-off between bias and mean absolute deviation considering again pointwise estimation in the Gaussian white noise model.
Section \ref{sec.discussion} discusses some aspects underlying a formal definition of the bias-variance trade-off and the connection between the approach in this work and minimax lower bounds. All proofs are deferred to the Supplement. 

\medskip

\noindent
{\it Notation:} Whenever the domain $D$ is clear from the context, we write $\|\cdot \|_p$ for the $L^p(D)$-norm. Moreover, $\|\cdot\|_2$ denotes also the Euclidean norm for vectors.
We denote by $A^\top$ the transpose of a matrix $A$.
For mathematical expressions involving several probability measures, it is assumed that those are defined on the same measurable space. If $P$ is a probability measure, we write $E_P$ and $\Var_P$ for the expectation and variance with respect to $P,$ respectively. For probability measures $P_\theta$ depending on a parameter $\theta,$ $E_\theta$ and $\Var_\theta$ denote the corresponding expectation and variance. Throughout the article, we consider estimators $\wh \theta$ for which the expectation $E_\theta[\wh \theta]$ exists and is finite for all parameters $\theta$ in the parameter space. This guarantees that the bias is always well-defined. If a random variable $X$ is not square integrable with respect to $P$, we assign the value $+\infty$ to $\Var_P(X).$ For any finite number of measures $P_1,\dots,P_M,$ defined on the same measurable space, we can find a measure $\nu$ dominating all of them (e.g. $\nu := \tfrac1M \sum_{j=1}^M P_j$). Henceforth, $\nu$ will always denote a dominating measure and $p_j$ stands for the $\nu$-density of $P_j.$ The total variation is $\TV(P,Q) := \tfrac 12 \int |p(\om)-q(\om)| \, d\nu(\om).$ The squared Hellinger distance is defined as $H(P,Q)^2 := \tfrac12 \int (\sqrt{p(\om)}-\sqrt{q (\om)})^2 \, d\nu(\om)$ (in the literature sometimes also defined without the factor $1/2$). If $P$ is dominated by $Q$, the Kullback-Leibler divergence is defined as $\KL(P,Q) := \int \log(p(\om)/q(\om)) p(\om) \, d\nu(\om)$ and the $\chi^2$-divergence is defined as $\chi^2(P,Q) := \int (p(\om)/q(\om) -1)^2 q(\om) \, d\nu(\om).$  If $P$ is not dominated by $Q,$ both Kullback-Leibler and $\chi^2$-divergence are assigned the value $+\infty.$

\section{General lower bounds on the variance}
\label{sec.general_low_bds}

\subsection{Lower bounds based on two distributions}

Given an upper bound on the bias, the goal is to find a lower bound on the variance. For parametric models, the natural candidate is the Cram\'er-Rao lower bound. Given a statistical model with real parameter $\theta \in \Theta \subseteq \Rb,$ and an estimator $\wh \theta$ with bias $B(\theta):=E_\theta[\wh \theta]-\theta,$ variance $V(\theta):=\Var_\theta(\wh \theta),$ and Fisher information $F(\theta),$ the Cram\'er-Rao lower bound states that 
$V(\theta) \geq \frac{(1+B'(\theta))^2}{F(\theta)},$
where $B'(\theta)$ denotes the derivative of the bias with respect to $\theta.$ The basic idea is that if the bias is small, we cannot have $B'(\theta)\leq -1/2$ everywhere, so there must be a parameter $\theta^*$ such that $V(\theta^*)\geq 1/(4F(\theta^*)).$ The constant $-1/2$ could be replaced of course by any other number in $(-1,0)$. There are various extensions of the Cram\'er-Rao lower bound to multivariate and semi-parametric settings \cite{MR1836577}. Although the Cram\'er-Rao lower bound seems to provide a straightforward path to lower bounds on the bias-variance trade-off, the imposed regularity conditions make this approach problematic for nonparametric and high-dimensional models. For example, when the parameter space is the set of $s$-sparse vectors, this is not an open set and it is unclear how to define the gradient of the bias function or the Fisher information.
% A major obstacle is the proper definition of a nonparametric Fisher information. It is moreover unclear how to interpret the Fisher information for parameter spaces that are not open sets such as for instance the space of all sparse vectors. 

Instead of trying to fix the shortcomings of the Cram\'er-Rao lower bound for complex statistical models, we derive a number of inequalities that bound the change of expectation with respect to two different distributions by the variance and one of the four standard divergence measures: total variation, Hellinger distance, Kullback-Leibler divergence and the $\chi^2$-divergence. As we will see later, these inequalities are much better suited for nonparametric problems as no notion of differentiability of the distribution with respect to the parameter is required. Moreover, the Cram\'er-Rao lower bound reappears by taking a suitable limit.

\begin{lemma}\label{lem.general_lb}
Let $P$ and $Q$ be two probability distributions on the same measurable space. Denote by $E_P$ and $\Var_P$ the expectation and variance with respect to $P$ and let $E_Q$ and $\Var_Q$ be the expectation and variance with respect to $Q.$ Then, for any random variable $X,$
\begin{align}
    \frac{( E_P[X]-E_Q[X])^2}2
    \Big( \frac1{\TV(P,Q)}-1\Big) 
    &\leq \Var_P(X)+\Var_Q(X),
    \hspace{-0.3cm}
    \label{eq.lem1_1} \\
    %
    % \frac{( E_P[X]-E_Q[X])^2}{4}
    % \Big( \frac{1}{H(P,Q)}-H(P,Q)\Big)^2
    \frac{( E_P[X]-E_Q[X])^2}{4-2H^2(P,Q)}
    \Big( \frac{1}{H(P,Q)}-H(P,Q)\Big)^2
    &\leq \Var_P(X)+ \Var_Q(X), 
    \hspace{-0.3cm}
    \label{eq.lem1_2} 
    \\
    \hspace{-0.1cm}
    ( E_P[X]-E_Q[X])^2
    \Big( \frac1{\KL(P,Q)+\KL(Q,P)}-\frac 14\Big)
    &\leq \Var_P(X) \vee \Var_Q(X), 
    \hspace{-0.3cm}
    \label{eq.lem1_3} \\
    ( E_P[X]-E_Q[X])^2
    \leq \chi^2(Q,P)\Var_P(X)
    \label{eq.lem1_4} 
    & \wedge  \chi^2(P,Q)\Var_Q(X).
\end{align}
\end{lemma}

The inequality in \eqref{eq.lem1_4} is known \cite[Lemma 2]{nishiyama2019newKL} and can also be viewed as a consequence of the Hammersley-Chapman-Robbins inequality~\cite[Example 5.2]{lehmann2006theory}. \cite[Lemma 5.3]{MR1836577} derives analogous formulas for \eqref{eq.lem1_2} and \eqref{eq.lem1_4} with the variance replaced by the second moment. \eqref{eq.lem1_2} is derived from \cite[Theorem 1]{nishiyama2020tightHellinger}.
To the best of our knowledge, the inequalities in~\eqref{eq.lem1_1} and~\eqref{eq.lem1_3} have not been stated yet in the literature. A proof is provided in Supplement \ref{sec.proofs_gen_lower_bd}. 

If one of the information measures is zero, the left-hand side of the corresponding inequality should be assigned the value zero as well. The inequalities are based on different decompositions for $E_P[X]-E_Q[X]=\int X(\om)(dP(\om)-dQ(\om)).$ All of them involve an application of the Cauchy-Schwarz inequality. For deterministic $X$, both sides of the inequalities are zero and hence we have equality. For \eqref{eq.lem1_4}, the choice $X=dQ/dP$ yields equality and in this case, both sides are $(\chi^2(Q,P))^2.$
Another line of related inequalities bound the change of expectations in terms of $f$-divergences, without involving the variance, see for instance \cite{chen2016bayes,Gerchinovitz2020Fano}.

To obtain lower bounds for the variance, our inequalities can be applied similarly as the Cram\'er-Rao inequality. Indeed, small bias implies that $E_{\theta}[\wh \theta]$ is close to $\theta$ and $E_{\theta'}[\wh \theta]$ is close to $\theta'.$ If $\theta$ and $\theta'$ are sufficiently far from each other, we obtain a lower bound for $|E_{\theta}[\wh \theta]-E_{\theta'}[\wh \theta]|$ and a fortiori a lower bound for the variance. 
This argument suggests that the lower bound becomes stronger by picking parameters $\theta$ and $\theta'$ that are as far as possible away from each other. But then, also the information measures of the distributions $P_\theta$ and $P_{\theta'}$ are typically larger, making the lower bounds worse. This shows that an optimal application of the inequalities should balance these two aspects. 

%For the examples studied in later sections, it turns out that the distributions $P_\theta$ and $P_{\theta'}$ should be chosen such that these information measures stay bounded with increasing sample size, and are bounded away from one for the total variation and the Hellinger distance. 

Example~\ref{ex.bds_for_normal} in the supplementary material illustrates these inequalities in the case of the Gaussian distribution.
For other distributions, one of these four divergence measures might be easier to compute and the four inequalities can lead to substantially different lower bounds. For instance, if the measures $P$ and $Q$ are not dominated by each other, the Kullback-Leibler and  $\chi^2$-divergence are both infinite but the Hellinger distance and total variation version still produce non-trivial lower bounds. This justifies deriving for each divergence measure a separate inequality. It is also in line with the formulation of the theory on minimax lower bounds (see for instance Theorem 2.2 in \cite{tsybakov2009introduction}).

Except for the total variation version, all derived inequalities in Lemma \ref{lem.general_lb} are generalizations of the Cram\'er-Rao lower bound. The Cram\'er-Rao lower bound appears by taking $P$ and $Q$ to be $P_\theta$ and $P_{\theta+\Delta}$ and letting $\Delta$ tend to zero. A proof and a variation of Lemma \ref{lem.general_lb} for a family of distributions $(P_t)_{t \in [0,1]}$ (Lemma \ref{lem.lb_limit}) can be found in Supplement \ref{sec.proofs_gen_lower_bd}.

\subsection{Information matrices and lower bound based on multiple distributions}

For minimax lower bounds based on hypotheses tests, it has been observed that lower bounds based on two hypotheses are only rate-optimal in specific settings such as for some functional estimation problems. If the local alternatives surrounding a parameter $\theta$ spread over many different directions, estimation of $\theta$ becomes much harder. To capture this in the minimax lower bounds, we need instead to reduce the problem to a multiple testing problem involving potentially a large number of tests. 

A similar phenomenon occurs also for bias-variance trade-off lower bounds. Given $M+1$ probability measures $P_0, P_1, \dots, P_M,$ the $\chi^2$-version of Lemma \ref{lem.general_lb} states that for any $j=1,\dots,M,$ $(E_{P_j}[X]-E_{P_0}[X])^2/\chi^2(P_j,P_0) \leq \Var_{P_0}(X).$
If $P_1, \dots, P_M$ describe different directions around $P_0$ in a suitable information theoretic sense, one would hope that in this case a stronger inequality holds with the sum on the left-hand side, that is, $\sum_{j=1}^M (E_{P_j}[X]-E_{P_0}[X])^2/\chi^2(P_j,P_0) \leq \Var_{P_0}(X).$ In a next step, two notions of information matrices are introduced, measuring to which extent $P_1,\dots,P_M$ represent different directions around $P_0.$
If $P_0$ dominates $P_1,\dots,P_M,$ the $\chi^2$-divergence matrix $\chi^2(P_0,\dots,P_M)$ is defined as the $M\times M$ matrix with $(j,k)$-th entry
\begin{align*}
    \chi^2(P_0,\dots,P_M)_{j,k}
    := \int \frac{dP_j}{dP_0} dP_k - 1.
\end{align*}
The $M \times M$ Hellinger affinity matrix is defined entrywise by 
\begin{align*}
    \rho(P_0 | P_1,\dots,P_M)_{j,k}
    := \frac {\int \sqrt{p_j p_k}  \, d\nu}
    {\int \sqrt{p_j p_0  }\, d\nu  \int \sqrt{p_k p_0} \, d\nu} -1, \quad j,k=1,\dots, M.
\end{align*}
Here and throughout the article, we implicitly assume that the distributions $P_0,\dots,P_M$ are chosen such that the Hellinger affinities $\int \sqrt{p_j p_0  }\, d\nu$ are positive and the Hellinger affinity matrix is well-defined. This condition is considerably weaker than assuming that $P_0$ dominates the other measures (which is necessary for finiteness of the $\chi^2$-divergence matrix).
These two notions of information matrices are studied in more detail in \cite{DerumignySchmidtHieber2023}.

For a matrix $A,$ the Moore-Penrose inverse $A^+$ always exists and satisfies the property $AA^+A=A$ and $A^+AA^+=A^+.$ We can now state the generalization of \eqref{eq.lem1_4} to an arbitrary number of distributions.
The following theorem is proved in Appendix~\ref{sec.proofs_gen_lower_bd}.

\begin{thm}\label{thm.multiple_lb} For $M\geq 1$, let $P_0, P_1, \dots, P_M$ be probability measures defined on the same probability space, and $X$ be a random variable.

\begin{enumerate}[(i)]
    \item Set $\Delta:=(E_{P_1}[X]-E_{P_0}[X],\dots,E_{P_M}[X]-E_{P_0}[X])^\top.$ If $P_j \ll P_0$ for all $j=1,\dots, M$, then
    $\Delta^\top \chi^2(P_0,\dots,P_M)^+ \Delta
        \leq  \Var_{P_0}(X),$
    where $\chi^2(P_0,\dots,P_M)^+$ denotes the Moore-Penrose inverse of the $\chi^2$-divergence matrix. 
    
    \medskip
    
    \item Let $A_\ell:=\rho(P_\ell|P_1,\dots,P_{\ell-1},P_{\ell+1},\dots, P_M).$ Then, for $M\geq 2,$
    \begin{align*}
        2M\sum_{j=1}^M &\Big(E_j[X]-\frac 1M \sum_{\ell=1}^M E_\ell[X]\Big)^2 \\
        &= \sum_{j,k=1}^M (E_j[X]-E_k[X])^2
        \leq 4\max_{\ell=1,\ldots,M} \lambda_1(A_\ell)\sum_{k=1}^M  \Var_{P_k}(X),
    \end{align*}
    where $\lambda_1(A_\ell)$ denotes the largest eigenvalue (spectral norm) of the positive semi-definite Hellinger affinity matrix $A_\ell$.
\end{enumerate}
\end{thm}

Instead of using a finite number of probability measures, it is in principle possible to extend Theorem \ref{thm.multiple_lb} to families of probability measures. The divergence matrices become then operators and the sums have to be replaced by integral operators. 

If the $\chi^2$-divergence matrix is diagonal with positive entries on the diagonal, we obtain that $\sum_{j=1}^M (E_{P_j}[X]-E_{P_0}[X])^2/\chi^2(P_j,P_0)\leq \Var_{P_0}(X).$ It should be observed that because of the sum, this inequality produces better lower bounds than~\eqref{eq.lem1_4}.

Theorem~\ref{thm.multiple_lb}(i) contains the multivariate Cramér-Rao lower bound as a special case, see Section~\ref{sec:multCRLB}. The connection to the Cram\'er-Rao inequality suggests that for a given statistical problem with a $p$-dimensional parameter space, one should apply Theorem \ref{thm.multiple_lb} with $M=p.$ It turns out that for the high-dimensional models discussed in Section \ref{sec.highdimensional} below, the number of distributions $M$ will be chosen as $\binom{p-1}{s-1}$ with $p$ the number of parameters and $s$ the sparsity. Depending on the sparsity, this can be much larger than~$p.$ 

We are aware of two existing inequalities that are related to Theorem~\ref{thm.multiple_lb}(i). \cite[Equation (3.1)]{wahl2021van} rewritten in our notation is $\sum_{j=1}^M \Var_{P_j}(X) \geq \big( \sum_{j=1}^M E_{P_j}[X] - E_{P_0}[X] \big)^2 / \sum_{j=1}^M \chi^2(P_j,P_0)$ and \cite[p.330]{polyanskiy2021notes} states that for any $p \geq 1$ and any distributions $P,Q$ on $\Rb^p$, $\chi^2(P,Q) \geq (E_P[X] - E_Q[X])^\top \Cov_Q(X)^{-1} (E_P[X] - E_Q[X])$, where $\Cov_Q(X)$ denotes the covariance matrix of $X$ under $Q$. The concept of Fisher $\Phi$-information also generalizes the Fisher information using information measures, see \cite{MR2081075, MR3506739}. It is worth mentioning that this notion is not comparable with our approach and only applies to Markov processes.

To apply Theorem \ref{thm.multiple_lb}(i), we now introduce several variations. As a consequence of Proposition 3.2(ii) in \cite{DerumignySchmidtHieber2023},
a vector $v=(v_1,\dots,v_M)$ lies in the kernel of the $\chi^2$-divergence matrix if and only if $\sum_{j=1}^M v_j (P_j-P_0)=0.$ This shows that such a $v$ and the vector $\Delta$ must be orthogonal. Thus, $\Delta$ is orthogonal to the kernel of $\chi^2(P_0,\dots,P_M)$ and
\begin{align}
    \sum_{j=1}^M \big(E_{P_j}[X]-E_{P_0}[X]\big)^2
    \leq \lambda_1\big(\chi^2(P_0,\dots,P_M)\big) \Var_{P_0}(X),
    \label{eq.bound_lambda1_chi_var}
\end{align}
where $\lambda_1\big(\chi^2(P_0,\dots,P_M)\big)$ denotes the largest eigenvalue (spectral norm) of the $\chi^2$-divergence matrix.
Given a symmetric matrix $A=(a_{ij})_{i,j=1,\dots,M},$ the maximum row sum norm is defined as $\|A\|_{1,\infty}:=\max_{i=1,\ldots,M} \sum_{j=1}^M |a_{ij}|.$ For any eigenvalue $\lambda$ of $A$ with corresponding eigenvector $v=(v_1,\dots,v_M)^\top$ and any $i \in \{1, \dots, M\},$ we have that $\lambda v_i=\sum_{j=1}^M a_{ij} v_j$ and therefore
$|\lambda| \max_{i=1,\ldots,M} |v_i| \leq \max_{i=1,\ldots,M} \sum_{j=1}^M |a_{ij}| \|v\|_\infty.$ Therefore, $\|A\|_{1,\infty}$ is an upper bound for the spectral norm and
\begin{align}
    \sum_{j=1}^M \big(E_{P_j}[X]-E_{P_0}[X]\big)^2
    \leq \big\|\chi^2(P_0,\dots,P_M)\big\|_{1,\infty} \Var_{P_0}(X).
    \label{eq.row_sum_norm_nd}
\end{align}

Whatever variation of Theorem \ref{thm.multiple_lb} is applied to derive lower bounds on the bias-variance trade-off, the key problem is the computation of the information matrix for given probability measures $P_{\theta_j},$ $j=0,\dots, M$ in the underlying statistical model $(P_\theta: \theta \in \Theta).$ Suppose there exists a more tractable statistical model $(Q_\theta: \theta \in \Theta)$ with the same parameter space such that the data in the original model can be obtained by a transformation of the data generated from $(Q_\theta: \theta \in \Theta).$ Theorem 4.1 in the companion paper \cite{DerumignySchmidtHieber2023} states a data processing inequality for $\chi^2$-divergence matrices. 
In the setting considered above, this data processing inequality can be written as matrix inequality 
\begin{align}
    \chi^2(P_{\theta_0},\dots,P_{\theta_M})\leq \chi^2(Q_{\theta_0},\dots,Q_{\theta_M}),
    \label{eq.data_processing}
\end{align} where $\leq$ is understood with respect to the partial order on the set of positive semi-definite matrices. We therefore can apply the upper bounds \eqref{eq.bound_lambda1_chi_var} and \eqref{eq.row_sum_norm_nd} with $\chi^2(P_{\theta_0},\dots,P_{\theta_M})$ replaced by $\chi^2(Q_{\theta_0},\dots,Q_{\theta_M}).$ In Theorem~\ref{thm.multiple_lb}(i), $\chi^2(P_{\theta_0},\dots,P_{\theta_M})^+$ can be replaced by $\chi^2(Q_{\theta_0},\dots,Q_{\theta_M})^+$ if the matrix $\chi^2(P_{\theta_0},\dots,P_{\theta_M})$ is invertible. A specific application for the combination of general lower bounds and the data processing inequality is given in Section \ref{sec.highdimensional}.

% Before discussing a number of specific statistical models, it is worth mentioning that the proper definition of the bias-variance trade-off depends on some subtleties underlying the choice of the space of values that can be attained by an estimator, subsequently denoted by $\Ac$. To illustrate this, suppose we observe $X \sim \Nc(\theta,1)$ with parameter space $\Theta=\{-1,1\}.$ For any estimator $\wh \theta$ with $\Ac=\Theta,$ $E_{1}[\wh \theta]<1$ or $E_{-1}[\wh \theta]>-1.$ Thus, no unbiased estimator with $\Ac=\Theta$ exists. If the estimator is, however, allowed to take values on the real line, then $\widehat \theta = X$ is an unbiased estimator for $\theta.$ We believe that the correct way to derive lower bounds on the bias-variance trade-off is to allow the action space $\Ac$ to be very large. Whenever $\Theta$ is a class of functions on $[0,1]$, the lower bounds below are over all estimators with $\Ac$ the real-valued functions on $[0,1];$ for high-dimensional problems with
% $\Theta \subseteq \Rb^p,$ the lower bounds are over all estimators with $\Ac=\Rb^p.$ { In particular, if the true parameter vector is assumed to be sparse, we do not require that the estimator is sparse as well.}

For various distributions, closed-form expression for the information matrices are derived in \cite{DerumignySchmidtHieber2023}.
In particular, if $P_j=\Nc(\theta_j, \sigma^2 I_d)$ with $\theta_j \in \Rb^d$ and $\sigma>0,$ then
\begin{align}
    \chi^2(P_0, P_1, \ldots, P_M)_{j,k} =\exp\bigg(\dfrac{\langle \theta_j - \theta_0, \theta_k - \theta_0 \rangle}{\sigma^2}\bigg) -1.
    \label{eq.chi2_normal_distr}
\end{align}

\section{The bias-variance trade-off for pointwise estimation in the Gaussian white noise model}
\label{sec:gauss_wn}

In the Gaussian white noise model, we observe a random function $Y=(Y_x)_{x\in [0,1]},$ with
\begin{equation}
    dY_x = f(x) \, dx + n^{-1/2} \, dW_x,
    \label{eq.mod_GWN}
\end{equation}
where $W$ is an unobserved standard Brownian motion. The aim is to recover the regression function $f:[0,1]\to \Rb$ from the data $Y$. In this section, the bias-variance trade-off for estimation of $f(x_0)$ with fixed $x_0\in [0,1]$ is studied. In Section \ref{sec.reduction}, we will also derive a lower bound for the trade-off between integrated squared bias and integrated variance. 

Denote by $\|\cdot\|_{2}$ the $L^2([0,1])$-norm. For $f\in L^2([0,1]),$ the likelihood ratio in the Gaussian white noise model is given by Girsanov's formula $dP_f/dP_0(Y)=\exp(n \int_0^1 f(t) dY_t - \tfrac n2 \|f\|_2^2).$ In particular, for $Y \sim P_f$ and for any function $g\in L^2([0,1])$, we have that
\begin{align*}
    \frac{dP_f}{dP_g}(Y)
    &=\exp\bigg(n\int \big(f(x)-g(x)\big) \, dY_x- \frac n2 \|f\|_2^2+\frac n2 \|g\|_2^2
    \bigg) \\
    &= \exp\bigg( \sqrt{n} \int \big(f(x)-g(x)\big) \,  dW_x+\frac{n}{2}\big\|f-g\big\|_2^2\bigg) \\
    &= \exp\bigg( \sqrt{n} \big\|f-g\big\|_2 \xi +\frac{n}{2}\big\|f-g\big\|_2^2\bigg),
\end{align*}
with $W$ a standard Brownian motion and $\xi \sim \Nc(0,1)$. From this representation, we can easily deduce that $1-H^2(P_f,P_g)=E_f[(dP_f/dP_g)^{-1/2}]$ 
$=\exp(-\tfrac n8 \|f-g\|_2^2),$ $\KL(P_f,P_g)=E_f[\log(dP_f/dP_g)]= \tfrac{n}{2}\|f-g\|_2^2$ and $\chi^2(P_f,P_g)=E_f[dP_f/dP_g]-1=\exp(n\|f-g\|_2^2)-1.$

Let $R > 0$, $\beta > 0$ and denote by $\floorbeta$ the largest integer that is strictly smaller than $\beta$. On a domain $D \subseteq \Rb,$ we define the $\beta$-H\"older norm by $\| f \|_{\Cc^\beta(D)}= \sum_{\ell \leq \floorbeta}  \|f^{(\ell)} \|_{L^\infty(D)}+ \sup_{x, y \in D, x\neq y} |f^{(\floorbeta)}(x) - f^{(\floorbeta)}(y)  | / | x - y  |^{\beta - \floorbeta},$ with $L^\infty(D)$ the supremum norm on $D$ and $f^{(\ell)}$ denoting the $\ell$-th (strong) derivative of $f$ for $\ell \leq \floorbeta$. For $D=[0,1],$ let $\Cc^\beta(R):=\{f:[0,1] \to \Rb :\| f \|_{\Cc^\beta([0,1])} \leq R\}$ be the ball of $\beta$-H\"older smooth functions $f:[0,1] \to \Rb$ with radius $R.$ We also write $\Cc^\beta(\Rb):=\{K:\Rb \to \Rb :\| K \|_{\Cc^\beta(\Rb)} < \infty\}.$

%For a given estimator $\wh f$ and a point $x_0 \in [0,1]$, we define the bias of $\wh f$ in estimating $f$ at point $x_0$ by $\Bias_{f}(\wh f(x_0)) := E_f[\wh f(x_0)] - f(x_0)$.

To explore the bias-variance trade-off for pointwise estimation in more detail, consider for a moment the kernel smoothing estimator, defined by $\wh f(x_0)= (2h)^{-1}\int_{x_0-h}^{x_0+h} dY_t.$ Assume that $x_0$ is not at the boundary such that $0\leq x_0-h$ and $x_0+h\leq 1.$ Bias and variance for this estimator are
\begin{align*}
    \Bias_f\big(\wh f(x_0)\big) 
    = \frac 1{2h} \int_{x_0-h}^{x_0+h} \big(f(u)-f(x_0)\big) \, du, \quad 
    \Var_f\big(\wh f(x_0)\big) 
    = \frac{1}{2nh}.
\end{align*}
%If $f\in \Cc^\beta(R)$ and $\beta \leq 1,$ it follows immediately that $|\Bias_f(\wh f(x_0))|\leq R h^\beta.$ 
While the variance is independent of $f,$ the bias vanishes for large subclasses of $f$ such as, for instance, any function $f$ satisfying $f(x_0-v)=-f(x_0+v)$ for all $0\leq v\leq h.$ The largest possible bias over this parameter class is of the order $h^\beta$ and it is attained for functions that lie on the boundary of $\Cc^\beta(R).$ Because of this asymmetry between bias and variance, the strongest lower bound on the bias-variance trade-off that we can hope for is that any estimator $\wh f(x_0)$ satisfies an inequality of the form
\begin{align}
    \sup_{f\in \Cc^\beta(R)}|\Bias_f(\wh f(x_0))|^{1/\beta} \inf_{f\in \Cc^\beta(R)}\Var_f(\wh f(x_0)) \gtrsim \frac 1n.
    \label{eq.wwwlth_ptw}
\end{align}

Since for fixed $x_0,$ $f \mapsto f(x_0)$ is a linear functional, pointwise reconstruction is a specific linear functional estimation problem. This means in particular that the theory in \cite{low1995bias} for arbitrary linear functionals in the Gaussian white noise model applies. We now summarize the implications of this work on the bias-variance trade-off and state the new lower bounds based on the change of expectation inequalities derived in the previous section afterwards. 

\cite{low1995bias} shows that the bias-variance trade-off for estimation of functionals in the Gaussian white noise model can be reduced to the bias-variance trade-off for estimation of a bounded mean in a normal location family. If $f\mapsto Lf$ denotes a linear functional, $\wh{Lf}$ stands for an estimator of $Lf$, $\Theta$ is the parameter space and $w(\eps):=\sup \big\{|L(f-g)|:\|f-g\|_{L^2[0,1]}\leq \eps, f,g \in \Theta\big\}$ is the so-called modulus of continuity, Theorem 2 in~\cite{low1995bias} rewritten in our notation states that, if $\Theta$ is closed and convex and $\lim_{\eps \downarrow 0} w(\eps) = 0,$ then 
\begin{align*}
    &\inf_{\wh{Lf}: \, \sup_{f\in \Theta} \Var_f(\wh{Lf})\leq V}
    \, \sup_{f\in \Theta} \, \Bias_f(\wh{Lf})^2
    = \frac 14 \sup_{\eps >0} \big(w(\eps)-\sqrt{nV}\eps \big)_+^2,
    \text{ and,} \\
    &\inf_{\wh{Lf}: \, \sup_{f\in \Theta} |\Bias_f(\wh{Lf})|\leq B}
    \, \sup_{f\in \Theta} \, \Var_f(\wh{Lf})
    = \frac 1n \sup_{\eps >0} \eps^{-2} \big(w(\eps)-2B \big)_+^2,
\end{align*}
with $(x)_+:=\max(x,0).$ Moreover, an affine estimator $\wh{Lf}$ can be found attaining these bounds. For pointwise estimation on H\"older balls, $Lf=f(x_0)$ and $\Theta = \Cc^\beta(R).$ To find a lower bound for the modulus of continuity in this case, choose $K\in \Cc^\beta(\Rb),$ $f=0$ and $g= h^\beta K((x-x_0)/h)$. By Lemma~\ref{lem.kernel_norm_bd}, $g\in \Cc^\beta(R)$ whenever $R\geq \|K\|_{\Cc^\beta(\Rb)}$ and by substitution, $\|f-g\|_2=\|g\|_2\leq h^{\beta+1/2}\|K\|_2\leq \eps$ for $h=(\eps/\|K\|_2)^{1/(\beta+1/2)}.$ This proves $w(\eps)\geq (\eps/\|K\|_2)^{\beta/(\beta+1/2)}K(0).$ In Appendix~\ref{subsec:computation_constants_Low}, we show that this further implies
\begin{align}
    & \hspace{-2cm} \inf_{\wh f(x_0)}
    \, \sup_{f\in \Cc^\beta(R)} \, \Big|\Bias_f \big(\wh f(x_0) \big)\Big|^{1/\beta} \sup_{f\in \Cc^\beta(R)}
    \Var_f(\wh f(x_0))
    \geq \frac{\gamma_{\text{Low}}(R,\beta)}{n},
    \label{eq.low_derived_LB} \\
    \text{where } 
    \gamma_{\text{Low}}(R,\beta)
    &:= \sup_{K \in \Cc^\beta(\Rb): \, R \geq \|K\|_{\Cc^\beta(\Rb)}}
    \frac{(2\beta)^2}{2^{1/\beta} (2\beta + 1)^{2+1/\beta}}
    \frac{K(0)^{2+1/\beta}}{\|K\|_2^2}.
    \nonumber
\end{align}
The result is comparable to \eqref{eq.wwwlth_ptw} with a supremum instead of an infimum in front of the variance. 

\medskip

%Such worst-case bounds do not rule out the possibility that there exists an estimator $\wh f(x_0)$ such that for any function $f$ either the bias or the variance is zero. 
We now derive the lower bounds on the bias-variance trade-off for the pointwise estimation problem, that are based on the general framework developed in the previous section. Define
$$\gamma(R,\beta)
    := \sup_{K \in \Cc^\beta(\Rb): K(0) = 1}
    \Bigg(\|K\|_2^{-1} \bigg( 1 - \dfrac{\|K\|_{\Cc^\beta(\Rb)}}{R} \bigg)_+ \Bigg)^2.$$
For fixed $\beta>0,$ this quantity is positive if and only if $R > 1$. Indeed, if $R \leq 1$, for any function $K$ satisfying $K(0) = 1$, we have $R \leq 1 \leq \| K \|_\infty \leq \| K \|_{\Cc^\beta(\Rb)}$ and therefore, $\|K\|_{\Cc^\beta(\Rb)}/R \geq 1$, implying $\gamma(R, \beta) = 0$. On the contrary, when $R > 1$, we can take for example $K(x) = \exp(-x^2/A)$ with $A$ large enough  such that $1 \leq \| K \|_{\Cc^\beta(\Rb)} < R$. This shows that $\gamma(R, \beta) > 0$ in this case.

If $C$ is a positive constant and $a\in [0,R),$ define moreover
\begin{align*}
    \ol \gamma(R,\beta,C,a)
    &:= \sup_{K \in \Cc^\beta(\Rb): K(0) = 1}
    \Bigg(\|K\|_2^{-1} \bigg( 1 - \dfrac{\|K\|_{\Cc^\beta(\Rb)}}{R-a} \bigg)_+ \Bigg)^2 \\
    &\hspace{5cm} \times \exp\bigg(-C(R-a)^2 \frac{\|K\|_2^2}{\|K\|_{\Cc^\beta(\Rb)}^2} \bigg).
\end{align*}
Arguing as above, for fixed $\beta>0,$ this quantity is positive if and only if $a+1< R.$ We can now state the main result of this section. 

%The following theorem is proved in Section~\ref{proof:thm:GaussWN:bias_var_tradeoff_pointwise}.

\begin{thm}\label{thm.pointwise}
    Given $\beta, R, C > 0$ and $x_0 \in [0,1],$ let $\gamma(R,\beta)$ and $\ol \gamma(R,\beta,C,a)$ be the constants defined above. Assign to $(+\infty)\cdot 0$ the value $+\infty.$ 
    \newline
    \noindent
    {\it (i):} If $\Tc=\{\wh f:\sup_{f\in \Cc^\beta(R)} \big|\Bias_f\big(\wh f(x_0)\big)\big|<1\},$ then,
    \begin{align}
        \inf_{\wh f \in \Tc} \, \sup_{f\in \Cc^\beta(R)} \big|\Bias_f\big(\wh f(x_0)\big)\big|^{1/\beta}
        \sup_{f\in \Cc^\beta(R)}\Var_f\big(\wh f(x_0)\big) 
        \geq \frac{\gamma(R,\beta)}{n}.
        \label{eq.thm_ptw_assertion1}
    \end{align}
    \newline
    \noindent
    {\it (ii):} Let $\Sc(C):= \{\wh f: \sup_{f\in \Cc^\beta(R)}|\Bias_f(\wh f(x_0))|<  (C/n)^{\beta/(2\beta+1)}\}\cap \Tc,$ then,
    \begin{align}
        \inf_{\wh f \in \Sc(C)} \, \sup_{f\in \Cc^\beta(R)} \big|\Bias_f\big(\wh f(x_0)\big)\big|^{1/\beta}
        \inf_{f\in \Cc^\beta(R)} \frac{\Var_f(\wh f(x_0))}{\ol \gamma(R,\beta,C,\|f\|_{\Cc^\beta})} 
        \geq \frac{1}{n}.
    \end{align}
\end{thm}

Both statements can be easily derived from the abstract lower bounds in Section \ref{sec.general_low_bds}. A full proof is given in Supplement \ref{sec.proofs_gauss_wn} where statement (i) is derived from Lemma~\ref{lem.lb_limit} and statement (ii) is derived from Lemma~\ref{lem.general_lb}. The first statement quantifies a worst-case bias-variance trade-off that must hold for any estimator. The case that $\sup_{f\in \Cc^\beta(R)}|\Bias_f(\wh f(x_0))|$ exceeds one is not covered. As it leads to inconsistent mean squared error it is of little interest and therefore omitted. The second statement restricts attention to estimators with minimax rate-optimal bias. Because of the infimum, we obtain a lower bound on the variance for any function $f.$ Note that this statement is much stronger than \eqref{eq.low_derived_LB} or \eqref{eq.thm_ptw_assertion1} as it holds for the best-case variance instead of the worst-case variance. Compared with \eqref{eq.wwwlth_ptw}, the lower bound depends on the $\Cc^\beta$-norm of $f$ through $\ol \gamma(R,\beta,C,\|f\|_{\Cc^\beta}).$ This quantity becomes large if $f$ is close to the boundary of the H\"older ball. A consequence of $(ii)$ is the uniform bound
\begin{align}
    \hspace{-0.1cm}
    \inf_{\wh f \in \Sc(C)} \sup_{f\in \Cc^\beta(R)}
    \big| \hspace{-0.05cm} \Bias_f\big(\wh f(x_0)\big) \big|^{1/\beta}
    \hspace{-0.2cm}
    \inf_{f\in \Cc^\beta(a)} \Var_f(\wh f(x_0)) 
    \geq \frac{\displaystyle \inf_{b\leq a}\ol \gamma(R,\beta,C,b)}{n},
    \hspace{-0.6cm} 
    \label{eq.8888}
\end{align}
providing a non-trivial lower bound if $a<R-1$, see Supplement \ref{sec.proof8888} for a proof.
The established lower bound requires that the radius of the H\"older ball $R$ is sufficiently large. Such a condition is necessary. To see this, suppose $R\leq 1$ and consider the estimator $\wh f(x_0)=0.$ Notice that for any $f\in \Cc^\beta(R),$ $|\Bias_f(\wh f(x_0))|= |f(x_0)|\leq \|f\|_\infty\leq 1$ and $\Var_f(\wh f(x_0))=0.$ The left-hand side of the inequality \eqref{eq.thm_ptw_assertion1} is hence zero and even such a worst-case bias-variance trade-off does not hold.

Thanks to the bias-variance decomposition of the mean squared error, for every estimator $\wh f(x_0) \in \Tc$,
% there exists an $f\in \Cc^\beta(R)$ such that we have  $|\Bias_{f}(\wh f(x_0))|^{1/\beta} \Var_f(\wh f(x_0))\geq \gamma(R,\beta)/n$ and thus for such an $f,$ 
% \begin{align*}
%     \MSE_f\big(\wh f(x_0)\big) 
%     &=\Bias_{f}\big(\wh f(x_0)\big)^2 + \Var_f\big(\wh f(x_0)\big) \\
%     &\geq 
%     \bigg(\frac{\gamma(R,\beta)}{n \Var_f\big(\wh f(x_0)\big)}\bigg)^{2\beta}
%     \wedge \frac{\gamma(R,\beta)}{n|\Bias_{f}\big(\wh f(x_0)\big)|^{1/\beta}}.
% \end{align*}
\begin{align*}
    \sup_{f\in \Cc^\beta(R)} \MSE_f\big(\wh f(x_0)\big) 
    % &= \sup_{f\in \Cc^\beta(R)} \Bias_{f}\big(\wh f(x_0)\big)^2 + \sup_{f\in \Cc^\beta(R)} \Var_f\big(\wh f(x_0)\big) \\
    &\geq 
    \bigg(\frac{\gamma(R,\beta)}{n \sup_{f\in \Cc^\beta(R)} \Var_f\big(\wh f(x_0)\big)}\bigg)^{2\beta} \\
    &\hspace{3cm} \wedge \frac{\gamma(R,\beta)}{n \sup_{f\in \Cc^\beta(R)} |\Bias_{f}\big(\wh f(x_0)\big)|^{1/\beta}},
\end{align*}
showing that, in a worst case sense, small bias or small variance increases the mean squared error.

\begin{cor}[Classical unconstrained minimax rates]
    Under the same conditions as Theorem~\ref{thm.pointwise}, we have
    \begin{align*}
        \inf_{\wh f} \, \sup_{f\in \Cc^\beta(R)}
        \MSE_f \big( \wh f(x_0) \big)
        \geq \bigg(\frac{\gamma(R,\beta)}{n}\bigg)^{2\beta/(2\beta+1)} \wedge 1,
    \end{align*}
    where the infimum is over all measurable estimators. Moreover, the minimax estimation rate $n^{-2\beta/(2\beta+1)}$ can only be achieved for estimators balancing the rate of the worst-case squared bias and the rate of the worst-case variance.
\end{cor}

For nonparametric problems, an estimator can be superefficient for many parameters simultaneously, see \cite{MR1604424}. Based on that, one might wonder whether it is possible to take for instance a kernel smoothing estimator and shrink small values to zero such that the variance for the regression function $f=0$ is of a smaller order but the order of the variance and bias for all other parameters remains the same. Statement (ii) of Theorem \ref{thm.pointwise} shows that such constructions are impossible if the H\"older radius $R$ is large enough. This question can be viewed as a bias-variance formulation of the constrained risk problem. In the constrained risk problem, we wonder whether an estimator achieving a faster rate for a fixed parameter will have necessarily a suboptimal rate for some other parameter in the parameter space. For pointwise estimation in nonparametric regression, this was studied in Section B of \cite{MR1425965}.

The proof of Theorem \ref{thm.pointwise} depends on the Gaussian white noise model only through the Kullback-Leibler divergence and $\chi^2$-divergence. This indicates that an analogous result can be proved for other nonparametric models with a similar likelihood geometry. As an example consider the Gaussian nonparametric regression model with fixed and uniform design on $[0,1],$ that is, we observe $(Y_1,\dots,Y_n)$ with $Y_i=f(i/n)+\eps_i,$ $i=1,\dots,n$ and $\eps_i \simiid \Nc(0,1).$ Again, $f$ is the (unknown) regression function and we write $P_f$ for the distribution of the observations with regression function $f.$ By evaluating the Gaussian likelihood, we obtain the well-known explicit expressions $\KL(P_f,P_g)=\tfrac n2 \|f-g\|_n^2$ and $\chi^2(P_f,P_g)=\exp(n\|f-g\|_n^2)-1$ where $\|h\|_n^2:=\tfrac 1n \sum_{i=1}^n h(i/n)^2$ is the empirical $L^2([0,1])$-norm. Compared to the Kullback-Leibler divergence and $\chi^2$-divergence in the Gaussian white noise model, the only difference is that the $L^2([0,1])$-norm is replaced here by the empirical $L^2([0,1])$-norm. These norms are very close for functions that are not too spiky. Thus, by following exactly the same steps as in the proof of Theorem \ref{thm.pointwise}, a similar lower bound can be obtained for the pointwise loss in the nonparametric regression model.

\section{The bias-variance trade-off for support boundary recovery}
\label{sec.boundary}

Compared to approaches using the Cram\'er-Rao lower bound, the abstract lower bounds based on information measures have the advantage to be applicable also for irregular models. This is illustrated in this section by deriving lower bounds on the bias-variance trade-off for a support boundary estimation problem.

\begin{wrapfigure}[11]{R}{0.42\textwidth}
    \vspace{-0.8cm}
  \centering
    \includegraphics[trim=0 40 0 50,clip, width=0.44\textwidth]{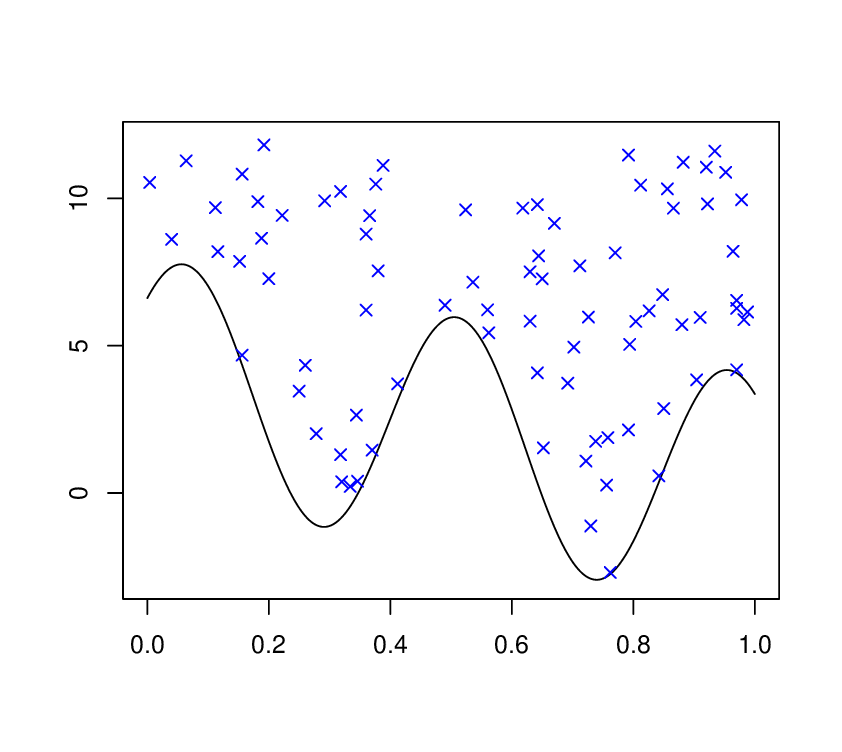}
    \vspace{-0.8cm}
  \caption{\label{fig.1} Generated data (blue) and support boundary (black) for PPP model.}
    \vspace{-0.3cm}
\end{wrapfigure}
Consider the model, where we observe a Poisson point process (PPP) $N=\sum_i \delta_{(X_i,Y_i)}$ with intensity $\lambda_f(x,y)=n\mathbf{1}(f(x) \leq y)$ in the plane $(x,y) \in [0,1]\times \Rb.$ Differently speaking, the Poisson point process has intensity $n$ on the epigraph of the function $f$ and zero intensity on the subgraph of $f.$ The unknown function $f$ appears therefore as a boundary if the data are plotted, see Figure \ref{fig.1}. Throughout the following, $n$ plays the role of the sample size and we refer to $(X_i,Y_i)$ as the support points of the PPP. Estimation of $f$ is also known as support boundary recovery problem.  Similarly as the Gaussian white noise model is a continuous analogue of the nonparametric regression model with Gaussian errors, the support boundary problem arises as a continuous analogue of the nonparametric regression model with one-sided errors, see \cite{MR3010397}.

For a parametric estimation problem, we can typically achieve the estimation rate $n^{-1}$ in this model. For squared loss, this becomes $n^{-2}.$ The $n^{-1}$ rate is to be contrasted with the classical $n^{-1/2}$ rate in regular parametric models. Also for nonparametric problems, faster rates can be achieved. If $\beta$ denotes the H\"older smoothness of the support boundary $f,$ the optimal MSE for estimation of $f(x_0)$ is $n^{-2\beta/(\beta+1)}$ which can be considerably faster than the typical nonparametric rate $n^{-2\beta/(2\beta+1)},$ \cite{MR4158798}.
The following theorem is proved in Supplement~\ref{sec.proofs_boundary} applying the $\chi^2$-divergence version of Lemma~\ref{lem.general_lb}. 

\begin{thm}\label{thm.B_V_tradeoff_boundary}
Let $0<\beta<1,$ $C>0$ and $R > \kappa
:= 2 \inf \{\|K\|_{\Cc^\beta(\Rb)}:
K\in L^2(\Rb), K(0)=1, K\geq 0 \}.$

For any estimator $\wh f$ with 
$\sup_{f\in \Cc^\beta(R)} \, \MSE_f\big(\wh f(x_0)\big)
< (C/n)^{2\beta/(\beta+1)},$
there exist positive constants $c:=c(\beta, C, R)$ and $c':=c'(\beta, C, R)$ such that 
\begin{align}
    &\sup_{f\in  \Cc^\beta(R)} \, \Bias_f\big(\wh f(x_0)\big)^2 
    \geq c n^{-\frac{2\beta}{\beta+1}},
    \textrm{ and,}
    \label{eq.thm_boundary_claim1} \\
    &\Var_f\big(\wh f(x_0)\big)
    \geq c' n^{-\frac{2\beta}{\beta+1}}, \quad \text{for all} \ f\in \Cc^\beta\big((R-\kappa)/2\big).
    \label{eq.thm_boundary_claim2}    
\end{align}
\end{thm}

The result shows that any estimator achieving the optimal $n^{-2\beta/(\beta+1)}$ MSE rate must also have worst-case squared bias of the same order. Moreover no superefficiency is possible for functions that are not too close to the boundary of the H\"older ball. Indeed the variance, and therefore also the mean squared error, is always lower-bounded by $\gtrsim n^{-2\beta/(\beta+1)}.$ The smoothness constraint $\beta \leq 1$ is fairly common in the literature on support boundary estimation, see \cite{MR3606758}.

\section{The trade-off between integrated bias and integrated variance in the Gaussian white noise model}
\label{sec.reduction}

All lower bounds so far are based on change of expectation inequalities. In this section we combine this with a different proving strategy for bias-variance lower bounds based on two types of reduction. Firstly, one can in some cases relate the bias-variance trade-off in the original model to the bias-variance trade-off in a simpler model. We refer to this as model reduction. The second type of reduction constraints the class of estimators by showing that it is sufficient to consider estimators satisfying additional symmetry properties. 

To which extent such reductions are possible is highly dependent on the structure of the underlying problem. In this section we illustrate the approach deriving a lower bound on the trade-off between the integrated squared bias ($\IB$) and the integrated variance ($\IVar$) in the Gaussian white noise model \eqref{eq.mod_GWN}. Recall that the mean integrated squared error (MISE) can be decomposed as 
\begin{align}
    \MISE_f\big(\wh f\big)
    := E_f\big[\big\|\wh f-f\big\|_{L^2[0,1]}^2 \big]
    &= \int_0^1 \Bias_f^2\big(\wh f(x)\big) \, dx 
    + \int_0^1 \Var_f\big(\wh f(x)\big) \, dx \nonumber \\
    &=: \IB_{f}(\wh f)+ \IVar_f\big(\wh f\big).  
    \label{eq.MISE_decomp}
\end{align}
To establish a trade-off between integrated bias and integrated variance, turns out to be a hard problem. In particular, we cannot simply integrate the pointwise lower bounds. Below we explain the major reduction steps to prove a lower bound. To avoid unnecessary technicalities involving the Fourier transform, we only consider integer smoothness $\beta=1,2,\dots$ and denote by $S^\beta(R)$ the ball of radius $R$ in the $L^2$-Sobolev space with index $\beta$ on $[0,1]$, that is, all $L^2$-functions satisfying $\|f\|_{S^\beta([0,1])}\leq R,$ where for a general domain $D,$  $\|f\|_{S^\beta(D)}^2:=\|f\|_{L^2(D)}^2+\|f^{(\beta)}\|_{L^2(D)}^2.$ Define
\begin{align}
    \Gamma_\beta:=\inf \Big\{\|K\|_{S^\beta}:\|K\|_{L^2(\Rb)}=1, \supp K\subset [-1/2,1/2]\Big\}.
    \label{eq.def_Gamma_beta_L2}
\end{align}

\begin{thm}\label{thm.LB_L2}
Consider the Gaussian white noise model \eqref{eq.mod_GWN} with parameter space $S^\beta(R)$ and $\beta$ a positive integer. If $R > 2\Gamma_\beta$ and $0\cdot(+\infty)$ is assigned the value $+\infty,$ then,
    \begin{align}
        \inf_{\wh f \in T} \,
        \sup_{f \in S^\beta(R)} \big|\operatorname{IBias}_{f}(\wh f)\big|^{1/\beta}
        \sup_{f\in S^\beta(R)}
        \IVar_f\big(\wh f\big) 
        \geq \frac{1}{8n},
        \label{eq.thmL2_assertion}
    \end{align}
    with $T:= \{\wh f: \sup_{f \in S^\beta(R)} \IB_{f}(\wh f)<2^{-\beta}\}.$
    \label{thm:GaussWN:bias_var_tradeoff_int}
\end{thm}

As in the pointwise case, estimators with larger bias are of little interest as they will lead to procedures that are inconsistent with respect to the MISE. Thanks to the bias-variance decomposition of the MISE \eqref{eq.MISE_decomp}, for every estimator $\wh f \in T$ the following lower bound on the MISE holds
\begin{align*}
    \sup_{f\in S^\beta(R)}
    \MISE_f\big(\wh f\big)
    &\geq 
    \bigg(\dfrac{1}{8n
    \sup_{f\in S^\beta(R)} \IVar_{f}(\wh f)}\bigg)^{2\beta} \\
    &  \hspace{4cm} \vee \dfrac{1}{8n
    \sup_{f\in S^\beta(R)} |\operatorname{IBias}_{f}(\wh f)|^{1/\beta}}.
\end{align*}
Small worst-case bias or variance will therefore automatically enforce a large MISE. This provides a lower bound for the widely observed $U$-shaped bias-variance trade-off and shows in particular that $n^{-2\beta/(2\beta+1)}$ is a lower bound for the minimax estimation rate with respect to the MISE.
% Moreover, this rate is attained if the worst-case integrated squared bias and the worst-case integrated variance are balanced to be of the same order.

\begin{cor}[Classical unconstrained minimax rates]
    Under the same conditions as Theorem~\ref{thm:GaussWN:bias_var_tradeoff_int}, we have
    \begin{align*}
        \inf_{\wh f} \, \sup_{f\in S^\beta(R)}
        \MISE_f \big( \wh f \big)
        \geq \bigg(\frac{1}{8n}\bigg)^{2\beta/(2\beta+1)} \wedge 1,
    \end{align*}
    where the infimum is over all measurable estimators. Moreover, the minimax rate $n^{-2\beta/(2\beta+1)}$ can only be achieved for estimators balancing the rates of the worst-case integrated squared bias and the worst-case integrated variance.
\end{cor}

If applied to functions, recall that $\|\cdot\|_p$ denotes the $L^p([0,1])$-norm. Let $p\geq 2.$ Since $\|\cdot\|_2 \leq \|\cdot\|_p,$ another direct consequence of the previous theorem is
    \begin{align*}
        \sup_{f \in S^\beta(R)}
        \big\| E_f[\wh f] - f \big\|_p^{1/\beta}
        \sup_{f\in S^\beta(R)}
        E_f\left[ \big\| \wh f - E_f[\wh f] \big\|_p \right]^{2}
        \geq \frac{1}{8n},
    \end{align*}
    
for any estimator with $\sup_{f \in S^\beta(R)} \| E_f[\wh f] - f \|_p < 2^{-\beta}$.

We now sketch the main reduction steps in the proof of Theorem \ref{thm.LB_L2}. The first step is a model reduction to a Gaussian sequence model
\begin{align}
    X_i = \theta_i +\frac{1}{\sqrt{n}} \eps_i, \quad i=1,\dots, m
    \label{eq.mod_seq_mod}
\end{align}
with independent noise $\eps_i \sim \Nc(0,1)$. For any estimator $\wh \theta$ of the parameter vector $\theta=(\theta_1, \dots,\theta_m)^\top,$ we have the bias-variance type decomposition
\begin{align*}
    E_\theta\big[\big\|\wh \theta -\theta \big\|_2^2 \big]
    = \big\|E_\theta \big[\wh \theta\big] -\theta \big\|_2^2
    +\sum_{i=1}^m \Var_\theta\big(\wh \theta_i\big)
\end{align*}
recalling that $\|\cdot\|_2$ denotes the Euclidean norm if applied to vectors.

\begin{prop}\label{prop.L2_reduction_lb_1}
Let $m$ be a positive integer and let $\Gamma_\beta$ be defined as in \eqref{eq.def_Gamma_beta_L2}. Then, for any estimator $\wh f$ of the regression function $f$ in the Gaussian white noise model \eqref{eq.mod_GWN} with parameter space $S^\beta(R)$, there exists a non-randomized estimator $\wh \theta$ in the Gaussian sequence model with parameter space $\Theta_m^\beta(R):=\{\theta:\|\theta\|_2\leq  R /(\Gamma_\beta m^\beta)\},$ such that
\begin{align*}
    &\sup_{\theta \in \Theta_m^\beta(R)} \big\|E_\theta \big[\wh \theta\big] -\theta \big\|_2^2\leq \sup_{f\in S^\beta(R)}\IB_f(\wh f),
    \text{ and, } \\
    &\sup_{\theta \in \Theta_m^\beta(R)}\sum_{i=1}^m \Var_\theta\big(\wh \theta_i\big) \leq \sup_{f\in S^\beta(R)}\IVar_f\big(\wh f\big).
\end{align*}
\end{prop}

A proof is given in Supplement~\ref{sec.proofs_reduction}. The rough idea is to restrict the parameter space $S^\beta(R)$ to a suitable ball in an $m$-dimensional subspace. Denoting the $m$ parameters in this subspace by $\theta_1,\dots,\theta_m,$ every estimator $\wh f$ for the regression function induces an estimator for $\theta_1,\dots, \theta_m$ by projection on this subspace. It has then to be checked that the projected estimator can be identified with an estimator $\wh \theta$ in the sequence model and that the projection does not increase squared bias and variance.

Proposition \ref{prop.L2_reduction_lb_1} reduces the original problem to deriving lower bounds on the bias-variance trade-off in the sequence model \eqref{eq.mod_seq_mod} with parameter space $\Theta_m^\beta(R).$ Observe that $X=(X_1,\dots,X_m)$ is an unbiased estimator for $\theta.$ The existence of unbiased estimators suggests that the reduction to the Gaussian sequence model is unsuitable for deriving lower bounds as it destroys the original bias-variance trade-off. This is, however, not true as the bias will be induced through the choice of $m$. Indeed, to prove Theorem \ref{thm.LB_L2}, $m$ is chosen such that $m^{-\beta}$ is proportional to the worst-case bias and it is shown that the worst-case variance in the sequence model is lower-bounded by $m/n.$ Rewriting $m$ in terms of the bias yields finally a lower bound of form \eqref{eq.thmL2_assertion}.

To obtain bias-variance lower bounds in the sequence model \eqref{eq.mod_seq_mod} is, however, still a very difficult problem as superefficient estimators exist with simultaneously small bias and variance for some parameters. An example is the James-Stein estimator $\wh \theta_{\JS} := (1-(m-2)/(n\|X\|_2^2))X$ with $X=(X_1,\dots,X_m)^\top$ for $m>2$. While its risk $E_\theta[\|\wh \theta -\theta\|_2^2]= \|E_\theta [\wh \theta] -\theta \|_2^2 + \sum_{i=1}^m \Var_\theta(\wh \theta_i)$ is upper bounded by $m/n$ for all $\theta \in \Rb^m,$ the risk for the zero vector $\theta=(0,\dots,0)^\top$ is bounded by the potentially much smaller value $2/n$ (see Proposition 2.8 in \cite{johnstone}). Thus, for the zero parameter vector both $\|E_\theta[\wh \theta] -\theta \|_2^2$ and $\sum_{i=1}^m \Var_\theta(\wh \theta_i)$ are simultaneously small. Furthermore, for any parameter vector $\theta^*$ there exists an estimator $\wh \theta$ with small bias and variance at $\theta^*.$ For instance, the shifted James-Stein estimator $\wh \theta_{\JS,\theta^*} := (1-(m-2)/(n\|X-\theta^*\|_2^2))(X-\theta^*)+\theta^*$ has this property. This suggests that fixing a number of parameters in the neighborhood of some $\theta^*$ and applying an abstract lower bound that applies to all estimators $\wh \theta$ will always lead to a suboptimal rate in this lower bound. 

Instead, we will first show that it is sufficient to study a smaller class of estimators with additional symmetry properties. Denote by $\Om$ the class of $m\times m$ orthogonal matrices. For any $D\in\Om$, $D\theta \in \Theta_m^\beta(R)$ and $DX \sim \Nc(D\theta , I_m/n).$ Therefore the model is rotation-invariant \cite[Chapter 3]{lehmann2006theory}. Following Stein~\cite{stein1956inadmissibility}, we say that a function $f:\Rb^m \to \Rb^m$ is spherically symmetric if for any $x \in \Rb^m$ and any $D\in \Om,$ $f(x)=D^{-1}f(Dx).$ An estimator $\wh\theta =\wh\theta(X)$ is called spherically symmetric if $X\mapsto \wh \theta(X)$ is spherically symmetric. In particular, the James-Stein estimator $\wh \theta_{\JS}$ is spherically symmetric but, unless $\theta^*= 0,$ the shifted James-Stein estimator $\wh \theta_{\JS,\theta^*}$ is not. The discussion above suggests that if we can reduce the class of estimators to spherically symmetric estimators, all parameters with both small bias and variance must be close to the origin. We can then apply one of the abstract lower bounds to probability measures $P_{\theta_0},\ldots,P_{\theta_M}$ with $\theta_0,\dots,\theta_M$ suitably chosen parameter vectors in the neighborhood of some $\theta^*$ that is far enough away from the origin.

This proof strategy works. In a first step we show the reduction to spherically symmetric estimators.

% By extending this argument we show that this is also true for the worst-case bias-variance trade-off. 

\begin{prop}\label{prop.spheri_symm}
Consider the sequence model \eqref{eq.mod_seq_mod} with parameter space $\Theta_m^\beta(R).$ For any estimator $\wh \theta$ there exists a spherically symmetric estimator $\wt \theta$ such that 
\begin{align*}
    &\sup_{\theta \in \Theta_m^\beta(R)}\big\|E_\theta \big[\wt \theta\big] -\theta \big\|_2^2\leq \sup_{\theta \in \Theta_m^\beta(R)}\big\|E_\theta \big[\wh \theta\big] -\theta \big\|_2^2, \text{ and, }  \\
    &\sup_{\theta \in \Theta_m^\beta(R)}\sum_{i=1}^m \Var_\theta\big(\wt \theta_i\big) \leq \sup_{\theta \in \Theta_m^\beta(R)}\sum_{i=1}^m \Var_\theta\big(\wh \theta_i\big).
\end{align*}
\end{prop}

The main idea of the proof is to define $\wt \theta$ as a spherically symmetrized version of $\wh \theta.$ 

To establish lower bounds, it is therefore sufficient to consider spherically symmetric estimators. It has been mentioned in \cite{stein1956inadmissibility} that any spherically symmetric function $h$ is of the form $h(x) = r(\|x\|_2) x,$
for some real-valued function $r.$ In Lemma \ref{lem.minimx_inv2} in the supplement, we provide a more detailed proof of this fact. Using this property, we can then also show that if $\wt \theta(X)$ is a spherically symmetric estimator, the expectation map $\theta \mapsto E_\theta[\wt \theta(X)]$ is a spherically symmetric function. To see this, rewrite $\wt \theta(X)=s(\|X\|_2) X$ and define $\phi(u):=(2\pi/n)^{-m/2}\exp(-nu^2/2).$ Substituting $y=D^{-1}x$ and noticing that the determinant of the Jacobian matrix of this transformation is one since $D$ is orthogonal, we obtain
 \begin{align}
 \begin{split}
        E_{D\theta}\big[\wt \theta(X)\big]
        &= \int s(\|x\|_2) x
        \phi(\|x - D \theta\|_2) \, dx \\
        &= \int s(\|D^{-1}x\|_2) x \phi(\|D^{-1} x - \theta \|_2) \, dx \\
        &= \int s(\|y\|_2) D y 
        \phi(\|y - \theta \|_2) \, dy
        = D E_{\theta}\big[\wt \theta(X)\big].
\end{split}\label{eq.sph_symm_bias}        
\end{align}
Together with Lemma~\ref{lem.minimx_inv2}, this implies that there exists a function $t$ such that for any $\theta,$ $E_\theta[\wt \theta(X)]=t(\|\theta\|_2)\theta$ and hence 
\begin{align}
    \big\|E_\theta \big[\wt \theta(X)\big] -\theta \big\|_2^2 =\|t(\|\theta\|_2)\theta-\theta\|_2^2 
    = \|\theta\|_2^2 \big(t\big(\|\theta\|_2\big)-1\big)^2.
    \label{eq.IBias_formula}
\end{align}
Based on these reductions, we can now prove Theorem \ref{thm.LB_L2} by applying the change of expectation inequality in Theorem \ref{thm.multiple_lb} (i). The details can be found in Appendix \ref{sec.proofs_reduction}. 

\section{The bias-variance trade-off for high-dimensional models with sparsity constraints}{The bias-variance trade-off for high-dimensional models with sparsity constraints}
\label{sec.highdimensional}

In the Gaussian sequence model, we observe $n$ independent random variables $X_i \sim \Nc(\theta_i,1).$ The space of $s$-sparse signals $\Theta(s)$ is the collection of all vectors $(\theta_1,\dots,\theta_n)$ with at most $s$ non-zero components. For any estimator $\wh \theta,$ the bias-variance decomposition of the mean squared error of $\hat \theta$ is
\begin{align}
    E_\theta\big[\big\|\wh \theta -\theta\big\|_2^2\big]
    = \big\|E_\theta\big[\wh \theta\big] -\theta\big\|_2^2
    + \sum_{i=1}^n \Var_\theta\big(\wh \theta_i\big),
    \label{eq.BV_vectors}
\end{align}
where the first term on the right-hand side plays the role of the squared bias. For this model it is known that the exact minimax risk is $2s\log(n/s)$ up to smaller order terms and that the risk is attained by a soft-thresholding estimator \cite{MR1157714}. This estimator exploits the sparsity by shrinking small values to zero. Shrinkage obviously causes some bias but at the same time reduces the variance for sparse signals. We now show that there is indeed a non-trivial bias-variance trade-off both for estimation of the full vector $\theta$ and for estimation of the quadratic functional $\theta\mapsto \|\theta\|_2^2$. The two main results of this section are stated next.

\begin{thm}\label{thm.main_minimax_est_vector_hd}
Consider the Gaussian sequence model with sparsity $s\ll\sqrt{n}.$ Any estimator $\wh \theta$ that attains the minimax estimation rate $s\log(n)$ with respect to the worst case risk $\sup_{\theta \in \Theta(s)} E_\theta[\|\wh \theta-\theta\|_2^2]$ also satisfies for all sufficiently large $n,$
\begin{align*}
    &\sup_{\theta\in \Theta(s)}\big\|E_\theta\big[\wh \theta\big] -\theta\big\|_2^2 \asymp s\log(n),
    \text{ and } 
    \sup_{\theta\in \Theta(s)}\sum_{i=1}^n \Var_\theta\big(\wh \theta_i\big) \geq \frac s2.
\end{align*}
Moreover, if $s\leq n^{1/2-\delta}$ for some  $0<\delta<1/2,$ then there exists an estimator attaining the minimax estimation rate with $\sup_{\theta\in \Theta(s)}\sum_{i=1}^n \Var_\theta(\wh \theta_i)\lesssim s.$
\end{thm}

The result shows that for a minimax rate optimal estimator, squared bias and variance do not necessarily need to be of the same order and the rate of the variance can be slower by at most a $\log(n)$-factor. 

One might wonder whether the proposed lower bound technique can be extended for sparsity $s \gg \sqrt{n}.$ While this question remains open, we now prove that for estimation of the quadratic functional a phase transition occurs if the sparsity is of the order $\sqrt{n}.$ For sparsity $s\ll \sqrt{n}$ the bias-variance trade-off is non-trivial, but for sparsity $s\gtrsim \sqrt{n},$ we can find an unbiased estimator achieving the minimax estimation rate.

%The fact that $\sqrt{n}$ appears as an upper bound on the sparsity might be related to the testing theory in the Gaussian sequence model. It is well-known that for sparse models with sparsity $s \ll \sqrt{n},$ we cannot consistently test for signal in the sparse Gaussian mixture formulation. On the contrary, for any $s = n^{1/2+\delta}$ with $\delta>0$ this is indeed possible, see \cite{MR1680087, MR2065195, MR2382653}.

For estimation of the quadratic functional, consider the parameter space
\begin{align}\label{eq.Theta_n^2(s)_def}
    \Theta_n^2(s) := \Theta(s)\cap
    \Big\{\theta:\sum_{i=1}^n \theta_i^2
    \leq 2s\log\Big(1+\frac{\sqrt{n}}s\Big)\Big\}.
\end{align}
Those are all $s$-sparse vectors with squared Euclidean norm bounded by $2s\log(1+\sqrt{n}/s).$ We have chosen this specific threshold as it leads to the most unusual behavior of the bias-variance trade-off. For this parameter space, the minimax estimation rate for the functional $\theta \mapsto \|\theta\|_2^2$ with respect to the MSE is 
\begin{align}
    s^2\log^2\Big(1+\frac{\sqrt{n}}s\Big)\asymp s^2\log^2\Big(\frac n{s^2}\Big)\vee n
    \label{eq.minimax_rate_fctal}
\end{align}
as stated in \cite{MR3662444}, Corollary 1. See also Appendix \ref{sec:proofs_highdimensional} for more details about \eqref{eq.minimax_rate_fctal}.

\begin{thm}\label{thm.main_minimax_est_fctal_hd}
Consider estimation of the functional $\theta \mapsto \|\theta\|_2^2$ in the Gaussian sequence model with sparsity $s$ and parameter space $\Theta_n^2(s).$ 
\begin{itemize}
    \item[(i)] If $s\ll \sqrt{n},$ then, the minimax estimation rate is $s^2\log^2(n/s^2)$ and any estimator $\wh{\|\theta\|_2^2}$ attaining the minimax optimal estimation rate must satisfy
    \begin{align*}
        \sup_{\theta\in \Theta_n^2(s)}\big(E_\theta\big[\wh{\|\theta\|_2^2}\big] -\|\theta\|_2^2\big)^2\asymp s^2\log^2\Big(\frac n{s^2}\Big).
    \end{align*}
    for all sufficiently large $n.$ Moreover, if $s\leq n^{1/2-\delta}$ for some $0<\delta<1/2,$ then there exists a minimax rate optimal estimator $\wh{\|\theta\|_2^2}$ with $\sup_{\theta\in \Theta_n^2(s)} \Var_\theta(\wh{\|\theta\|_2^2}) \lesssim s\log(n/s^2).$

    \item[(ii)] If $s\gtrsim \sqrt{n},$ then, there exists a minimax rate optimal estimator that is unbiased.
\end{itemize}
\end{thm}

For sparsity of the order $o(\sqrt{n}),$ every minimax rate optimal estimator will have necessarily a worst case squared bias that is of the same order as the minimax rate. But worst case squared bias and variance do not have to be of the same order if $s\to \infty$. Indeed, the second part of $(i)$ shows existence of a minimax rate optimal estimator with variance $s\log(n/s^2) \ll s^2\log^2(n/s^2)=$ minimax estimation rate.

Surprisingly there is a phase transition if $s$ is of the order $\sqrt{n}.$ If $s\gtrsim \sqrt{n},$ suddenly unbiased estimation is possible, which means that now the variance is dominating the risk.

That typically either squared bias or variance dominates seems to be symptomatic for estimation of functionals. For instance, for estimation of the squared functional $f\mapsto \int f^2$ in the Gaussian white noise model, we conjecture that if the H\"older smoothness of $f$ is below $1/4$, the squared bias will dominate, whereas for smoothness indices above $1/4$, the convergence rate is driven in first order by the variance.

Below we analyze the two main results above in more detail. Since the bias-variance lower bounds are very different from the ones in the previous chapters, we discuss the lower bounds on the variance and the lower bounds on the bias in separate subsections. All proofs of this section are deferred to Appendix \ref{sec:proofs_highdimensional}.
% The most extreme variance reduction occurs for the case of a completely black signal, that is, $\theta=(0,\dots,0)^\top.$

\textit{Lower bounds on the variance:} Using the lower bound technique based on multiple probability distributions, we can derive a lower bound for the variance at zero of any estimator that satisfies a bound on the bias. 

\begin{thm}\label{thm.sparsity_lb}
    Consider the Gaussian sequence model with sparsity $0<s\leq \sqrt{n}/2.$ Given an estimator $\wh \theta$ and a real number $\gamma$ such that $4\gamma + 1/\log(n/s^2)\leq 0.99$ and
    \begin{align*}
        \sup_{\theta \in \Theta(s)} \big\|E_\theta\big[\wh \theta\big] -\theta\big\|_2^2 \leq \gamma s \log\Big(\frac{n}{s^2}\Big),
    \end{align*}
     then, for all sufficiently large $n,$
    \begin{align*}
        \sum_{i=1}^n \Var_0\big(\wh \theta_i\big)
        \geq \frac{(1-(1/2)^{0.01})}{25 e \log(n/s^2)} n \Big(\frac{s^2}{n}\Big)^{4\gamma},
    \end{align*}
    where $\Var_0$ denotes the variance for parameter vector $\theta=(0,\dots,0)^\top.$
\end{thm}

Compared to pointwise estimation, the result shows a different type of bias-variance trade-off. Decreasing the constant $\gamma$ in the upper bound for the bias, increases the rate in the lower bound for the variance. For instance, in the regime $s \leq n^{1/2-\delta},$ with $0<\delta<1/2,$ we can find for any $\rho>0$ a sufficiently small constant $\gamma$, such that the lower bound is of the form constant$\times n^{1-\rho}.$ As a consequence of the bias-variance decomposition \eqref{eq.BV_vectors}, the maximum quadratic risk of such an estimator in this regime is also lower-bounded by $\gtrsim n^{1-\rho}.$ Reducing the constant of the bias will therefore necessarily lead to estimators with highly suboptimal estimation risk.

The proof of Theorem \ref{thm.sparsity_lb} applies the $\chi^2$-divergence lower bound \eqref{eq.row_sum_norm_nd} by comparing the data distribution induced by the zero vector to the $\binom{n}{s}$ many distributions corresponding to $s$-sparse vectors with non-zero entries $\sqrt{4\gamma\log(n/s^2)+1}$. By \eqref{eq.chi2_normal_distr},
the size of the $(j,k)$-th entry of the $\chi^2$-divergence matrix is completely described by the number of components on which the corresponding $s$-sparse vectors are both non-zero. The whole problem reduces then to a combinatorial counting argument. The key observation is that if we fix an $s$-sparse vector, say $\theta^*,$ there are of the order $n/s^2$ more $s$-sparse vectors that have exactly $r-1$ non-zero components in common with $\theta^*$ than $s$-sparse vectors that that have exactly $r$ non-zero components in common with $\theta^*$. This means that as long as $s \ll \sqrt{n},$ most of the $s$-sparse vectors are (nearly) orthogonal to $\theta^*$.

The lower bound in Theorem \ref{thm.sparsity_lb} can be extended to several related problems by invoking the data processing inequality
\eqref{eq.data_processing}.
As an example suppose that we observe only $X_1^2,\dots,X_n^2$ with $(X_1,\dots,X_n)$ the data from the Gaussian sequence model. As parameter space, consider the class $\Theta_+(s)$ of $s$-sparse vectors with non-negative entries. This choice is natural as the parameter $\theta$ is not identifiable in this model over the full space of $s$-sparse vectors $\Theta(s)$. Since the proof of Theorem \ref{thm.sparsity_lb} only uses parameters in $\Theta_+(s),$ the same lower bound as in Theorem \ref{thm.sparsity_lb} holds also in this modified setting.
The next result shows an analogous version of Theorem \ref{thm.sparsity_lb} for estimation of the functional $\theta \mapsto \|\theta\|_2^2.$

\begin{thm}\label{thm.sparsity_sq_functional}
    Consider the Gaussian sequence model with parameter space $\Theta_n^2(s)$ defined in \eqref{eq.Theta_n^2(s)_def} and sparsity $0<s\leq \sqrt{n}/2.$ Given an estimator $\wh{\|\theta\|_2^2}$ of $\|\theta\|_2^2$ and a real number $\gamma$ such that $2\gamma + 1/\log(n/s^2)\leq 0.99$ and
    \begin{align*}
        \sup_{\theta \in \Theta_n^2(s)} \big|\Bias_\theta\big(\wh{\|\theta\|_2^2}\big)\big| \leq \gamma s \log\Big(\frac{n}{s^2}\Big),
    \end{align*}
     then, for all sufficiently large $n,$
    \begin{align*}
        \Var_0 \big( \wh{\|\theta\|_2^2} \big)
        \geq \frac{1-(1/2)^{0.01}}{e} n
        \Big(\frac{s^2}{n}\Big)^{2\gamma},
    \end{align*}
    where $\Var_0$ denotes the variance for parameter vector $\theta=(0,\dots,0)^\top.$
\end{thm}

Notice that the upper bound in the previous result is for the bias, not the squared bias.

\medskip

\textit{A lower bound for the bias:} What can be said about the bias for small variance? The next result shows that if the variance is strictly smaller than $s/2$, then the worst case bias is infinite. 

\begin{thm}\label{thm.functional_GSM}
Consider the Gaussian sequence model with sparsity $1\leq s \leq n$ and assume that $\wh \theta=(\wh \theta_1,\ldots,\wh \theta_n)$ is an estimator such that $E_\theta[\wh \theta_i]$ exists and is finite for all $i=1,\ldots,n$ and all $\theta\in \Theta(s).$  \\ If
$\sup_{\theta \in \Theta(s)} \, \sum_{i=1}^n \Var_\theta\big(\wh \theta_i\big) < s/2,$
then
$\sup_{\theta \in \Theta(s)} \,
\big\| E_\theta\big[\wh \theta\big]-\theta\big\|_2 = \infty.
$
\end{thm}

\textit{Nearly matching upper bounds:} To show that the rates in the derived lower bounds are nearly sharp, we now establish corresponding upper bounds. For an estimator thresholding small observations, the variance under $P_0$ is determined by both the probability that an observation falls outside the truncation level and the value it is then assigned to. One can further reduce the variance at zero if large observations are shrunk as much as possible to zero. The bound on the bias dictates the largest possible truncation level. To obtain matching upper bounds, this motivates then to study the soft-thresholding estimator
\begin{align}
    \wh \theta_i = \sign(X_i)\Big(|X_i|-\sqrt{\gamma \log(n/s^2)}\Big)_+, \quad i=1,\dots,n.
    \label{eq.soft_thresh_est}
\end{align}
If $\theta_i=0,$ then $E_\theta[\wh \theta_i]=0.$ For $\theta_i\neq 0,$ one can use $|\wh \theta_i-X_i|\leq \sqrt{\gamma \log(n/s^2)}$ and $E_\theta[X_i]=\theta_i$ to verify that the squared bias $\|E_\theta[\wh \theta]-\theta\|_2^2$ is bounded by $\gamma s \log(n/s^2),$ uniformly over the space of $s$-sparse vectors $\Theta(s).$ As an estimator for the functional $\|\theta\|_2^2$, we study
\begin{align}
    \wh{\|\theta\|_2^2}
    = \sum_{i=1}^n \Big( \big(X_i^2 - \gamma\log(n/s^2)\big)_+- E_{\xi \sim \Nc(0,1)} \big[\big(\xi^2 - \gamma\log(n/s^2)\big)_+\big] \Big).
    \label{eq.soft_thresh_est_func}
\end{align}

\begin{lemma}\label{lem.soft_thresh_ub}
For the soft-thresholding estimator $\wh \theta=(\wh \theta_1,\dots, \wh\theta_n)^\top$ defined in \eqref{eq.soft_thresh_est}, we have
\begin{align}
    &\sum_{i=1}^n\Var_0\big(\wh \theta_i\big)
    \leq \frac{\sqrt{2}}{\sqrt{\pi \gamma^3 \log^3(n/s^2)}}
    n \Big(\frac{s^2}{n}\Big)^{\frac{\gamma}{2}},
    \text{ and,}
    \label{eq.soft_thresh_ub1} \\
    &\text{for any } \theta \in \Theta(s), \
    \sum_{i=1}^n\Var_\theta\big(\wh \theta_i\big)
    \leq 4s+\frac{\sqrt{2}}{\sqrt{\pi \gamma^3 \log^3(n/s^2)}}
    n \Big(\frac{s^2}{n}\Big)^{\frac{\gamma}{2}}.
    \label{eq.soft_thresh_ub2}
\end{align}
Moreover, for any $n, s, \gamma,$ for which $\gamma \log(n/s^2) \geq 2,$ we have for the estimator $\wh{\|\theta\|_2^2}$ defined in \eqref{eq.soft_thresh_est_func},
\begin{align}
    &\sup_{\theta \in \Theta(s)}\big|\Bias_\theta(\wh{\|\theta\|_2^2})\big|\leq \gamma s \log\Big(\frac n{s^2}\Big),
    \label{eq.soft_thresh_ub3} \\
    &\Var_0\big(\wh{\|\theta\|_2^2}\big)
    \leq 
    \frac{8}{\sqrt{\gamma \log(n/s^2)}}
    n \Big(\frac{s^2}{n}\Big)^{\frac{\gamma}2},
    \text{ and,}
    \label{eq.soft_thresh_ub4} \\
    &
    \Var_\theta\big(\wh{\|\theta\|_2^2}\big)
    \leq \|\theta\|_2^2+3s+
    \frac{8}{\sqrt{\gamma \log(n/s^2)}}
    n \Big(\frac{s^2}{n}\Big)^{\frac{\gamma}2}.
    \label{eq.soft_thresh_ub5}
\end{align}
\end{lemma}

The constraint $\gamma \log(n/s^2) \geq 2$ holds for all sufficiently large $n$, whenever $\gamma$ is fixed and $s\ll \sqrt{n}.$

Compared with Theorem \ref{thm.sparsity_lb}, the corresponding upper bound \eqref{eq.soft_thresh_ub1} has the same structure. Key difference is that the exponent is $4\gamma$ in the lower bound and $\gamma/2$ in the upper bound. As discussed already, this discrepancy seems to be due to the lower bound. If instead of a tight control of the variance at zero, one is interested in a global bound on the variance over the whole parameter space, one could gain a factor $4$ in the exponent by relying on the Hellinger version using Theorem \ref{thm.multiple_lb}(ii) instead of (i). A second difference is that there is an additional factor $1/\sqrt{\log(n/s^2)}$ in the upper bound. This extra factor tends to zero which seems to be a contradiction. Notice, however, that this is compensated by the different exponents $(s^2/n)^{\gamma/2}$ and $(s^2/n)^{4\gamma}.$
It is also not hard to see that for the hard thresholding estimator with truncation level $\sqrt{\gamma \log(n/s^2)},$ the variance $\sum_{i=1}^n\Var_0(\wh \theta_i)$ is of order $n(s^2/n)^{\gamma/2}.$

The upper bound in \eqref{eq.soft_thresh_ub4} corresponds to the lower bound in Theorem \ref{thm.sparsity_sq_functional}. The differences between upper and lower bound are similarly as the ones between the upper bound \eqref{eq.soft_thresh_ub1} and Theorem \ref{thm.sparsity_lb} discussed in the previous paragraph. 

If $s\leq n^{1/2-\delta}$ for some $0<\delta<1/2,$ then by choosing $\gamma$ large enough, one can show that \eqref{eq.soft_thresh_ub2} implies $\sup_{\theta\in\Theta(s)} \sum_{i=1}^n\Var_\theta(\wh \theta_i)\lesssim s.$ This yields then the last statement of Theorem \ref{thm.main_minimax_est_vector_hd}. Similarly, one can use \eqref{eq.soft_thresh_ub5} to construct an estimator satisfying the variance bound in Theorem \ref{thm.main_minimax_est_fctal_hd}(i).

The soft-thresholding estimator \eqref{eq.soft_thresh_est} does not produce an $s$-sparse model. Indeed, from the tail decay of the Gaussian distribution, one expects that the sparsity of the reconstruction for $\theta=(0,\ldots,0)$ is $n(s^2/n)^{\gamma/2}$ which can be considerably bigger than $s$ for small values of $\gamma$. Because testing for signal is very hard in the sparse sequence model, it is unclear whether one can reduce the variance further by projecting it to an $s$-sparse set without inflating the bias.

\textit{Extension to high-dimensional regression:} The lower bound can also be extended to a useful lower bound on the interplay between bias and variance in sparse high-dimensional regression. Suppose we observe $Y = X \beta + \varepsilon$ where $Y$ is a vector of size $n$, $X$ is an $n \times p$ design matrix, $\varepsilon \sim \Nc(0, I_n)$ and $\beta$ is a vector of size $p$ to be estimated. Again denote by $\Theta(s)$ the class of $s$-sparse vectors. We impose the common assumption that the diagonal coefficients of the Gram matrix $X^\top X$ are standardized such that $(X^\top X)_{i,i} = n$ for all $i=1, \dots, p$ (see for instance also Section 6 in \cite{MR2807761}). Define the mutual coherence condition number by $\mc(X) := \max_{1 \leq i \neq j \leq n} (X^\top X)_{i,j} / (X^\top X)_{i,i}$. This notion goes back to \cite{MR2237332}. Below, we work under the restriction $\mc(X)\leq 1/(s^2\log(p/s^2)).$ This is stronger than the mutual coherence bound of the form constant$/s$ normally encountered in high-dimensional statistics. As this is not the main point of the paper, we did not attempt to derive the theorem under the sharpest possible condition and also only provide the generalization of Theorem \ref{thm.sparsity_lb}.

\begin{thm}\label{thm.sparseReg_lb}
Consider the sparse high-dimensional regression model with Gaussian noise. Let $0 < s \leq \sqrt{p}/2,$ and $\mc(X)\leq 1/(s^2\log(p/s^2)).$ Given an estimator $\wh \beta$ and a real number $\gamma$ such that $4\gamma + 1/\log(p/s^2)\leq 0.99$ and 
$\sup_{\beta \in \Theta(s)} \big\|E_\beta\big[\wh \beta\big] -\beta\big\|^2 \leq (\gamma s/n) \log(p/s^2),$
     then, for all sufficiently large $p,$
    \begin{align*}
        \sum_{i=1}^p \Var_0\big(\wh \beta_i\big)
        \geq \frac{(1-(1/2)^{0.01})}{25 e^2 \log(p/s^2)}  \frac{p}{n} \Big(\frac{s^2}{p}\Big)^{4\gamma},
    \end{align*}
    where $\Var_0$ denotes the variance for parameter vector $\beta=(0,\dots,0)^\top.$
\end{thm}

\section{Discussion}
\label{sec.discussion}

\subsection{General definition of a bias-variance trade-off}~

The proper definition of the bias-variance trade-off depends on some subtleties underlying the choice of the space of values that can be attained by an estimator, subsequently denoted by $\Ac$. To illustrate this, suppose we observe $X \sim \Nc(\theta,1)$ with parameter space $\Theta=\{-1,1\}.$ For any estimator $\wh \theta$ with $\Ac=\Theta,$ $E_{1}[\wh \theta]<1$ or $E_{-1}[\wh \theta]>-1.$ Thus, no unbiased estimator with $\Ac=\Theta$ exists. If the estimator is, however, allowed to take values on the real line, then $\widehat \theta = X$ is an unbiased estimator for $\theta.$ We believe that the correct way to derive lower bounds on the bias-variance trade-off is to allow the action space $\Ac$ to be large. Whenever $\Theta$ is a class of functions on $[0,1]$, the derived lower bounds are over all estimators with $\Ac$ the real-valued functions on $[0,1];$ for high-dimensional problems with
$\Theta \subseteq \Rb^p,$ the lower bounds are over all estimators with $\Ac=\Rb^p.$ In particular, if the true parameter vector is assumed to be sparse, we do not require the estimator to be sparse.

Given a statistical model $(P_\theta)_{\theta \in \Theta}$, consider a symmetric (and non-negative) loss function $\ell(\theta,\theta')=\ell(\theta',\theta).$ The risk of an estimator $\wh \theta$ is then $E_\theta[\ell(\wh \theta,\theta)].$ If $E_\theta[\wh \theta]$ exists, we call $\ell(\theta,E_\theta[\wh \theta])$ the deterministic error and $E_\theta[\ell(\wh \theta,E_\theta[\wh \theta])]$ the stochastic error. 
If for all estimators $\wh \theta$ and all parameters $\theta,$ $E_\theta[\ell(\wh \theta,\theta)]=\ell(\theta,E_\theta[\wh \theta])+E_\theta[\ell(\wh \theta,E_\theta[\wh \theta])],$ then we say that a (generalized) bias-variance decomposition holds and refer to the deterministic error $\ell(\theta,E_\theta[\wh \theta])$ as the squared bias part and to the stochastic error $E_\theta[\ell(\wh \theta,E_\theta[\wh \theta])]$ as the variance part. Note that the squared bias is defined directly without introducing first a notion of bias.

A bias-variance decomposition exists if $\ell(\theta,\theta')=\|\theta-\theta'\|^2$ with $\|\cdot\|$ a Hilbert space norm. In particular for $\ell(\theta,\theta')=(\theta-\theta')^2,$ we have the classical bias-variance decomposition of the MSE. On a vector space $\Theta,$ the loss function $\ell(\theta,\theta')=\|\theta-\theta'\|_2^2$ leads to the decomposition \eqref{eq.BV_vectors}. In this case the squared bias part is $\|E_\theta[\wh \theta] -\theta\|_2^2$ and the variance part is $\sum_{i=1}^n \Var_\theta(\wh \theta_i).$ For $\Theta$ consisting of $L^2$-functions, the decomposition of the MISE in integrated squared bias and integrated variance in \eqref{eq.MISE_decomp} is another example. 

With this definition of squared bias and variance, we can now define a bias-variance trade-off informally as either a restriction of the squared bias part that follows from imposing a constraint on the variance part or a restriction on the variance part that is implied by a constraint on the squared bias part. To introduce a formal definition, denote the squared bias part by $B_\theta(\wh \theta)^2$ and the variance part by $V_\theta(\wh \theta).$ The functions $\theta \mapsto B_\theta(\wh \theta)$ and $\theta \mapsto V_\theta(\wh \theta)$ belong to the space $[0,+\infty]^\Theta$. Let $\Tc \subset \Theta^\Xc$ be a class of estimators and let $\psi$ be a function $\psi: [0,+\infty]^\Theta \times [0,+\infty]^\Theta \to [0,+\infty]$ increasing in both of its arguments in the sense that for $b_1(\cdot)^2 \leq b_2(\cdot)^2$ and $v_1(\cdot) \leq v_2(\cdot)$, $\psi(b_1^2, v_1) \leq \psi(b_2^2,v_2)$.
We say that $\psi$ is a bias-variance trade-off for the class of estimators $\Tc$ if
$\inf_{\wh \theta \in \Tc}
\psi \Big( \theta \mapsto B_\theta(\wh \theta)^2 \, , \,
\theta \mapsto V_\theta(\wh \theta) \Big) 
\geq 1.$
All bias-variance trade-offs derived in this paper can be put into this form. For instance, for the two bias-variance trade-offs for pointwise estimation in Theorem \ref{thm.pointwise}, we can choose using the notation introduced in Section \ref{sec:gauss_wn},
\begin{align*}
    &\psi(b^2,v)
    := 
    \frac{n}{\gamma(R,\beta)} \, \sup_{f\in \Cc^\beta(R)} |b(f)|^{1/\beta}  \sup_{f\in \Cc^\beta(R)} v(f),
    \text{ and,} \\
    & \psi(b^2,v)
    := n \, \sup_{f\in \Cc^\beta(R)} |b(f)|^{1/\beta}
    \inf_{f\in \Cc^\beta(R)} \frac{v(f)}{\ol \gamma(R,\beta,C,\|f\|_{\Cc^\beta})}.
\end{align*}

\subsection{Comparison of the abstract lower bounds for the bias-variance trade-off and the hypothesis testing approach for minimax lower bounds}

While non-trivial minimax rates exist for parametric and non-parametric problems alike, the bias-variance trade-off phenomenon occurs mainly in high-dimen\-sional and infinite dimensional models. Despite these differences, the here proposed strategy for lower bounds on the bias-variance trade-off and the well-developed testing approach for lower bounds on the minimax estimation rate share some similarities. A clear similarity is that for both approaches, the problem is reduced in a first step by selecting a discrete subset of the parameter space. To achieve rate-optimal minimax lower bounds, it is well-known that for a large class of functionals, reduction to two parameters is sufficient. On the contrary, optimal lower bounds for global loss functions, such as $L^p$-loss in nonparametric regression, require to pick a number of parameter values that increases with the sample size. We argued in this work that a similar distinction occurs also for bias-variance trade-off lower bounds. As in the case of the minimax estimation risk, we can relate the two-parameter lower bounds to a bound with respect to any of the commonly used information measures including the Kullback-Leibler divergence. 

More pronounced differences occur in the formulation of both lower bound techniques for lower bounds involving more than two parameter values. While for minimax lower bounds the parameters correspond to several hypotheses that form a local packing of the parameter space, for bias-variance trade-off lower bounds the contribution of the selected parameters is determined by how orthogonal the corresponding distributions are. Here, the orthogonality of distributions is measured by the $\chi^2$-divergence divergence matrix or the Hellinger affinity matrices, see Table 1 in
% \ref{tab.1}
\cite{DerumignySchmidtHieber2023}
for examples.

By Proposition 3.2(ii) in \cite{DerumignySchmidtHieber2023}, $\v^\top \chi^2(P_0 | P_1, \dots, P_M) \v = \chi^2(\sum_{j=1}^M v_j P_j, P_0),$ 
    where $\sum_{j=1}^M v_j P_j$ is the mixture (signed) measure of $P_1, \dots, P_M.$ This suggests to interpret the case of multiple measures $P_0,\ldots,P_M$ as a two point testing problem, where we measure the information distance between $P_0$ and a linear combination $\sum_{j=1}^M v_jP_j.$  Viewed from this perspective, the proposed approach shares some similarities with the minimax lower bounds based on two fuzzy hypothesis and Fano's lemma. For a description of these approaches, see Section 2.7 in \cite{tsybakov2009introduction}.

\subsection{Bias-variance trade-off lower bounds and their proof techniques}

\begin{table}[ht]
    \centering
    \begin{tabular}{ccc}
        % \hline
        \textbf{Framework} & \textbf{Theorem} & \textbf{Proof technique} \\
        \hline
        Pointwise estimation in 
        & \multirow{2}{*}{\ref{thm.pointwise}}
        & \multirow{2}{*}{Univariate change of expectation} \\
        the Gaussian white noise model && \\
        \hline
        Pointwise estimation of the
        & \multirow{2}{*}{\ref{thm.B_V_tradeoff_boundary}}
        & \multirow{2}{*}{Univariate change of expectation} \\
        boundary of a Poisson point process && \\
        \hline
        Function estimation in the Gaussian
        & \multirow{2}{*}{\ref{thm.LB_L2}} & 2 reductions + \\
        white noise model with $L_2$-loss & & multivariate change of expectation \\
        \hline
        Estimation of $\theta$ in the Gaussian
        & \multirow{2}{*}{\ref{thm.sparsity_lb} and \ref{thm.functional_GSM}}
        & Multivariate change of expectation \\
        sequence model under sparsity && ( + 1 reduction for Theorem \ref{thm.functional_GSM}) \\
        \hline
        Estimation of $\| \theta \|_2^2$ in the Gaussian
        & \multirow{2}{*}{\ref{thm.sparsity_sq_functional}}
        & \multirow{2}{*}{Multivariate change of expectation} \\
        sequence model under sparsity && \\
        \hline
        Sparse high-dimensional regression
        & \multirow{2}{*}{\ref{thm.sparseReg_lb}}
        & \multirow{2}{*}{Multivariate change of expectation} \\
         with Gaussian noise && \\
        % \hline
    \end{tabular}
    \caption{Proof techniques for different examples}
    \label{tab:thm_proofs_tech}
\end{table}

Table \ref{tab:thm_proofs_tech} states the applied proof technique for each of the lower bounds proved in this article. Here ``univariate change of expectation'' refers to the use of Lemma~\ref{lem.general_lb} or Lemma~\ref{lem.lb_limit}; ``multivariate change of expectation'' refers to applying Theorem~\ref{thm.multiple_lb} or the variations stated in Equations~\eqref{eq.bound_lambda1_chi_var} and~\eqref{eq.row_sum_norm_nd}.

%======================
%======================

\section*{Acknowledgements}
We are grateful to Ming Yuan for helpful discussions during an early stage of the project, to Zijian Guo for pointing us to the article \cite{MR3662444}, and  to Tomohiro Nishiyama for mentioning a typo in an earlier version. We thank two reviewers and an Associate Editor for many helpful comments and suggestions that significantly improved the manuscript. The project has received funding from the Dutch Science Foundation (NWO) via the Vidi grant VI.Vidi.192.021.

\begin{supplement}[id=suppA]
  \sname{Supplement to ``On lower bounds for the bias-variance trade-off''}
  \stitle{}
  \slink[doi]{10.1214/00-AOSXXXXSUPP}
  \sdatatype{.pdf}
  \sdescription{All proofs are given in the supplement.}
\end{supplement}

\bibliographystyle{acm}       % (uses file "plain.bst")
\bibliography{biblio}           % expects file "refsPart1.bib"

\newpage
\setcounter{page}{1}

% indicate corresponding author with \corref{}
% \author{\fnms{John} \snm{Smith}\corref{}\ead[label=e1]{smith@foo.com}\thanksref{t1}}
% \thankstext{t1}{Thanks to somebody} 
% \address{line 1\\ line 2\\ printead{e1}}
% \affiliation{Some University}

\begin{frontmatter}

% "Title of the paper"
\title{Supplement to ``On lower bounds for the bias-variance trade-off''}
\runtitle{On lower bounds for the bias-variance trade-off}

% indicate corresponding author with \corref{}
% \author{\fnms{John} \snm{Smith}\corref{}\ead[label=e1]{smith@foo.com}\thanksref{t1}}
% \thankstext{t1}{Thanks to somebody} 
% \address{line 1\\ line 2\\ printead{e1}}
% \affiliation{Some University}

\author{\fnms{Alexis} \snm{Derumigny}\ead[label=e3]{a.f.f.derumigny@tudelft.nl}
\and 
\fnms{Johannes} \snm{Schmidt-Hieber}\ead[label=e4]{a.j.schmidt-hieber@utwente.nl}
}
\address{Delft University of Technology, \\
Mekelweg 4, \\ 2628 CD Delft \\ The Netherlands.\\ \printead{e3}}
\address{University of Twente,\\ P.O. Box 217, \\ 7500 AE Enschede \\ The Netherlands\\ \printead{e4}}
\affiliation{Delft University of Technology \and University of Twente}

\runauthor{A. Derumigny and J. Schmidt-Hieber}

\begin{abstract}
The supplement contains additional material for the article ``On lower bounds for the bias-variance trade-off''.
\end{abstract}

% -----------
% Special comments for deleting the keywords at this point so that no footnote is created 
\makeatletter
\let\keyword@fmt\relax
\makeatother

\end{frontmatter}

\appendix

\section{Proofs for Section \ref{sec.general_low_bds}}
\label{sec.proofs_gen_lower_bd}

\begin{proof}[Proof of Lemma \ref{lem.general_lb}]
We first prove \eqref{eq.lem1_1}. Applying the Cauchy-Schwarz inequality, for any real number $a,$
\begin{align*}
    \big|E_P[X]-E_Q[X] \big|
    &= \Big|\int (X(\om)-a) (p(\om)-q(\om)) \, d\nu(\om)\Big|\\
    &\leq \Big(\int(X(\om)-a)^2|p(\om)-q(\om)| \, d\nu(\om)\Big)^{1/2}\sqrt{2\TV(P,Q)}.
\end{align*}
We can bound $|p(\om)-q(\om)|\leq p(\om)+q(\om)$ and $E_P[(X-a)^2]=\Var_P(X)+(E_P[X]-a)^2$ (which holds for all $P$ and all $a$), to deduce that for $a_*:=(E_P[X]+E_Q[X])/2,$
\begin{align*}
    \int(X(\om)-a_*)^2 &|p(\om)-q(\om)| \, d\nu(\om) \\
    &\leq \Var_P(X)+\Var_Q(X)+ 2\Big(\frac{E_P[X]-E_Q[X]}{2}\Big)^2.
\end{align*}
This shows that 
\begin{align*}
    \big(E_P[X]-E_Q[X] \big)^2
    \leq 
    \Big( \Var_P(X) &+\Var_Q(X) \\
    & + \frac{(E_P[X]-E_Q[X])^2}{2}\Big) 2\TV(P,Q).
\end{align*}
Rearranging the inequality yields \eqref{eq.lem1_1}.

We now prove \eqref{eq.lem1_2}. 
\cite[Theorem 1]{nishiyama2020tightHellinger} states that for any random variable $X,$
\begin{align*}
    H^2(P,Q) \geq 1 - \sqrt{1 -
    \frac{(E_P[X] - E_Q[X])^2}
    {(E_P[X] - E_Q[X])^2 + (\sqrt{\Var_P(X)}+ \sqrt{\Var_Q(X)})^2}}.
\end{align*}
% \begin{align*}
%     (H^2(P,Q) - 1) \geq - \sqrt{1 -
%     \frac{(E_P[X] - E_Q[X])^2}
%     {(E_P[X] - E_Q[X])^2 + (\sqrt{\Var_{\vphantom{Q}P}(X)}+ \sqrt{\Var_Q(X)})^2}}
% \end{align*}
% \begin{align*}
%     (1 - H^2(P,Q)) \leq \sqrt{1 -
%     \frac{(E_P[X] - E_Q[X])^2}
%     {(E_P[X] - E_Q[X])^2 + (\sqrt{\Var_{\vphantom{Q}P}(X)}+ \sqrt{\Var_Q(X)})^2}}
% \end{align*}
% \begin{align*}
%     (1 - H^2(P,Q) )^2 - 1 \leq -
%     \frac{(E_P[X] - E_Q[X])^2}
%     {(E_P[X] - E_Q[X])^2 + (\sqrt{\Var_{\vphantom{Q}P}(X)}+ \sqrt{\Var_Q(X)})^2}
% \end{align*}
Rewriting the previous equation and squaring gives
\begin{align*}
    1 - (1 - H^2(P,Q) )^2 \geq
    \frac{(E_P[X] - E_Q[X])^2}
    {(E_P[X] - E_Q[X])^2 + (\sqrt{\Var_{\vphantom{Q}P}(X)}+ \sqrt{\Var_Q(X)})^2}.
\end{align*}
% \begin{align*}
%     &(E_P[X] - E_Q[X])^2
%     \big(1 - (1 - H^2(P,Q) )^2 \big) \\
%     &+ (\sqrt{\Var_{\vphantom{Q}P}(X)} + \sqrt{\Var_Q(X)})^2 
%     \big(1 - (1 - H^2(P,Q) )^2 \big)
%     \geq (E_P[X] - E_Q[X])^2
% \end{align*}
% \begin{align*}
%     (E_P[X] - E_Q[X])^2
%     \times \big( - (1 - H^2(P,Q) )^2 \big)
%     + \big(\sqrt{\Var_{\vphantom{Q}P}(X)} + \sqrt{\Var_Q(X)}\big)^2 
%     \big(1 - (1 - H^2(P,Q) )^2 \big)
%     \geq 0
% \end{align*}
% \begin{align*}
%     (\sqrt{\Var_{\vphantom{Q}P}(X)} + \sqrt{\Var_Q(X)})^2 
%     \big(1 - (1 - H^2(P,Q) )^2 \big)
%     \geq (1 - H^2(P,Q) )^2 (E_P[X] - E_Q[X])^2
% \end{align*}
% \begin{align*}
%     \frac{1 - (1 - H^2(P,Q) )^2}{(1 - H^2(P,Q) )^2}
%     (\sqrt{\Var_{\vphantom{Q}P}(X)} + \sqrt{\Var_Q(X)})^2 
%     \geq  (E_P[X] - E_Q[X])^2
% \end{align*}
% \begin{align*}
%     \frac{(1 - H^2(P,Q) )^2}{1 - (1 - H^2(P,Q) )^2}
%     (E_P[X] - E_Q[X])^2
%     \leq (\sqrt{\Var_{\vphantom{Q}P}(X)} + \sqrt{\Var_Q(X)})^2 
% \end{align*}
% \begin{align*}
%     \frac{(1 - H^2(P,Q) )^2}{H^2(P,Q) ( 2 - H^2(P,Q))}
%     (E_P[X] - E_Q[X])^2
%     \leq (\sqrt{\Var_{\vphantom{Q}P}(X)} + \sqrt{\Var_Q(X)})^2 
% \end{align*}
Solving for $(E_P[X] - E_Q[X])^2$, we obtain
\begin{align*}
    \frac{(1 - H^2(P,Q))^2}{H^2(P,Q) (2-H^2(P,Q))}
    (E_P[X] - E_Q[X])^2
    \leq \bigg(\sqrt{\Var_{\vphantom{Q}P}(X)} + \sqrt{\Var_Q(X)} \bigg)^2.
\end{align*}
Using that for any positive real numbers $u,v,$ $(\sqrt{u}+\sqrt{v})^2\leq 2u+2v$ yields 
\begin{align*}
    \frac{(1 - H^2(P,Q))^2}{2 H^2(P,Q) (2-H^2(P,Q))}
    (E_P[X] - E_Q[X])^2
    \leq \Var_{P}(X) + \Var_Q(X),
\end{align*}
as claimed.

\medskip

We also provide a more direct proof of the slightly weaker version of~\eqref{eq.lem1_2},
\begin{align}
    \frac{( E_P[X]-E_Q[X])^2}{4}
    \Big( \frac{1}{H(P,Q)}-H(P,Q)\Big)^2
    &\leq \Var_P(X)+ \Var_Q(X),
    \label{eq.lem1_2_old}
\end{align}
without involving \cite[Theorem 1]{nishiyama2020tightHellinger}.
Using $H^2(P,Q)=1-\int \sqrt{pq},$ triangle inequality and Cauchy-Schwarz twice, we find
\begin{align*}
    &\big|E_P[X]-E_Q[X] \big|\\
    &= \Big|\int (X(\om)-E_P[X]) \sqrt{p(\om)} (\sqrt{p(\om)}-\sqrt{q(\om)}) \, d\nu(\om)\\
    &\quad + \int (X(\om)-E_Q[X]) \sqrt{q(\om)} (\sqrt{p(\om)}-\sqrt{q(\om)}) \, d\nu(\om)\\
    &\quad + \big(E_P[X]-E_Q[X]\big) H^2(P,Q)  \Big|\\
    &\leq 
    \Big(\Var_P(X)^{1/2}+\Var_Q(X)^{1/2}\Big)\sqrt{2} H(P,Q)+ 
    \big|E_P[X]-E_Q[X] \big|H^2(P,Q).
\end{align*}
Squaring, rearranging the terms and using that for any positive real numbers $u,v,$ $(\sqrt{u}+\sqrt{v})^2\leq 2u+2v$ yields \eqref{eq.lem1_2_old}.

To prove \eqref{eq.lem1_3}, it is enough to consider the case that $K(P,Q)+K(Q,P)<\infty.$ This implies in particular that the Radon-Nikodym derivatives $dP/dQ$ and $dQ/dP$ both exist.  Set $h(t,\om) := \exp(t\log p +(1-t)\log q(\om)).$ Observe that $p(\om)-q(\om)=\int_0^1 \log(p(\om)/q(\om)) h(t,\om) \, dt.$ Due to the concavity of the logarithm, we also have that $h(t,\om)\leq t p(\om)+(1-t)q(\om).$ Choosing again $a_*:=(E_P[X]+E_Q[X])/2,$ and using $E_P[(X-a^*)^2]=\Var_P(X)+(E_P[X]-E_Q[X])^2/4$ and $E_Q[(X-a^*)^2]=\Var_Q(X)+(E_P[X]-E_Q[X])^2/4,$ we therefore have that
\begin{align*}
    &\int \big(X(\om)-a^*\big)^2 h(t,\om) d\om \\
    &\leq t\Var_P(X)+(1-t)\Var_Q(X)+\frac{(E_P[X]-E_Q[X])^2}4 \\
    &\leq \big(\Var_P(X)\vee \Var_Q(X)\big)+\frac{(E_P[X]-E_Q[X])^2}4.
\end{align*}
Also notice that
\begin{align*}
    \int \log^2 \Big(\frac{p(\om)}{q(\om)}\Big) \int_0^1 h(t,\om) \, dt \, d\om
    &= \int \log \Big(\frac{p(\om)}{q(\om)}\Big) \big(p(\om)-q(\om)\big) \, d\om \\
    &= \KL(P,Q)+\KL(Q,P).
\end{align*}
Changing the order of integration and applying the properties of the function $h(t,\om)$, the Cauchy-Schwarz inequality, and Jensen's inequality, we find
\begin{align*}
    \big|&E_P[X]-E_Q[X]\big| \\
    &=\Big| \int \big(X(\om)-a^*\big) \big(p(\om)-q(\om)\big) \, d\om \Big| \\
    &= \Big| 
    \int_0^1  \Big(\int \big(X(\om)-a^*\big) \sqrt{h(t,\om)} \log \Big(\frac{p(\om)}{q(\om)}\Big)\sqrt{h(t,\om)} \, d\om \, dt\Big| \\
    &\leq \int_0^1 \Big( \int \big(X(\om)-a^*\big)^2 h(t,\om) \, d\om \Big)^{1/2} \Big( \int \log^2 \Big(\frac{p(\om)}{q(\om)}\Big) h(t,\om) \, d\om \Big)^{1/2} \, dt \\
    &\leq \Big(\big(\Var_P(X) \vee \Var_Q(X)\big)
    + \frac{(E_P[X]-E_Q[X])^2}4\Big)^{1/2}  \\
    &\hspace{5cm} \times \Big(\int \log^2\Big(\frac{p(\om)}{q(\om)}\Big) \int_0^1 h(t,\om) \, dt \, d\om \Big)^{1/2} \\
    &=\Big(\big(\Var_P(X) \vee \Var_Q(X)\big)
    +\frac{(E_P[X]-E_Q[X])^2}4\Big)^{1/2} \\
    &\hspace{5cm} \times \Big(\KL(P,Q)+\KL(Q,P)\Big)^{1/2}.
\end{align*}
Squaring and rearranging the terms yields \eqref{eq.lem1_3}.

The proof for \eqref{eq.lem1_4} combines change of measure and the Cauchy-Schwarz inequality via
\begin{align*}
    \big|E_P[X]-E_Q[X]\big|
    =\big| E_P\Big[ \Big(\frac{dQ}{dP} -1\Big)
    \big(X- E_P[X]\big)\Big] \Big|
    \leq
    \sqrt{\chi^2(Q,P)\Var_P(X)}.
\end{align*}
Squaring and interchanging the role of $P$ and $Q$ completes the proof.
\end{proof}

\begin{proof}[Proof of Theorem \ref{thm.multiple_lb}] \textit{(i):} Let $a_j$ denote the $j$-th entry of the vector $\Delta^\top \chi^2(P_0,\dots,P_M)^+$ and $E_j$ the expectation $E_{P_j}.$ Observe that for any $j,k=1,\dots, M,$ $\chi^2(P_0,\dots,P_M)_{j,k}= \int (dP_j/dP_0) dP_k -1 = E_0[(dP_j/dP_0-1)(dP_k/dP_0-1)].$ Using the Cauchy-Schwarz inequality and the fact that for a Moore-Penrose inverse $A^+$ of $A,$ $A^+AA^+=A^+$, we find
\begin{align*}
    \Big(\Delta^\top &\chi^2(P_0,\dots,P_M)^+ \Delta \Big)^2 \\
    &=\Big(\sum_{j=1}^M a_j \big(E_j[X]-E_0[X]\big)\Big)^2 \\
    &= E_0^2\Big[ 
    \sum_{j=1}^M a_j
    \Big(\frac{dP_j}{dP_0}-1\Big) \big(X-E_0[X]\big)\Big] \\
    &\leq E_0\bigg[ \Big(\sum_{j=1}^M a_j
    \Big(\frac{dP_j}{dP_0}-1\Big) \Big)^2\bigg] 
    \Var_{P_0}(X) \\
    &=\bigg(\sum_{j,k=1}^M a_j \chi^2(P_0,\dots,P_M)_{j,k}a_k \bigg) \Var_{P_0}(X) \\
    &= \Delta^\top \chi^2(P_0,\dots,P_M)^+ \chi^2(P_0,\dots,P_M) \chi^2(P_0,\dots,P_M)^+ \Delta \Var_{P_0}(X) \\
    &= \Delta^\top \chi^2(P_0,\dots,P_M)^+ \Delta \Var_{P_0}(X).
\end{align*}
For $\Delta^\top \chi^2(P_0,\dots,P_M)^+\Delta=0,$ the asserted inequality is trivially true. For $\Delta^\top \chi^2(P_0,\dots,P_M)^+ \Delta>0$ the claim (i) follows by dividing both sides by $\Delta^\top \chi^2(P_0,\dots,P_M)^+ \Delta.$

\textit{(ii):} The first identity is elementary and follows from expansion of the squares. It therefore remains to prove the inequality. To keep the mathematical expressions readable, we agree to write $E_j:=E_{P_j}[X]$ and $V_j:=\Var_{P_j}(X).$ Furthermore, we omit the integration variable as well as the differential in the integrals. Rewriting, we find that for any real number $\alpha_{j,k},$
\begin{align}
\begin{split}
    \big(E_j-E_k\big) \int \sqrt{p_k p_j}
    &= \int (X-E_k)\sqrt{p_k}\big(\sqrt{p_j}-\alpha_{j,k} \sqrt{p_k}\big) \\
    &+ \int (X-E_j)\sqrt{p_j}\big( \alpha_{k,j} \sqrt{p_j}- \sqrt{p_k}\big).
\end{split}\label{eq.H_mult_1}    
\end{align}
From now on, we choose $\alpha_{j,k}$ to be $\int \sqrt{p_j p_k}.$ Observe that for this choice, the term $\alpha_{j,k} \sqrt{p_k}$ is the $L^2$-projection of $\sqrt{p_j}$ on $\sqrt{p_k}.$ Dividing by $\int \sqrt{p_kp_j},$ summing over $(j,k),$ interchanging the role of $j$ and $k$ for the second term in the first equality, applying the Cauchy-Schwarz inequality first to each of the $M$ integrals and then also to bound the sum over $k,$ and using Proposition 3.2(v) in \cite{DerumignySchmidtHieber2023},
gives
\begin{align*}
    &\sum_{j,k=1}^M \big(E_j-E_k\big)^2 \\
    &= \sum_{k=1}^M 
    \int (X-E_k)\sqrt{p_k} \sum_{j=1}^M 
    \Big(\frac{\sqrt{p_j}}{\int \sqrt{p_k p_j}}- \sqrt{p_k}\Big)(E_j-E_k) \\
    &\quad + \sum_{j=1}^M \int (X-E_j)\sqrt{p_j}
    \sum_{k=1}^M\Big(\sqrt{p_j}-\frac{\sqrt{p_k}}{\int \sqrt{p_k p_j}}\Big)(E_j-E_k) \\
    &= 2\sum_{k=1}^M 
    \int (X-E_k)\sqrt{p_k} \sum_{j=1}^M 
    \Big(\frac{\sqrt{p_j}}{\int \sqrt{p_k p_j}}- \sqrt{p_k}\Big) (E_j-E_k) \\
    &\leq 2\sum_{k=1}^M 
    \sqrt{V_k \int \Big(\sum_{j=1}^M 
    \Big(\frac{\sqrt{p_j}}{\int \sqrt{p_k p_j}}- \sqrt{p_k}\Big) (E_j-E_k)\Big)^2} \\
    &\leq 2
    \sqrt{\sum_{r=1}^M  V_r} \sqrt{\sum_{k=1}^M \int \Big(\sum_{j=1}^M 
    \Big(\frac{\sqrt{p_j}}{\int \sqrt{p_k p_j}}- \sqrt{p_k}\Big) (E_j-E_k)\Big)^2 } \\
    &\leq 2
    \sqrt{\sum_{r=1}^M  V_r} \sqrt{\sum_{k=1}^M \lambda_1(A_k) \sum_{j=1}^M (E_j-E_k)^2 } \\
    &\leq 2  \big(\max_{\ell=1,\ldots,M} \lambda_1(A_\ell)\big)^{1/2}
    \sqrt{\sum_{r=1}^M  V_r}\sqrt{\sum_{k,j=1}^M (E_j-E_k)^2 }.
\end{align*}
Squaring both sides and dividing by $\sum_{k,j=1}^M (E_j-E_k)^2$ yields the claim.
\end{proof}

\subsection{Examples for the case of a Gaussian distribution and a general lower bound based on a family of distribution}

\begin{example} \label{ex.bds_for_normal}  
    To illustrate the inequalities in Lemma \ref{lem.general_lb} for a specific example, consider multivariate normal distributions $P=\Nc(\theta,I)$ and $Q=\Nc(\theta',I),$ for vectors $\theta, \theta'$ and $I$ the identity matrix. In this case, closed-form expressions for all four information measures exist. Denote by $\Phi$ the c.d.f. of the normal distribution. Since 
    $\TV(P,Q)=1-P(dQ/dP>1)-Q(dP/dQ\geq 1)
    = 1-2\Phi(-\tfrac 12 \|\theta-\theta'\|_2),$
    $H^2(P,Q)=1-\exp(-\tfrac 18 \|\theta-\theta'\|_2^2),$ 
    $\KL(P,Q)=\KL(Q,P)=\tfrac 12 \|\theta-\theta'\|_2^2,$
    and $\chi^2(P,Q)=\exp(\|\theta-\theta'\|_2^2)-1,$ the inequalities \eqref{eq.lem1_1}-\eqref{eq.lem1_4} become
    \begin{align*}
    % \begin{split}
        \big(E_\theta[X]-E_{\theta'}[X]\big)^2 
        \frac{\Phi(-\tfrac 12 \|\theta-\theta'\|_2)}{1-2\Phi(-\tfrac 12 \|\theta-\theta'\|_2)}
        &\leq \Var_\theta(X)+\Var_{\theta'}(X) \\
        \big(E_\theta[X]-E_{\theta'}[X]\big)^2
        \frac{\tfrac 14 \exp(-\tfrac 14 \|\theta-\theta'\|_2^2)}{1-\exp(-\tfrac 18 \|\theta-\theta'\|_2^2)}
        &\leq \Var_\theta(X)+\Var_{\theta'}(X) \\
        \big(E_\theta[X]-E_{\theta'}[X]\big)^2
        \bigg(\frac1{\|\theta-\theta'\|_2^2}-\frac 14\bigg)
        &\leq \Var_\theta(X)+\Var_{\theta'}(X) \\
        \big(E_\theta[X]-E_{\theta'}[X]\big)^2
        % &\leq \big(\exp\big(\|\theta-\theta'\|_2^2\big)-1\big)\big(\Var_\theta(X)\wedge \Var_{\theta'}(X)\big).
        \leq \big(e^{\|\theta-\theta'\|_2^2}-1\big)
        &\big(\Var_\theta(X)\wedge \Var_{\theta'}(X)\big).
    % \end{split}
    % \label{eq.bds_for_normal}    
    \end{align*}
\end{example}

\begin{lemma} \label{lem.lb_limit}
    Given a family of probability measures $(P_t)_{t\in [0,1]}.$ For simplicity write $E_t$ and $\Var_t$ for $E_{P_t}$ and $\Var_{P_t},$ respectively.  
    \newline
    \noindent 
    {(i):} If $\kappa_H := \limsup_{\delta \rightarrow 0} \,  \delta^{-1} \sup_{t\in [0,1-\delta]} H(P_t,P_{t+\delta})$ is finite, then for any random variable $X$,
    \begin{align}
        \big(E_1[X]-E_0[X]\big)^2 
        \leq 8\kappa_H^2 \sup_{t\in [0,1]} \Var_t(X).
    \end{align}
    {(ii):} If $\kappa_K^2 := \limsup_{\delta \rightarrow 0} \,  \delta^{-2} \sup_{t\in [0,1-\delta]} \KL(P_t,P_{t+\delta})+\KL(P_{t+\delta},P_t)$ is finite, then for any random variable~$X,$
    \begin{align}
        \big(E_1[X]-E_0[X]\big)^2 
        \leq \kappa_K^2 \sup_{t\in [0,1]} \Var_t(X).
        \label{eq.lb_limit_KL}
    \end{align}
    {(iii):} If $\kappa_{\chi}^2:= \limsup_{\delta \rightarrow 0} \,  \delta^{-2}\sup_{t\in [0,1-\delta]} \chi^2(P_t,P_{t+\delta})$ is finite, then for any random variable $X,$
    \begin{align}
        \big(E_1[X]-E_0[X]\big)^2 
        \leq \kappa_\chi^2 \sup_{t\in [0,1]} \Var_t(X).
    \end{align}
\end{lemma}

\begin{example}
\label{ex.normal_change_of_expec_cts}
    As an example, consider the family $P_t=\Nc(t\theta+(1-t)\theta',I)$ $t\in [0,1].$ Then, $(i)-(iii)$ all lead to the inequality $$(E_\theta[X]-E_{\theta'}[X])^2\leq \|\theta-\theta'\|_2^2\sup_{t\in [0,1]}\Var_t(X).$$ 
    In Example~\ref{ex.bds_for_normal}, the bounds for the Hellinger distance and the $\chi^2$-divergence grow exponentially in $\|\theta-\theta'\|_2^2$ and the Kullback-Leibler bound only provides a non-trivial lower bound if $\|\theta-\theta'\|_2^2< 4.$ Lemma \ref{lem.lb_limit} leads thus to much sharper constants if $\|\theta-\theta'\|_2$ is large. On the other hand, compared to the earlier bounds, Lemma \ref{lem.lb_limit} results in a weaker statement on the bias-variance trade-off as it only produces a lower bound for the largest of all variances $\Var_t(X),$ $t\in [0,1].$
\end{example}

\begin{proof}[Proof of Lemma \ref{lem.lb_limit}]
Rewriting $E_1[X]-E_0[X]$ as the telescoping sum $\sum_{j=1}^K E_{j/K}[X]-E_{(j-1)/K}[X]$ and taking the limit $K\to \infty$ over a subset converging to the $\limsup$, we find that $$(E_1[X]-E_0[X])^2 \leq \limsup_{K \to \infty} K^2 \max_{j=1,\ldots,K} (E_{j/K}[X]-E_{(j-1)/K}[X])^2.$$ Applying \eqref{eq.lem1_2}, \eqref{eq.lem1_3} and \eqref{eq.lem1_4} to $(E_{j/K}[X]-E_{(j-1)/K}[X])^2,$ bounding $\Var_{j/K}(X)$ and $\Var_{(j-1)/K}(X)$ always by $\sup_{t\in [0,1]}\Var_t(X),$ and taking the limit $K\to \infty$ yields the three inequalities.
\end{proof}

\subsection{The univariate Cram\'er-Rao lower bound as a limit of the change of expectation inequalities in Lemma \ref{lem.general_lb}}
\label{sec:univarCRLB}

\begin{thm}[Univariate Cramér-Rao lower bound]
    Let $(P_\theta, \theta \in \Theta)$ be a one dimensional statistical model, with $\Theta$ an open subset of $\Rb$, assumed to be dominated by a measure $\nu$. If $\theta \mapsto E_\theta[\wh \theta]$ is differentiable at $\theta_0,$ the function $\theta \mapsto \Var_\theta[\wh \theta]$ is continuous at $\theta_0$ and one of the following domination conditions is satisfied
    \begin{enumerate}
        \item \textbf{Hellinger domination}: there exists a $\nu$-integrable function $\bar p$ such that for all $h$ small enough $h^{-2} \big(\sqrt{p_{\theta_0}} - \sqrt{p_{\theta_0+h}}\big)^2 \leq \bar p$;
        \item \textbf{KL domination}: there exists a $\nu$-integrable function $\bar p$ such that for all $h$ small enough
        $h^{-2} \big| \log ( p_{\theta_0+h}/p_{\theta_0}) \big|
        |p_{\theta_0+h} - p_{\theta_0}| \leq \bar p$;
        \item \textbf{$\chi^2$ domination}: there exists a $\nu$-integrable function $\bar p$ such that for all $h$ small enough $h^{-2}
        (p_{{\theta_0}+h} / p_{\theta_0} - 1 )^2 \, p_{\theta_0}
        \leq \bar p$;
    \end{enumerate}
    then the Cramér-Rao lower bound
    \begin{align*}
        \Var_{\theta_0}\big(\wh \theta\big) \geq
        \frac{\big( \partial E_{\theta}[\wh \theta] / \partial \theta \big)^2 |_{\theta=\theta_0}}{F(\theta_0)} 
    \end{align*}
    holds, where $F(\theta_0)$ denotes the Fisher information at $\theta_0.$
\end{thm}
\begin{proof}[Proof: Hellinger version]
    If $h \neq 0$ and $h \to 0$, then by the dominated convergence theorem,
    \begin{align*}
        \frac{H^2(P_{\theta_0}, P_{{\theta_0}+h})}{h^2}
        &= \frac{1}{2} \int
        \frac{\big( \sqrt{p_{\theta_0}(\omega)} - \sqrt{p_{{\theta_0}+h}(\omega)} \big)^2}{h^2}
         \, d\nu(\omega) \\
        &= \frac{1}{2} \int
        \frac{\big( (p_{\theta_0}(\omega) - p_{{\theta_0}+h}(\omega))/h \big)^2}{\big(\sqrt{p_{\theta_0}(\omega)} + \sqrt{p_{{\theta_0}+h}(\omega)}\big)^2} \, d\nu(\omega) \\
        &\to \frac{1}{2} \int
        \frac{\big( \partial p_{\theta_0}(\omega) / \partial {\theta_0} \big)^2}{4 p_{\theta_0}(\omega)^2} p_{\theta_0}(\omega) \, d\nu(\omega) \\
        &= \frac{F({\theta_0})}8.
    \end{align*}
    The definition of derivative gives $(E_{\theta+h}[\wh \theta] - E_\theta[\wh \theta])/h \to \partial E_\theta[\wh \theta] / \partial \theta.$ Using the change of expectation inequality for the Hellinger distance in \eqref{eq.lem1_2}, we get for $h \to 0,$
    \begin{align*}
        &\frac{(E_{{\theta_0}+h}[\wh \theta] - E_{\theta_0}[\wh \theta])^2}{4-hF(\theta_0)/4+o(h)}
        \Big( \frac{1}{h (F({\theta_0}) / 8 + o(1))^{1/2}}
        - h (F({\theta_0}) / 8 + o(1))^{1/2} \Big)^2 \\
        &\leq \Var_{\theta_0}(\wh \theta) + \Var_{{\theta_0}+h}(\wh \theta)
        \to 2\Var_{\theta_0}(\wh \theta).
    \end{align*}
    By letting $h \to 0$ on the right hand side of the inequality, we obtain the claimed inequality.
\end{proof}
\begin{proof}[Kullback-Leibler version]
    If $h \neq 0$ and $h \to 0$, then by the dominated convergence theorem,
    \begin{align*}
        &h^{-2} \big( \KL(P_{\theta_0}, P_{{\theta_0}+h})
        + \KL(P_{{\theta_0}+h}, P_{\theta_0}) \big) \\
        &= h^{-2} \int \log \bigg( \frac{p_{{\theta_0}+h}}{p_{\theta_0}} \bigg) (p_{{\theta_0}+h} - p_{\theta_0}) \, d\nu \\
        &= h^{-2} \int \log \bigg( 1 + \frac{p_{{\theta_0}+h} - p_{\theta_0}}{p_{\theta_0}} \bigg) \frac{p_{{\theta_0}+h} - p_{\theta_0}}{p_{\theta_0}} p_{\theta_0} \, d\nu \\
        &\to F({\theta_0}).
    \end{align*}
    Using the change of expectation inequality for the Kullback-Leibler divergence in~\eqref{eq.lem1_3}, we get
    \begin{align*}
        (E_{{\theta_0}+h}[\wh \theta] - E_{\theta_0}[\wh \theta])^2
        \Big( \frac{1}{h^2 (F({\theta_0})+o(1))} - \frac{1}{4}\Big)
        &\leq \Var_{\theta_0}(X) \vee \Var_{{\theta_0}+h}(X).
    \end{align*}
    By letting $h \to 0$, we obtain the claimed inequality.
\end{proof}
\begin{proof}[$\chi^2$ version]
    If $h \neq 0$ and $h \to 0$, then by the dominated convergence theorem,
    \begin{align*}
        \frac{\chi^2(P_{{\theta_0}+h} , P_{\theta_0})}{h^2}
        &= h^{-2} \int \bigg( \frac{p_{{\theta_0}+h}}{p_{\theta_0}} - 1 \bigg)^2 \,
        p_{\theta_0} \, d\nu \\
        &= h^{-2} \int \bigg( \frac{p_{{\theta_0}+h} - p_{\theta_0}}{p_{\theta_0}} \bigg)^2 \,
        p_{\theta_0}\,  d\nu
        \to F({\theta_0}).
    \end{align*}
    Using the change of expectation inequality for the $\chi^2$- divergence in~\eqref{eq.lem1_4}, we get
    \begin{align*}
        &\big(E_{{\theta_0}+h}[\wh \theta] - E_{\theta_0}[\wh \theta]\big)^2 \\
        &\leq h^2 (F({\theta_0})+o(1)) \Var_{\theta_0}(X)
        \wedge h^2 (F({\theta_0})+o(1)) \Var_{\theta_0+h}(X).
    \end{align*}
    Dividing by $h^2$ and letting $h \to 0$, we obtain the claimed inequality.
\end{proof}

% As this is not the main focus of the work, we give an informal argument without stating the exact regularity conditions. In the above inequalities, take $P$ and $Q$ to be $P_\theta$ and $P_{\theta+\Delta}$ and let $\Delta$ tend to zero. Recall that $B'(\theta)$ is the derivative of the bias at $\theta$ and $F(\theta)$ denotes the Fisher information. For any estimator $\wh \theta,$ we have for small $\Delta,$ $(E_{P_\theta}[\wh \theta]-E_{P_{\theta+\Delta}}[\wh \theta])^2 \approx \Delta^2 (1+B'(\theta))^2.$ Using that $(\sqrt{x}-\sqrt{y})^2=(x-y)^2/(\sqrt{x}+\sqrt{y})^2$ shows that $H^2(P_\theta,P_{\theta+\Delta})\approx \Delta^2 F(\theta)/8.$ Moreover, $\KL(P,Q)+\KL(Q,P)=\int \log (p/q) (p-q)$ and a first order Taylor expansion of $\log(x)$ shows that $\KL(P_\theta,P_{\theta+\Delta})+\KL(P_{\theta+\Delta},P_\theta)\approx \Delta^2 F(\theta).$ Similarly, we find $\chi^2(P_{\theta+\Delta},P_\theta)\approx \Delta^2 F(\theta).$ Replacing the divergences by their approximations then shows that for the Hellinger, Kullback-Leibler and $\chi^2$ versions, the Cram\'er-Rao lower bound can be retrieved taking the limit $\Delta \rightarrow 0$. 

\subsection{The multivariate Cram\'er-Rao lower bound as a limit of the change of expectation inequalities in Theorem \ref{thm.multiple_lb}(i)}~
\label{sec:multCRLB}

\medskip

We derive the general version of the multivariate Cramér-Rao lower bound (see e.g. \cite{rao1973linear}, p.326) as a limit of the change of expectation inequality established in Theorem~\ref{thm.multiple_lb}(i).
Let $(P_\theta, \theta \in \Theta)$ be a statistical model, with parameter space $\Theta$ an open subset of $\Rb^p$ and assuming existence of densities $p_\theta(x)$ with respect to a given dominating measure $\nu$.
For $\theta_0\in \Theta$ and a positive integer $r,$ let $f_1, \dots, f_r$ be $r$ statistics with finite variance and, for $i,j = 1, \dots, r$, and $k,l = 1, \dots, p$, set $g_i(\theta) := E_\theta[f_i],$ $V_{i,j} := E_{\theta_0} [ (f_i - g_i(\theta_0))(f_j - g_j(\theta_0))]$ and 
\begin{align*}
    F_{k,\ell} := E_{\theta} \bigg[ \frac{\partial \log p_\theta}{\partial \theta_k} \frac{\partial \log p_\theta}{\partial \theta_\ell} \bigg] \, \Big|_{\theta=\theta_0}.
\end{align*}
For the parameter $\theta_0$, $V:=(V_{i,j})_{1\leq i,j \leq r}$ is the covariance matrix of $f = (f_1, \dots, f_r)$, 
$F(\theta_0) := F:= (F_{k,l})_{1\leq k,l \leq p}$ is the Fisher information and
$\Delta:=\Delta(\theta_0) := \Jac_{\theta_0}(g_1, \dots g_j)
= (\partial_{\theta_i} g_j(\theta))_{1\leq i,j \leq p}\,|_{\theta=\theta_0}$ is the Jacobian matrix.
\begin{thm}[Multivariate Cramér-Rao lower bound]
    Assume the following local domination condition
    \begin{align*}
        \forall i,j = 1, \dots, p, \, \exists h_0 > 0, \ \text{such that} \ \forall h \in [-h_0, h_0], \,
        \frac{p_{{\theta_0} + h e_i} p_{{\theta_0} + h e_j}}{p_{\theta_0}}
        \leq \bar p,
    \end{align*}
    for some integrable function $\bar p$, where $(e_i)_i$ denotes the canonical basis of $\Rb^p$.
    Then 
    \begin{align*}
        V \geq \Delta^\top F^+ \Delta,
    \end{align*}
    in the sense of the difference being positive semi-definite. Here, $F^+$ denotes the Moore-Penrose pseudo-inverse of the Fisher information $F$.
    
    As a consequence, for any estimator $\wh \theta$ with values in $\Rb^p$, we obtain
    \begin{align*}
        \Cov_{\theta_0}(\wh \theta) \geq
        \Jac_{\theta_0}(E_{\theta_0}[\wh \theta])^\top F({\theta_0})^+
        \Jac_{\theta_0}(E_{\theta_0}[\wh \theta]).
    \end{align*}
\end{thm}

For the proof, we will use the following notation.
For any matrix $M$, and any two index sets $S,S'$, $[M]_{S,S'}$ denotes the submatrix of $M$ obtained by keeping the rows in $S$ and the columns in $S'$. $0$ denotes a matrix full of zeros and $I$ the identity matrix.

\begin{proof}
    Write $P_0 := P_{\theta_0}$ and $P_{i} = P_{\theta_0 + h e_i}$ for $i=1, \dots, p$ with $(e_i)_i$ the canonical basis of $\Rb^p$ and $h \in \Rb$.
    
    \textit{1. Reduction to proving positive determinant.}
    To show that $V - \Delta^\top F^+ \Delta$ is positive semi-definite, it is sufficient to prove that for each subset $S \subset \{1, \dots, p \}$,
    $\Det([V - \Delta^\top F^+ \Delta]_{S,S}) \geq 0$ (see \cite{meyer2000matrix}, Equation (7.6.12), p.566). As we can permute the statistics $f_1, \dots, f_r$, it is even sufficient to prove the inequality for $S = \{1, \dots, s\}$ for some $1 \leq s \leq r$. In this case, we have
    \begin{align*}
        [V - \Delta^\top F^+ \Delta]_{S,S}
        = [V]_{S,S} - 
        [\Delta]_{1:p,S}^\top F^+ [\Delta]_{1:p,S}.
    \end{align*}
    This expression is exactly the equivalent of $V - \Delta^\top F^+ \Delta$ with $r$ replaced by $s$.
    Therefore, it is sufficient to show that for any $r \geq 1$ and any statistics $f_1, \dots, f_r$, 
    $\Det(V - \Delta^\top F^+ \Delta) \geq 0$.

    \textit{2. Reduction to a non-singular and diagonal Fisher matrix.}
    If $F$ is the null matrix, then $F^+=0$ and the proof is completed. Indeed, since $V$ is a covariance matrix it is positive semi-definite and consequently $\Det(V) \geq 0.$
    
    Thus, we may assume that $F$ is not the zero matrix. Denote the rank of $F$ by $q.$ As $F$ is positive semi-definite, $F$ is of the form $F=U^\top D U$ where $U$ is orthogonal and $D$ is diagonal, the first $q$ diagonal components are positive and the remaining ones are zero. Write $U = (U_{i,j})$ for $1 \leq i,j \leq p$ and define $\v_i := \sum_{k=1}^p e_k U_{i,k}$. Our goal is to show that everything can be written locally in terms of the vectors $\v_i,$ $i=1, \dots, q$ (instead of $p$). By assumption, the Fisher information matrix in this ``diagonalized'' statistical model is invertible as we keep only the non-zero eigenvalues. Observe that
    \begin{align*}
        [U \Delta]_{i,j} &= \sum_{k=1}^p U_{i,k} \frac{\partial_\theta E_\theta[f_j]}{\partial \theta_k} \, \Big|_{\theta=\theta_0} \\
        &= \partial_t E_{\theta_0 + t \sum_{k=1}^p e_k U_{i,k}}[f_j] \, \big|_{t=0} \\
        &= \partial_t E_{\theta_0 + t \v_i}[f_j] \, \big|_{t=0},
    \end{align*}
    and 
    \begin{align*}
        [U F U^\top]_{i,j} &= \sum_{k,\ell=1}^p U_{i,k} F_{k,\ell} U_{j,\ell}
        = \sum_{k,\ell=1}^p U_{i,k} U_{j,\ell} 
        E_\theta \left[ \frac{\partial \log p_\theta}{\partial \theta_k}
        \frac{\partial \log p_\theta}{\partial \theta_\ell} \right] \, \Big|_{\theta=\theta_0} \\
        &= E_{\theta_0} \left[ \frac{\partial \log p_{\theta_0+\sum_{k=1}^p t_k \v_k}}{\partial t_i}
        \frac{\partial \log p_{\theta_0+\sum_{\ell=1}^p t_\ell \v_\ell}}{\partial t_j} \right] \, \Bigg|_{t_1=\ldots=t_k=0}.
    \end{align*}
    Consider now the statistical model
    $\Mk := (P_{\theta_0+\sum_{k=1}^q t_k \v_k}, (t_1,\dots, t_q) \in [-\eps, \eps]^q)$ for $\eps > 0$ small enough such that $\theta_0+\sum_{k=1}^q t_k \v_k \in \Theta$ for all $(t_1,\dots, t_q) \in [-\eps, \eps]^q$. Such a choice is possible as $\theta_0$ lies in the open set $\Theta.$ Therefore,
    \begin{align*}
        \Det(V - \Delta^\top F^+ \Delta)
        % = \Det(V - \Delta^\top U^\top D^+ U \Delta)
        = \Det \big(V - (U \Delta)^\top D^+ (U \Delta) \big)
        = \Det \big(V - \wt \Delta^\top \wt F^{-1} \wt \Delta \big),
    \end{align*}
    with $\wt \Delta = [U \Delta]_{1:p,1:q}
    = \Jac_{(t_1,\dots, t_q)} \big(E_{\theta_0+\sum_{k=1}^q t_k \v_k}[(f_1, \dots, f_r)] \big) \, \big|_{t_1=\ldots=t_k=0}$
    is the Jacobian matrix of the expectation of the statistic $(f_1, \dots, f_r)$ in $\Mk$ (with respect to the new model parameters $t_1, \dots, t_q$) evaluated at $t_1=\ldots=t_k=0$,
    and $\wt F = [U F U^\top]_{1:q,1:q} = [D]_{1:q,1:q}$ is the Fisher information matrix of $\Mk$. Note that $\wt F$ is non-singular and diagonal. Together with the last display, this shows that $\Det(V - \Delta^\top F^+ \Delta) \geq 0$ holds if and only if the inequality $\Det \big(V - \wt \Delta^\top \wt F^{-1} \wt \Delta \big) \geq 0$ holds. Therefore, it is sufficient to show that $\Det(V - \Delta^\top F^+ \Delta) \geq 0$ holds for a non-singular and diagonal Fisher matrix $F$.

    \textit{3. Applying change of expectation inequality.}
    We show in the next step that $\Det(V - \Delta^\top F^+ \Delta) \geq 0$ follows as a limit of the change of expectation inequality $V \geq \Delta_h^\top
    \big( \chi^2(P_0,\dots,P_M) / h^2 \big)^+ \Delta_h$ in Theorem \ref{thm.multiple_lb} (i). By Lebesgue's dominated convergence theorem, we obtain
    \begin{align}
    \begin{split}
        [\chi^2(P_0, \dots, P_M) / h^2]_{i,j}
        &= \int \frac{p_{\theta_0 + h e_i}}{p_{\theta_0}}
        \frac{p_{\theta_0 + h e_j}}{p_{\theta_0}} p_{\theta_0} d\nu - h^{-2}   \\
        &= \int \Bigg[ \bigg( \frac{p_{\theta_0 + h e_i} - p_{\theta_0}}{h p_{\theta_0}} \bigg)
        \bigg( \frac{p_{\theta_0 + h e_j} - p_{\theta_0}}{h p_{\theta_0}} \bigg)   \\
        &\quad \quad+ \frac{p_{\theta_0 + h e_i} - p_{\theta_0} + p_{\theta + h e_j} - p_\theta}{h^2 p_{\theta_0}} \Bigg] p_{\theta_0} d\nu  \\
        &= \int
        \frac{ \partial p_{\theta} / \partial \theta_i}{p_{\theta_0}}
        \frac{ \partial p_{\theta} / \partial \theta_j}{p_{\theta_0}} p_{\theta_0} d\nu \bigg|_{\theta=\theta_0}+ o(1) \\
        &= F_{i,j} + o(1),
        \end{split}
        \label{eq:cv_chi2_FisherInformation}
    \end{align}
    as $h \to 0$. By construction, $F$ is non-singular. Since the set of non-singular matrices is open, $\chi^2(P_0, \dots, P_M) / h^2$ must be non-singular for all sufficiently small $h$. In particular, \eqref{eq:cv_chi2_FisherInformation} implies that $\Det(\chi^2(P_0, \dots, P_M) / h^2) \to \Det(F^{-1})$ and that the adjugate of $\chi^2(P_0, \dots, P_M) / h^2$ converges entrywise to the adjugate of $F$. Since the inverse of a matrix is the same as dividing all entries of the adjugate by the determinant, $\Det(F^{-1})>0$ also guarantees that the inverse $(\chi^2(P_0, \dots, P_M) / h^2)^{-1}$ converges entrywise to the inverse $F^{-1}.$ Up to rescaling by $h,$ the matrix $\Delta_h := h^{-1} (E_{\theta_0 + h e_i}[f_j] - E_{\theta_0}[f_j])_{1\leq i \leq p, 1 \leq j \leq r}$ can be viewed as a discretized version of the Jacobian matrix $\Delta$, that is, $\Delta_h \to \Delta$ as $h \to 0.$ This proves that $\Det (V - \Delta_h^\top \big( \chi^2(P_0,\dots,P_M) / h^2 \big)^{-1} \Delta_h) \to \Det(V-\Delta^\top F^{-1}\Delta)$ for $h \to 0.$ Since by Theorem \ref{thm.multiple_lb} (i), $V \geq \Delta_h^\top
    \big( \chi^2(P_0,\dots,P_M) / h^2 \big)^+ \Delta_h$ we conclude that $\Det(V-\Delta^\top F^{-1}\Delta)\geq 0.$ This completes the proof.

\end{proof}

\section{Proofs for Section \ref{sec:gauss_wn}}
\label{sec.proofs_gauss_wn}

\begin{lemma}\label{lem.kernel_norm_bd}
For $0 < h \leq 1,$
\begin{align*}
    \Big\|h^\beta
    K \Big( \frac{ \cdot - x_0 }{h} \Big)\Big\|_{\Cc^\beta}
    \leq \|K\|_{\Cc^\beta(\Rb)}.
\end{align*}
\end{lemma}

\begin{proof}
Set $f(x) := h^\beta K((x- x_0)/h).$ Then,
\begin{align*}
    \| f \|_{\Cc^\beta}
    &= \sum_{\ell \leq \floorbeta} \big\|f^{(\ell)}\big\|_\infty
    + \sup_{x,y \in [0,1]}
    \frac{|f^{(\floorbeta)}(x) - f^{(\floorbeta)}(y)|}{|x-y|^{\beta - \floorbeta}} \\
    &= \sum_{\ell \leq \floorbeta} h^{\beta-\ell} \big\|K^{(\ell)}\big\|_\infty \\
    &\qquad + \sup_{x,y \in [0,1]}
    \frac{|K^{(\floorbeta)}((x-x_0)/h) - K^{(\floorbeta)}((y-x_0)/h)|}{|(x-y)/h|^{\beta -\floorbeta}} \\
    &\leq\|K\|_{\Cc^\beta(\Rb)}.
\end{align*}
\end{proof}

\subsection{Derivations of (\ref{eq.low_derived_LB})}
\label{subsec:computation_constants_Low}

We have
\begin{align} \label{eq:Low_bound_var}
    \inf_{\wh{Lf}: \, \sup_{f\in \Theta} \Var_f(\wh{Lf})\leq V}
    \, \sup_{f\in \Theta} \, \Bias_f(\wh{Lf})^2
    = \frac 14 \sup_{\eps >0} \big(w(\eps)-\sqrt{nV}\eps \big)_+^2
\end{align}
and 
\begin{align} \label{eq:Low_bound_bias}
    \inf_{\wh{Lf}: \, \sup_{f\in \Theta} |\Bias_f(\wh{Lf})|\leq B}
    \, \sup_{f\in \Theta} \, \Var_f(\wh{Lf})
    = \frac 1n \sup_{\eps >0} \eps^{-2} \big(w(\eps)-2B \big)_+^2.
\end{align}
We now show that \eqref{eq:Low_bound_var} implies 
\begin{align}\label{eq:Low_bound_var_2}
    \inf_{\wh f(x_0) : \,
    \sup_{f\in \Cc^\beta(R)} \Var_f(\wh f(x_0))\leq V}
    \, \sup_{f\in \Cc^\beta(R)} \, \Bias_f \big(\wh f(x_0) \big)^2
    \geq \Big(\frac{\gamma_{\text{Low}}(R,\beta)}{nV}\Big)^{2\beta},
\end{align}
and that \eqref{eq:Low_bound_bias} implies
\begin{align}\label{eq:Low_bound_bias_2}
    \inf_{\wh f(x_0): \, \sup_{f\in \Cc^\beta(R)} |\Bias_f(\wh f(x_0))|\leq B}
    \, \sup_{f\in \Cc^\beta(R)} \, \Var_f\big(\wh f(x_0)\big)
    &\geq \frac{\gamma_{\text{Low}}(R,\beta)}{nB^{1/\beta}}.
\end{align}
We already showed in Section \ref{sec:gauss_wn} that, for the functional $Lf=f(x_0)$ and for any $K\in \Cc^\beta(\Rb),$ $w(\eps) \geq (\eps/\|K\|_2)^{\beta/(\beta+1/2)}K(0)=\eps^{2\beta/(2\beta+1)} C_1$ with $C_1 := K(0)/\|K\|_2^{2\beta/(2\beta+1)}.$

To see that \eqref{eq:Low_bound_var} implies \eqref{eq:Low_bound_var_2}, observe that
\begin{align*}
    \inf_{\wh{Lf}: \,
    \sup_{f\in \Theta} \Var_f(\wh{Lf})\leq V}
    \, \sup_{f\in \Theta} \, \Bias_f(\wh{Lf})^2
    &\geq \frac{1}{4} \sup_{\eps > 0}
    \big( \eps^{2\beta/(2\beta+1)} C_1
    - \sqrt{nV}\eps \big)_+^2 \\
    &\geq \frac{C_1^2}4 \sup_{\eps > 0}
    \big( \eps^{2\beta/(2\beta+1)} - \sqrt{nV} \eps/ C_1 \big)_+^2.
\end{align*}
For $\alpha<1,$ the function $x\mapsto (x^\alpha-xC)_+$ attains its maximum for $x=(C/\alpha)^{1/(\alpha-1)}$ and the maximal function value is 
\begin{align*}
    \Big(\frac{C}{\alpha}\Big)^{\frac{\alpha}{\alpha-1}}-\Big(\frac{C}{\alpha}\Big)^{\frac{1}{\alpha-1}} C
    = \Big(\frac{C}{\alpha}\Big)^{\frac{\alpha}{\alpha-1}}\big(1-\alpha\big).
\end{align*}
Thus, with $\alpha=2\beta/(2\beta+1)<1,$
\begin{align*}
    \sup_{\eps > 0} \big( \eps^{2\beta/(2\beta+1)} - \sqrt{nV} \eps/ C_1 \big)_+^2
    &= \bigg( \frac{(2 \beta + 1) \sqrt{nV}}{2\beta C_1} \bigg)^{-4\beta}
    \frac{1}{(2 \beta+1)^2}.
\end{align*}
Using the definition of $C_1$ gives
\begin{align*}
    &\inf_{\wh{Lf}: \,
    \sup_{f\in \Theta} \Var_f(\wh{Lf})\leq V}
    \, \sup_{f\in \Theta} \, \Bias_f(\wh{Lf})^2
    \\
    &\hspace{1cm} \geq \frac{1}{4} \frac{K(0)^2}{\|K\|_2^{4\beta/(2\beta+1)}}
    \bigg(\frac{(2 \beta + 1) \sqrt{nV} \|K\|_2^{2\beta/(2\beta+1)}}{2\beta K(0)} \bigg)^{-4\beta}
    \frac{1}{(2 \beta+1)^2}.
\end{align*}
Optimizing over the kernel $K$, the right hand side becomes $(\gamma_{\text{Low}}(R,\beta)/(nV))^{2 \beta},$ where
\begin{align*}
    \gamma_{\text{Low}}(R,\beta)
    := \sup_{K \in \Cc^\beta(\Rb): \, R \geq \|K\|_{\Cc^\beta(\Rb)}}
    \frac{(2\beta)^2}{2^{1/\beta} (2\beta + 1)^{2+1/\beta}}
    \frac{K(0)^{2+1/\beta}}{\|K\|_2^2}.
\end{align*}
This proves \eqref{eq:Low_bound_var_2}.

\medskip

We now show that \eqref{eq:Low_bound_bias} implies \eqref{eq:Low_bound_bias_2}. Substituting $x=\eps^{-1}$ and arguing as for the variance-constrained case (taking now $\alpha=1/(2\beta+1)$), Equation~\eqref{eq:Low_bound_bias} implies that
\begin{align*}
    \inf_{\wh{Lf}: \, \sup_{f\in \Theta} |\Bias_f(\wh{Lf})|\leq B}
    \, \sup_{f\in \Theta} \, \Var_f(\wh{Lf})
    &= \frac{C_1^2}n \sup_{\eps > 0} \eps^{-2}
    \Big( \eps^{2\beta / (2\beta+1)} - \frac{2B}{C_1} \Big)_+^2 \\
    &= \frac{C_1^2}n \sup_{x > 0} 
    \Big( x^{1/ (2\beta+1)} - x\frac{2B}{C_1} \Big)_+^2 \\
    &= \frac{C_1^2}n \Big( \frac{2B(2\beta+1)}{C_1}\Big)^{-\frac 1{\beta}}\Big(\frac{2\beta}{2\beta+1}\Big)^2 \\
    &= \frac{C_1^{\frac{2\beta}{2\beta+1}} (2\beta)^2}{nB^{1/\beta} 2^{1/\beta}(2\beta+1)^{\frac{2\beta+1}{\beta}}}.
\end{align*}
Using the definition of $C_1$ and optimizing over the kernel $K,$ the right hand side becomes $\gamma_{\text{Low}}(R,\beta)/(nB^{1/\beta}),$ proving \eqref{eq:Low_bound_bias_2}. 

Whenever we have an estimator $\wh f(x_0),$ we can choose $V=\sup_{f\in \Cc^\beta(R)}\Var_f(\wh f(x_0))$ and then obtain from \eqref{eq:Low_bound_var_2} that 
\begin{align*}
    \sup_{f\in \Cc^\beta(R)} \, \Big|\Bias_f \big(\wh f(x_0) \big)\Big|^{1/\beta} \sup_{f\in \Cc^\beta(R)} \Var_f(\wh f(x_0))
    \geq \frac{\gamma_{\text{Low}}(R,\beta)}{n}.
\end{align*}
This proves \eqref{eq.low_derived_LB}.

In the same way, one can derive \eqref{eq.low_derived_LB} also as a consequence of \eqref{eq:Low_bound_bias_2}. This means that both statements \eqref{eq:Low_bound_var} and \eqref{eq:Low_bound_bias} lead to the same bias-variance lower bound.

\subsection{Proof of Theorem \ref{thm.pointwise}}

\begin{proof}[Proof of Theorem \ref{thm.pointwise}]
{\it (i):} Given an estimator $\wh f,$ let us define $B:=\sup_{f\in \Cc^\beta(R)}|\Bias_f(\wh f(x_0))|.$ It is sufficient to show that for an arbitrary estimator with $B<1$ and any $K\in \Cc^\beta(\Rb) \cap L^2(\Rb)$ satisfying $K(0)=1,$
\begin{align}
    B^{1/\beta}
    \sup_{f\in \Cc^\beta(R)} \Var_f\big(\wh f(x_0)\big) 
    \geq 
    \frac 1n \|K\|_2^{-2} \bigg( 1 - \dfrac{\|K\|_{\Cc^\beta(\Rb)}}{R} \bigg)_+^2.
    \label{eq.pointwise_to_show_i}
\end{align}
We first assume that the worst-case bias bound is positive, that is, $B>0$. In a second part, we treat the case $B=0.$

Assuming $B>0,$ we begin by constructing a subspace $\Fc$ of the parameter space $\Fc \subseteq \Cc^\beta(R)$ parametrized by $\theta \in [-1, 1]$.
For $K \in \Cc^\beta(\Rb)$ any function satisfying $K(0) = 1$ and $\|K\|_2 < +\infty,$ define $V:=R/\|K\|_{\Cc^\beta(\Rb)}$ and
\begin{equation*}
    \Fc:=\Big\{f_\theta(x) =
    \theta
    V B K \Big( \frac{ x - x_0 }{B^{1/\beta}} \Big) : |\theta| \leq  1 \Big\}.
\end{equation*}
Using Lemma \ref{lem.kernel_norm_bd} and $B<1$, we have that $\|f_\theta\|_{\Cc^\beta([0,1])}\leq |\theta| V \|K\|_{\Cc^\beta(\Rb)} \leq R$ for all $\theta \in [-1,1]$. This implies $\Fc \subseteq \Cc^\beta(R)$.

As we want to apply our information inequalities, we need to control the Kullback-Leibler divergence between two elements of $\Fc$.
As explained at the beginning of Section \ref{sec:gauss_wn},  $\KL(P_f,P_g)=\tfrac n2 \|f-g\|_{L^2[0,1]}^2.$ We will apply Lemma \ref{lem.lb_limit} (ii) to the family of distributions $(P_{f_{\theta}})_{\theta \in [0,1]}$ and $(P_{f_{\theta}})_{\theta \in [-1,0]}.$ Due to
\begin{align}
    \KL\big(P_{f_{\theta}},P_{f_{\theta+\delta}}\big)
    &= \KL\big(P_{f_{\theta+\delta}},P_{f_{\theta}}\big)
    = \frac n2  \big\|f_{\theta}-f_{\theta +\delta} \big\|_{L^2[0,1]}^2 \notag \\
    &= \frac n2  \Big\|\delta V B K\Big( \frac{ x - x_0 }{B^{1/\beta}} \Big)  \Big\|_{L^2[0,1]}^2
    \leq \frac{n}{2} \delta^2 V^2 B^{2+1/\beta} \|K\|_{L^2(\Rb)}^2,
\label{eq.pointwise_L2_bd}   
\end{align}
the constant $\kappa_K^2$ in the statement of Lemma \ref{lem.lb_limit} (ii) is bounded by the quantity $n V^2 B^{2+1/\beta} \|K\|_{L^2(\Rb)}^2.$ Now, we apply the information inequality~\eqref{eq.lb_limit_KL} to the random variable $\wh f(x_0).$ This gives 
\begin{align*}
    \big(E_{f_{\pm 1}}\big[\wh f(x_0)\big]-E_{f_0}\big[\wh f(x_0)\big]\big)^2
    \leq n V^2 B^{2+1/\beta} \|K\|_{L^2(\Rb)}^2
    \sup_{|\theta|\leq 1} \Var_{f_{\theta}}\big(\wh f(x_0)\big),
\end{align*}
where $E_{f_{\pm 1}}$ stand for either $E_{f_1}$ or $E_{f_{-1}}.$

\begin{figure}[tb]
    % \centering
    % \resizebox{!}{8cm}{
    \begin{tikzpicture}
        \begin{axis}[
        width=14cm,height=9cm,
        axis lines=middle,
        xlabel=$\theta$, ylabel={True value $f_\theta(x_0)=\theta \times VB$},
        xlabel style = {anchor=north east},
        ylabel style = {anchor=north east},
        xtick=\empty, ytick=\empty,
        clip=false,
        xmin = -2.5, xmax = 2.5, ymin = -5.3, ymax = 5.3]
        
        \addplot [name path=A,domain=-2:2, samples=201, color=black] {3/2*x} node[right] {};
        
        % The positive part
        \coordinate (V1) at ($(2,3.8)$);
        \coordinate (V2) at ($(2,2.2)$);
        \node[align=right] (vp1) at ($(-0.4,3.8)$) {$(V+1)B$};
        \fill (0,3.8) circle (2pt);
        \node[align=right] at ($(-0.24,3)$) {$VB$};
        \fill (0,3) circle (2pt);
        \node[align=right] at ($(-0.4,2.2)$) {$(V-1)B$};
        \fill (0,2.2) circle (2pt);
        
        \draw[|-|, thick, color=red] (V1.center) -- (V2.center);
        \draw[dashed] (0,3.8) -- (V1.center);
        \draw[dashed] (0,3) -- (2,3);
        \draw[dashed] (0,2.2) -- (V2.center);
        
        \node at ($(2,4.3)$) {\color{red}$E_{f_1}[\wh f(x_0)]$};
        \draw[dashed] (2,2.2) -- (2,0);
        \fill (2,0) circle (2pt);
        \node at ($(2,-0.4)$) {$1$};
        
        % The negative part
        \coordinate (V3) at ($(-2,-3.8)$);
        \coordinate (V4) at ($(-2,-2.2)$);
        \node at ($(0.4,-3.8)$) {$-(V+1)B$};
        \fill (0,-3.8) circle (2pt);
        \node at ($(0.24,-3)$) {$-VB$};
        \fill (0,-3) circle (2pt);
        \node at ($(0.4,-2.2)$) {$-(V-1)B$};
        \fill (0,-2.2) circle (2pt);
        
        \draw[|-|, thick, color=red] (V3.center) -- (V4.center);
        \draw[dashed] (0,-3.8) -- (V3.center);
        \draw[dashed] (0,-3) -- (-2,-3);
        \draw[dashed] (0,-2.2) -- (V4.center);
        
        \node at ($(-2,-4.3)$) {\color{red}$E_{f_{-1}}[\wh f(x_0)]$};
        \draw[dashed] (-2,-2.2) -- (-2,0);
        \node at ($(-2,0.4)$) {$-1$};
        \fill (-2,0) circle (2pt);
        \end{axis}
    \end{tikzpicture}
    % }
    \caption{How an upper bound $B$ on the bias results in a lower bound for $E_{f_{\theta}}[\wh f(x_0)]$ with $\theta = \pm 1$ (red intervals).}
    \label{fig:contraint_bias}
\end{figure}
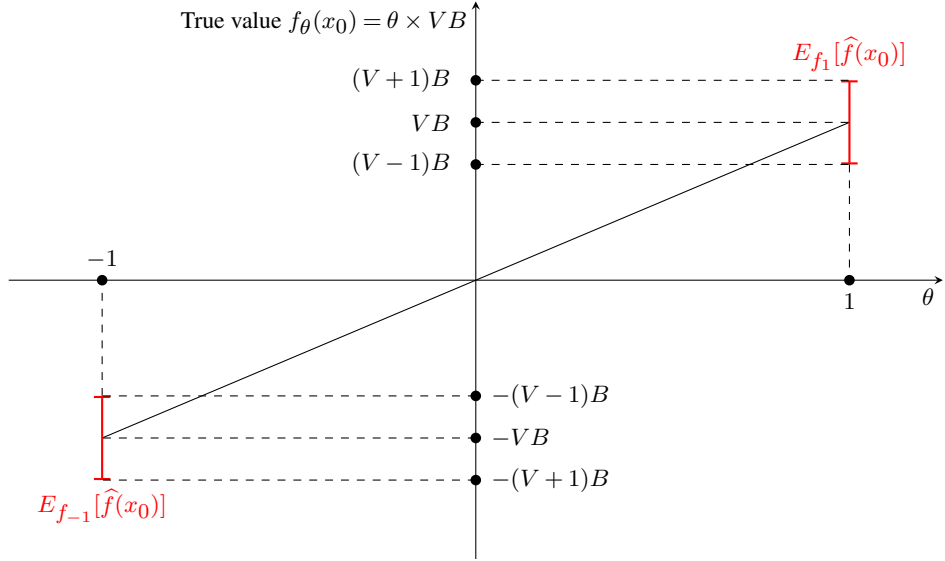

Recall that $K(0)=1$ and notice that it is enough to prove the result for $V \geq 1.$ Therefore, $\Bias_{f_\theta}(\wh f(x_0)) = E_{f_{\theta}}[\wh f(x_0)]-f_{\theta}(x_0) = E_{f_{\theta}}[\wh f(x_0)]-\theta V B$. As displayed in Figure~\ref{fig:contraint_bias}, applied to the parameter values $\theta = \pm 1$ (that is, the extreme elements of the parametric family) this yields the constraints $E_{f_1}[\wh f(x_0)]\geq (V-1)B$ and $E_{f_{-1}}[\wh f(x_0)]\leq -(V-1)B.$
Choosing for the lower bound $f_1$ if $E_{f_0}[\wh f(x_0)]$ is negative and $f_{-1}$ if $E_{f_0}[\wh f(x_0)]$ is positive, we find that $\big(E_{f_{\theta}}\big[\wh f(x_0)\big]-E_{f_0}\big[\wh f(x_0)\big]\big)^2 \geq (V-1)^2 B^2$ for either $\theta = 1$ or $\theta=-1$. Therefore,
\begin{align*}
    (V-1)^2 B^2
    \leq n V^2 B^{2+1/\beta} \|K\|_{L^2(\Rb)}^2
    \sup_{|\theta|\leq 1} \Var_{f_{\theta}}\big(\wh f(x_0)\big).
\end{align*}
Dividing both sides by $n V^2 B^2 \|K\|_{L^2(\Rb)}^2$ yields \eqref{eq.pointwise_to_show_i}.

To complete the proof, it remains to consider the case $B = 0$. Let $\wh f$ be an estimator such that $B=\sup_{f\in \Cc^\beta(R)}|\Bias_f(\wh f)| = 0$. Define the estimator $\wh f_\delta := \wh f + \delta$ with $\delta\in (0,1).$ Since $\delta$ is deterministic, $\Var_f(\wh f_\delta(x_0))=\Var_f(\wh f(x_0)).$ Applying the lower bound derived above gives
\begin{align*}
    \delta^{1/\beta }\sup_{f\in \Cc^\beta(R)}
    \Var_f\big(\wh f(x_0)\big)
    \geq 
    \frac 1n \Bigg(\|K\|_2^{-1} \bigg( 1 - \dfrac{\|K\|_{\Cc^\beta(\Rb)}}{R} \bigg)_+ \Bigg)^2.
\end{align*}
For $\delta \rightarrow 0,$ we obtain $\sup_{f\in \Cc^\beta(R)}\Var_f(\wh f(x_0)) \rightarrow \infty$ and the conclusion holds because of $(+\infty)\cdot 0=+\infty.$ This completes the proof for $(i)$.

\medskip

{\it (ii):} We use the same notation as for the proof of $(i).$ It is sufficient to show that for an arbitrary estimator $\wh f$ with worst-case bias $B<1$ and any $f\in \Cc^\beta(R),$
\begin{align}
    B^{1/\beta}
    \Var_f\big(\wh f(x_0)\big) 
    \geq  
    \frac {\ol\gamma(R,\beta,C,\|f\|_{\Cc^\beta([0,1])})} n.
    \label{eq.pointwise_to_show_ii}
\end{align}
Assume first that $B>0.$ For any function $K \in \Cc^\beta(\Rb)$ satisfying $K(0) = 1$ and $\|K\|_2 < +\infty,$ define $U:=(R-\|f\|_{\Cc^\beta([0,1])})/\|K\|_{\Cc^\beta(\Rb)}$ and
\begin{equation*}
    \Gc:=\Big\{f_\theta(x) = f(x)+
    \theta
    U B K \Big( \frac{ x - x_0 }{B^{1/\beta}} \Big) : |\theta| \leq  1 \Big\}.
\end{equation*}
Combining the fact that the triangle inequality holds for any norm with Lemma \ref{lem.kernel_norm_bd} (using $B<1$) and $|\theta|\leq 1,$ we obtain $\|f_\theta\|_{\Cc^\beta([0,1])}\leq \|f\|_{\Cc^\beta([0,1])}+U \|K\|_{\Cc^\beta(\Rb)}\leq R.$ Hence $\Gc \subseteq \Cc^\beta(R).$ As explained at the beginning of Section \ref{sec:gauss_wn}, the $\chi^2$-divergence in this model is $\chi^2(P_f,P_g)=\exp(n\|f-g\|_{L^2[0,1]}^2)-1.$ By assumption, $B^{2+1/\beta}\leq C/n.$ Combining this with the inequality $e^x-1\leq xe^x$ and arguing as in \eqref{eq.pointwise_L2_bd}, we find that 
\begin{align*}
    \chi^2(P_{f_{\pm 1}},P_{f_0})
    &\leq n\|f_{\pm 1}-f\|_{L^2[0,1]}^2\exp\Big(n\|f_{\pm 1}-f\|_{L^2[0,1]}^2\Big) \\
    &\leq nU^2 B^{2+1/\beta} \|K\|_2^2
    \exp\big( C U^2\|K\|_2^2\big). 
\end{align*}
Applying the $\chi^2$-divergence version of Lemma \ref{lem.general_lb} to the random variable $\wh f(x_0)$ and using the just derived bound for the $\chi^2$-divergence in the Gaussian white noise model yields
\begin{align*}
    \big(E_{f_{\pm 1}}\big[\wh f(x_0)\big] - 
    E_0\big[\wh f(x_0)\big] \big)^2
    \leq nU^2 B^{2+1/\beta} \|K\|_2^2
    \exp\big(C U^2\|K\|_2^2\big) \Var_f\big(\wh f(x_0)\big).
\end{align*}
By arguing as for the proof of $(i)$ with the constant $V$ replaced by $U,$ we obtain 
\begin{align*}
    (U-1)_+^2 B^2
    \leq nU^2 B^{2+1/\beta} \|K\|_2^2
    \exp\big(C U^2\|K\|_2^2\big) \Var_f\big(\wh f(x_0)\big).
\end{align*}
Rearranging the terms and taking the supremum over all kernels $K\in \Cc^\beta(\Rb)$ with $K(0)=1$ yields \eqref{eq.pointwise_to_show_ii}.

The case $B=0$ can be treated in the same way as in the proof for $(i)$ since we can always choose a sufficiently small $\delta>0,$ such that $\wh f_{\delta}=\wh f+\delta \in \Sc.$
\end{proof}

\subsection{Proof of \eqref{eq.8888}}
\label{sec.proof8888}

Assume $a < R-1$, and consider the function $K_A := \exp(-x^2/A)$, where $A > 0$ is chosen such that $\|K_A\|_{\Cc^\beta} = (1 + R-a) / 2$. Such a choice is possible since $A>0 \mapsto \|K_A\|_{\Cc^\beta}$ is a continuous function onto $(1, +\infty)$.
Note that $\|K_A\|_2^2 = \sqrt{A \pi / 2}$ and that $K_A(0) = 1$.
Therefore, for every $b \in [0,a]$, $\ol \gamma(R,\beta,C,b) \geq \ol \gamma^*(R,\beta,C,b),$ where
\begin{align*}
    \ol \gamma^*(R,\beta, C, b)
    &:= \|K_A\|_2^{-2} \bigg( 1 - \dfrac{\|K_A\|_{\Cc^\beta(\Rb)}}{R-b} \bigg)_+^2 \exp\bigg(-C(R-b)^2 \frac{\|K_A\|_2^2}{\|K_A\|_{\Cc^\beta(\Rb)}^2} \bigg) \\
    &= \frac{1}{\sqrt{A \pi / 2}} \bigg( 1 - \dfrac{1 + R - a}{2(R - b)} \bigg)^2 \exp\bigg(-C(R-b)^2 \frac{\sqrt{8 A \pi}}{(1 + R - a)^2} \bigg) \\
    &> 0,
\end{align*}
using $2(R - b) > 1 + R - a$ which holds since $b \leq a < R - 1$.
This is a positive continuous function over the compact interval $[0,a]$, so $0 < \inf_{b\leq a}\ol \gamma^*(R,\beta,C,b) \leq \inf_{b\leq a}\ol \gamma(R,\beta,C,b)$.\qed

\section{Proofs for Section \ref{sec.boundary}}
\label{sec.proofs_boundary}

The information measures in the support boundary model are governed by the $L^1$-geometry. For a detailed description of the following results, see Section 2 in \cite{MR4158798}. If $P_f$ denotes the distribution of the data for support boundary $f,$ then it can be shown that $P_f$ is dominated by $P_g$ if and only if $g \leq f$ pointwise. If $g\leq f$, then, the likelihood ratio is given by $dP_f/dP_g =\exp(n\int_0^1 (f(x)-g(x)) \, dx)\mathbf{1}(\forall i: f(X_i) \leq Y_i)$. In particular, we have for $g \leq f,$ $\alpha>0,$ and $\|\cdot\|_1$ the $L^1([0,1])$-norm, $E_g[(dP_f/dP_g)^\alpha]=\exp(n\|f-g\|_1(\alpha-1))E_g[dP_f/dP_g]=\exp(n\|f-g\|_1(\alpha-1))$ and so $H^2(P_f,P_g)=1-\exp(-\tfrac{n}2 \|f-g\|_1)$ and $\chi^2(P_f,P_g)=\exp(n\|f-g\|_1)-1.$

\medskip

Before proving Theorem \ref{thm.B_V_tradeoff_boundary}, we briefly comment on the change of expectation inequalities (Lemma \ref{lem.general_lb}) applied to this model. Since $\KL(P_f,P_g)+\KL(P_g,P_f)=\infty$ whenever $f \neq g,$ the Kullback-Leibler version of Lemma \ref{lem.general_lb} is not applicable in this case. Also we argued earlier that for regular models, we can retrieve the Cram\'er-Rao lower bound from the lower bounds in Lemma \ref{lem.general_lb} by choosing $P=P_\theta,$ $Q=P_{\theta+\Delta}$ and letting $\Delta$ tend to $0.$ As no Fisher information exists in the support boundary model, it is of interest to study the abstract lower bounds in Lemma \ref{lem.general_lb} under the limit $\Delta \to 0$. For this, consider constant support boundaries $f_\theta=\theta.$ It is then natural to evaluate the lower bounds for the sufficient statistic $X = \min_i Y_i$ for $\theta.$ That this is indeed a sufficient statistic can be shown by first observing that from the likelihood ratio formula given above, it follows that $X$ is the MLE for $\theta$ and then applying Proposition 3.1 in \cite{MR3606758}. Moreover, under $P_{f_\theta}$, $X-\theta$ follows an exponential distribution with rate parameter $n$, see Section 4.1 in \cite{reissSHAnnals} for more details. With $P=P_{f_{\theta}}$ and $Q=P_{f_{\theta+\Delta}},$ $(H^{-1}(P,Q)-H(P,Q))^{-2}=e^{n\Delta}(1-e^{-n\Delta/2}),$ $2+2(1-H^2(P,Q))=2(1+e^{-n\Delta/2})$ and $\chi^2(P,Q)\wedge \chi^2(Q,P)=e^{n\Delta}-1.$ Since $E_P[X]=\theta+1/n,$ $E_Q[X]=\theta+\Delta+1/n,$ and $\Var_P(X)=\Var_Q(X)=1/n^2,$ we find that the Hellinger lower bound \eqref{eq.lem1_2} can be rewritten as $\Delta^2 \leq 4(e^{n\Delta}-1)/n^2$ and the $\chi^2$-divergence lower bound \eqref{eq.lem1_4} becomes $\Delta^2\leq (e^{n\Delta}-1)/n^2.$ In both inequalities the upper bound is of the order $\Delta^2$ if $\Delta \asymp 1/n.$ Otherwise the inequalities are suboptimal in the sense that the rates on the right hand side and left hand side of the inequalities do not match. While the Cram\'er-Rao asymptotics $\Delta \to 0$ for fixed $n$ does not yield anything useful here, we still can obtain rate-optimal lower bounds for the bias-variance trade-off by applying a change of expectation inequality in the regime $\Delta=\Delta_n\asymp 1/n.$

\begin{proof}[Proof of Theorem \ref{thm.B_V_tradeoff_boundary}]
We follow the same strategy as in the proof of Theorem \ref{thm.pointwise}. Let $B:=\sup_{f\in  \Cc^\beta(R)} \, |\Bias_f(\wh f(x_0))|.$ Assume first that $B>0.$ By assumption, we can find a function $K \in L^2(\Rb)$ satisfying $\|K\|_{\Cc^\beta(\Rb)}< (R+\kappa)/4,$ $K(0) = 1$ and $K \geq 0.$ For such a $K,$ define $U:=(R-\|f\|_{\Cc^\beta([0,1])})/\|K\|_{\Cc^\beta(\Rb)}$ and observe that $U>2,$ whenever $f\in \Cc^\beta((R-\kappa)/2).$ Let
\begin{equation*}
    \Gc:=\Big\{f_\theta(x) = f(x)+
    \theta
    U B K \Big( \frac{ x - x_0 }{B^{1/\beta}} \Big) : |\theta| \leq  1 \Big\}.
\end{equation*}
As seen in the proof of Theorem \ref{thm.pointwise}, this defines a subset of the H\"older ball $\Cc^\beta(R).$ As derived in Section \ref{sec.boundary}, the $\chi^2$-divergence in this model is $\chi^2(P_f,P_g)=\exp(n\|f-g\|_1)-1,$ whenever $f \geq g.$ By assumption, $B^2\leq \sup_{f\in \Cc^\beta(R)} \MSE_{f}(\wh f(x_0)) \leq (C/ n)^{2\beta/(\beta+1)}.$ Rewriting gives $B^{1+1/\beta}=B^{(\beta+1)/\beta}\leq C/n.$ Combining this with the inequality $e^x-1\leq xe^x$ and using that $f \leq f_1$ pointwise, we find that 
\begin{align*}
    \chi^2(P_{f_1},P_f)
    &\leq n\|f_1-f\|_1\exp\Big(n\|f_1-f\|_1\Big) \\
    &\leq n U B^{1+1/\beta} \|K\|_1
    \exp\big( C U\|K\|_1\big). 
\end{align*}
Applying the $\chi^2$-divergence version of Lemma \ref{lem.general_lb} to the random variable $\wh f(x_0)$ and using the just derived bound for the $\chi^2$-divergence yields
\begin{align*}
    \big(E_{f_1}\big[\wh f(x_0)\big] - 
    E_f\big[\wh f(x_0)\big] \big)^2
    \leq n U B^{1+1/\beta} \|K\|_1
    \exp\big( C U\|K\|_1\big) \Var_f\big(\wh f(x_0)\big).
\end{align*}
Due to $K(0)=1,$ we have that $f_1(x_0)-f(x_0)= U B.$ Since $B$ is the supremum over the absolute value of the bias, it follows that $E_{f_1}[\wh f(x_0)] - E_f[\wh f(x_0)]\geq UB-2B$ and consequently
\begin{align}
    (U-2)_+^2 B^2
    \leq n U B^{1+1/\beta} \|K\|_1
    \exp\big( 2C U\|K\|_1\big) \Var_f\big(\wh f(x_0)\big).
    \label{eq.thm_boundary1}
\end{align}
Recall that $U>2,$ whenever  $f\in \Cc^\beta((R-\kappa)/2).$ Due to $\beta<1,$ the bound $B^2<c n^{-2\beta/(\beta+1)}$ implies $B^{1-1/\beta}/n \geq c^{(1-1/\beta)/2}n^{-2\beta/(\beta+1)}.$ By making $c$ sufficiently small, \eqref{eq.thm_boundary1} shows that eventually $\Var_0(\wh f(x_0))\geq (C/n)^{2\beta/(\beta+1)}.$ This is a contradiction, since we have also $\Var_0(\wh f(x_0))\leq \MSE_{0}(\wh f(x_0))< (C/n)^{2\beta/(\beta+1)}.$ Hence, there exists a value $c=c(\beta, C, R),$ such that $B^2\geq c n^{-2\beta/(\beta+1)}.$ This proves \eqref{eq.thm_boundary_claim1}. 

To verify \eqref{eq.thm_boundary_claim2}, we can use the inequality $B^2\leq \sup_{f\in \Cc^\beta(R)} \, \MSE_f(\wh f(x_0)) \leq (C/n)^{2\beta/(\beta+1)}.$ This gives $B^{1-1/\beta}/n \geq C^{(1-1/\beta)/2}n^{-2\beta/(\beta+1)}$ and if inserted in \eqref{eq.thm_boundary1} shows the existence of a positive constant $c'(\beta, C, R)$ such that $\Var_f(\wh f(x_0))\geq c'(\beta, C, R)n^{-2\beta/(\beta+1)}.$

Suppose now that $B=0$ holds. Then we can add a (deterministic) positive sequence $\delta_n < \sqrt{c} n^{-\beta/(\beta+1)}$ to the estimator such that for the perturbed estimator $\wh f_\delta,$ we still have $\sup_{f\in \Cc^\beta(R)}\MSE_f(\wh f_\delta(x_0))<(C/n)^{2\beta/(\beta+1)}.$ Since $B^2< cn^{-2\beta/(\beta+1)},$ applying the argument above shows that such an estimator cannot exist. Therefore, $B=0$ is impossible.
\end{proof}

\section{Proofs for Section \ref{sec.reduction}}
\label{sec.proofs_reduction}

\begin{proof}[Proof of Proposition \ref{prop.L2_reduction_lb_1}] It will be enough to prove the result for $\Gamma_\beta$ replaced by $\|K\|_{S^\beta}$ for an arbitrary function $K \in S^\beta(\Rb)$ with $\|K\|_{L^2(\Rb)}=1$ and support contained in $[-1/2, 1/2]$. Introduce
\begin{equation}
    \Fc:=\bigg\{ f_\theta(x) = \sum_{i = 1}^m \theta_i \sqrt{m}
    K \big( m x- (i-1/2)\big) : \|\theta\|_2 \leq \frac R{\|K\|_{S^\beta}m^\beta} \bigg\}.
    \label{eq.Fc_def_L2-version}
\end{equation}
The support of the function $K( m x- (i-1/2))$ is contained in $[i-1,i].$ For different $i$ and $j,$ the dilated and scaled kernel functions have therefore disjoint support and
\begin{align*}
    \| f_\theta \|_{S^\beta}^2
    &= \int_0^1 \bigg( \sum_{i = 1}^m \theta_i
    \sqrt{m} K \big( m x- (i-1/2)\big)\bigg)^2 \, dx \\
    &\qquad + \int_0^1 \bigg(\sum_{i = 1}^m \theta_i m^{\beta+1/2}
    K^{(\beta)} \big( m x- (i-1/2)\big)\bigg)^2 \, dx \\
    &= \sum_{i = 1}^m \theta_i^2 \int_0^1 m K \big( m x- (i-1/2)\big)^2 \\
    &\qquad + m^{2\beta+1}
    K^{(\beta)} \big( m x- (i-1/2)\big)^2 \, dx \\
    &= \sum_{i = 1}^m \theta_i^2 m^{2\beta} \|K\|_{S^\beta}^2 \leq R,
\end{align*}
so that $\Fc \subset S^\beta(R),$ since $\|\theta\|_2\leq  R /(\|K\|_{S^\beta}m^{\beta}).$ It is therefore sufficient to prove Proposition \ref{prop.L2_reduction_lb_1} with $S^\beta(R)$ replaced by $\Fc.$ We say that two statistical models are equivalent if the data can be transformed into each other without knowledge of the unknown parameters. The Gaussian white noise model \eqref{eq.mod_GWN} is by definition equivalent to observing all functionals $\int_0^1 \phi(t) \, dY_t$ with $\phi \in L^2([0,1]).$ In particular, for any orthonormal $L^2([0,1])$ basis $(\phi_i)_{i=1,\dots},$ the Gaussian white noise model is equivalent to observing $X_i:=\int_0^1 \phi_i(t) \, dY_t,$  $i=1,\dots$ The latter is the well-known sequence space formulation. The functions $\psi_i:=\sqrt{m}K( m \cdot - (i-1/2))$ are orthogonal (because of the disjoint support) and $L^2$-normalized. Choosing $\phi_i=\psi_i$ for $i=1,\dots,m$ and extending this to an orthonormal basis of $L^2([0,1]),$ we find that the Gaussian white noise model with parameter space $\Fc$ is equivalent to observing 
\begin{align*}
    X_i = \theta_i \mathbf{1}(i\leq m) +\frac1{\sqrt{n}} \eps_i,\quad i=1,\dots 
\end{align*}
with independent $\eps_i \sim \Nc(0,1).$ Here we have used that $\int_0^1 \phi_i(t) \, dY_t = \int_0^1 \phi_i(t) f(t) \, dt + n^{-1/2}\int_0^1 \phi_i(t) \, dW_t$ and that $\eps_i:=\int_0^1 \phi_i(t) \, dW_t$ are standard normal and independent.

Because of the equivalence, every estimator $\wh f$ in the Gaussian white noise model with parameter space $\Fc$ can be rewritten as an estimator $\wh f=\wh f(X_1,\dots)$ depending on the transformed data $X_1,X_2,\dots$ Moreover, for any estimator $\wh f$ for the regression $f$ in the Gaussian white noise model, we can consider the estimator $\wt \theta=(\wt \theta_1, \dots, \wt \theta_m)$ with $\wt \theta_i := \int_0^1 \wh f(x) \psi_i(x) \, dx.$ This is now an estimator depending on $X_1,X_2,\dots$ Observe that $(X_1,\dots, X_m)$ is a sufficient statistic for the vector $\theta.$ In view of the Rao-Blackwell theorem, it is then natural to eliminate the dependence on $X_{m+1},X_{m+2}, \dots$ by considering the estimator $\wh \theta_i := E[\wt \theta_i|X_1,\dots,X_m].$ This estimator only depends on the Gaussian sequence model with data $(X_1,\dots, X_m).$ 

The proof is complete if we can show that $\|E_\theta [\wh \theta] -\theta \|_2^2 \leq \IB_{f_\theta}(\wh f)$ and $\sum_{i=1}^m \Var_\theta\big(\wh \theta_i\big) \leq \IVar_{f_\theta}(\wh f)$ for all $f_\theta \in \Fc,$ or equivalently, for all $\theta \in \Theta.$ First observe that $\|E_\theta [\wh \theta] -\theta \|_2^2=\|E_\theta [\wt \theta] -\theta \|_2^2$ and by using the formula for the conditional variance, we have $\Var_\theta(\wh \theta_i)=\Var_\theta(\wt \theta_i)-E[\Var_\theta(\wt \theta_i |X_1,\dots,X_m)]\leq \Var_\theta(\wt \theta_i)$ for all $i=1,\ldots,m.$ It is therefore sufficient to show that $\|E_\theta [\wt \theta] -\theta \|_2^2 \leq \IB_{f_\theta}(\wh f)$ and $\sum_{i=1}^m \Var_\theta(\wt \theta_i) \leq \IVar_{f_\theta}(\wh f)$ for all $f_\theta \in \Fc.$

Denote by $\Gc$ the linear span of $(\psi_i)_{i=1,\dots,m}$ and by $\Gc^c$ the orthogonal complement of $\Gc$ in $L^2([0,1]).$ Obviously, $\Gc$ is a finite-dimensional subspace of $L^2([0,1])$ and hence closed. Let $\wt f := \sum_{i=1}^m \wt \theta_i \psi_i$ with $\wt \theta_i$ as defined above.  Since $\wt f$ is the $L^2$-projection of $f$ on $\Gc,$ it holds that $\wh f-\wt f\in \Gc^c.$ Consequently, $\wt f$ and $\wh f-\wt f$ must be orthogonal in $L^2([0,1])$. Moreover, also $E_{f_\theta}[\wt f] \in \Gc$ and $E_{f_\theta}[\wh f-\wt f] \in \Gc^c.$ Therefore, for any $f_\theta\in \Fc,$
\begin{align*}
    \IVar_{f_\theta}(\wh f)
    &= \int_0^1 \Var_{f_\theta}(\wh f(x)) \, dx
    = \int_0^1 E_{f_\theta} \Big[ \big( \wh f(x)) - E_{f_\theta}[\wh f(x)] \big)^2 \Big] \, dx \\
    &= E_{f_\theta} \Big[ \int_0^1 \big( \wh f(x)) - E_{f_\theta}[\wh f(x)] \big)^2 \, dx \Big] \\
    &= E_{f_\theta} \Big[ \big\| \wt f + (\wh f-\wt f)
    - E_{f_\theta} \big[\wt f + (\wh f-\wt f)\big] \big\|_2^2 \Big] \\
    &= E_{f_\theta}\Big[ \big\| \wt f - E_{f_\theta}[\wt f ] \big\|_2^2 \Big]
    + E_{f_\theta}\Big[ \big\| \wh f-\wt f - E_{f_\theta}[\wh f-\wt f ] \big\|_2^2 \Big] \\
    &\geq E_{f_\theta} \Big[ \big\| \wt f - E_{f_\theta}[\wt f ] \big\|_2^2 \Big] \\
    &= \IVar_{f_\theta}(\wt f).
\end{align*}
Using that the $\psi_i$ are orthonormal with respect to $L^2([0,1])$,
\begin{align*}
    \IVar_{f_\theta}(\wt f)
    &= \int_0^1 E_{f_\theta}\Bigg[ \bigg( \sum_{i=1}^m (\wt \theta_i - E_{f_\theta}[\wt \theta_i]) \psi_i(x) \bigg)^2 \Bigg] \, dx \\
    &= \int_0^1 E_{f_\theta} \Bigg[ \sum_{i=1}^m \Big(\wt \theta_i - E_{f_\theta} [\wt \theta_i]\Big)^2 \psi_i^2(x) \Bigg] \, dx \\
    &= \sum_{i=1}^m \Var_\theta\big(\wt \theta_i\big).
\end{align*}
Combined with the previous display, this proves that $\sum_{i=1}^m \Var_\theta(\wt \theta_i) \leq \IVar_{f_\theta}(\wh f)$ for all $f_\theta \in \Fc.$

With the same notation as above, we find using $f_\theta \in \Gc,$
\begin{align*}
    \IB_{f_\theta}\big(\wh f\big)
    &= \int_0^1 \big( E_{f_\theta}[\wh f(x)] - f_\theta(x) \big)^2 \, dx \\
    &= \big\| E_{f_\theta}[\wh f] - f_\theta \big\|_2^2 \\
    &= \big\| E_{f_\theta}[\wt f] - f_\theta \big\|_2^2
    + \big\| E_{f_\theta}\big[\wh f- \wt f\big] \big\|_2^2 \\
    &\geq \big\| E_{f_\theta}[\wt f] - f_\theta \big\|_2^2 \\ 
    &= \IB_{f_\theta}\big(\wt f\big)
\end{align*}
and
\begin{align*}
    \IB_{f_\theta}(\wt f)
    &= \int_0^1 \Big( E_{f_\theta}\big[\wt f(x)\big] - f_\theta(x) \Big)^2 \, dx \\
    &= \int_0^1 \bigg( \sum_{i=1}^m \Big(E_{f_\theta} \big[\wt \theta_i\big] - \theta_i\Big)
    \psi_i(x) \bigg)^2 \, dx \\
    &= \int_0^1 \sum_{i=1}^m \Big(E_{f_\theta}\big[\wt \theta_i\big]  - \theta_i\Big)^2
    \psi_i^2(x) \, dx \\
    &= \sum_{i=1}^m \Big( E_{f_\theta}\big[\wt \theta_i\big] - \theta_i \Big)^2 \\
    &= \big\|E_\theta \big[\wt \theta\big] -\theta \big\|_2^2.
\end{align*}
This finally proves $\|E_\theta [\wt \theta] -\theta \|_2^2 \leq \IB_{f_\theta}(\wh f).$ The proof is complete.
\end{proof}

\begin{proof}[Proof of Proposition \ref{prop.spheri_symm}]
We follow Stein~\cite[p.201]{stein1956inadmissibility} and denote by $\mu$ the Haar measure on the orthogonal group $\Om.$ In particular, $\mu(\Om)=1.$ We write $\wh\theta(X)$ and $\wt \theta(X)$ to highlight the dependence on the sample $X \in \Rb^m.$ Given $\wh \theta(X),$ define
\begin{align*}
    \wt \theta(X)
    := \int D^{-1} \wh\theta(DX) \, d\mu(D),
\end{align*}
where the integral is over the orthogonal group. By construction, $\wt \theta(X)$ is a spherically symmetric estimator. Using Jensen's inequality, the fact that $DX \sim \Nc(D\theta,I_m/n)$ with $I_m$ the $m\times m$ identity matrix, and $\theta=D^{-1}D \theta$ yields for any $\theta\in \Theta_m^\beta(R),$
\begingroup\allowdisplaybreaks
\begin{align*}
    \big\| E_\theta \big[ \wt \theta(X) \big] - \theta \big\|_2^2
    &= \bigg\| E_\theta \bigg[ \int_{D \in \Om} 
    D^{-1} \wh\theta(DX) \, d\mu(D) \bigg]
    - \theta \bigg\|_2^2 \\
    &\leq \int_{D \in \Om}
    \Big\| E_\theta \big[ D^{-1} \wh\theta(DX) \big]
    - \theta \Big\|_2^2 d\mu(D) \\
    &\leq \int_{D \in \Om}
    \Big\| E_{D\theta} \big[ D^{-1} \wh\theta (X) \big]
    - \theta \Big\|_2^2 d\mu(D) \\
    &\leq \int_{D \in \Om}
    \Big\| E_{D\theta} \big[ \wh\theta (X) \big]
    - D\theta \Big\|_2^2 d\mu(D) \\
    &\leq \sup_{\theta\in \Theta_m^\beta(R)}\big\|E_\theta \big[\wh \theta(X)\big] -\theta \big\|_2^2.
\end{align*}
\endgroup
With $e_i$ the $i$-th standard basis vector of $\Rb^m,$ we also find using that $\Tr(AB)=\Tr(BA),$ $D=(D^{-1})^\top,$ and again $DX \sim \Nc(D\theta,I_m/n),$
\begin{align*}
    \sum_{i=1}^m \Var_\theta\big(\wt \theta_i(X)\big)
    &= \int_{D \in \Om} \sum_{i=1}^m
    \Var_\theta \Big(e_i^\top D^{-1}\wh\theta(DX) \Big) \, d\mu(D) \\
    &= \int_{D \in \Om} \Tr \Big[
    \Var_\theta \big(D^{-1}\wh\theta(DX) \big)
    \Big] \, d\mu(D) \\
    &= \int_{D \in \Om} \Tr \Big[ D^{-1}
    \Var_\theta \big( \wh\theta(DX) \big) (D^{-1})^\top
    \Big] \, d\mu(D) \\
    &= \int_{D \in \Om} \Tr \Big[
    \Var_\theta \big( \wh\theta(DX) \big)
    \Big] \, d\mu(D) \\
    &= \int_{D \in \Om} \Tr \Big[
    \Var_{D\theta} \big( \wh\theta(X) \big)
    \Big] \, d\mu(D) \\
    &\leq  \sup_{\theta\in \Theta_m^\beta(R)} \sum_{i=1}^m \Var_\theta\big(\wh \theta_i\big).
\end{align*}
\end{proof}

\begin{lemma}
\label{lem.minimx_inv2}
Any function $h(x)$ satisfying $h(x)=D^{-1}h(Dx)$ for all $x\in \Rb^m$ and all orthogonal transformations $D$ must be of the form 
\begin{align*}
	h(x)= r(\|x\|_2) x
\end{align*} 
for some univariate function $r.$
\end{lemma}

\begin{proof}
Throughout the proof, we write $\|\cdot\|$ for the Euclidean norm. In a first step of the proof, we show that 
\begin{align}
	h(x)= \lambda(x) x
	\label{eq.proof_minimax_inv_1}
\end{align} 
for some univariate function $\lambda.$

Fix $x$ and consider an orthogonal basis $v_1 :=x/\|x\|, v_2,\ldots, v_m$ of $\Rb^m.$ The orthogonal matrix $D := \sum_{j=1}^m (-1)^{\mathbf{1}(j\neq 1)} v_j v_j^\top$ has eigenvector $v_1=x/\|x\|$ with corresponding eigenvalue one. For all other eigenvectors the eigenvalue is always $-1.$ Using that $h(x)=D^{-1}h(Dx),$ we find that $h(x)=D^{-1}h(x)$ which implies that $h(x)$ is a multiple of $x$ and therefore $h(x)=\lambda(x)x,$ proving \eqref{eq.proof_minimax_inv_1}.

Let $x$ and $y$ be such that $\|x\|=\|y\|.$ Let $v=x-y,$ and observe that $D = I - 2 v v^\top/\|v\|^2$ is an orthogonal matrix. Since $\|v\|^2=2\|x\|^2-2y^\top x=2\|y\|^2-2y^\top x,$ we also have that $D x=y$ and $Dy=x.$ For this $D,$ we have
\begin{align*}
	\lambda(x)x=h(x)=h(Dy)=Dh(y)=\lambda(y)Dy =\lambda(y)x
\end{align*}
which shows that $\lambda(x)=\lambda(y)$ whenever $\|x\|=\|y\|.$ Differently speaking, $\lambda$ only depends on $y$ through $\|y\|.$ This completes the proof.
\end{proof}

\begin{proof}[Proof of Theorem \ref{thm.LB_L2}]
Fix an estimator $\wh f$ in the Gaussian white noise model \eqref{eq.mod_GWN} and set $B:=\sup_{f \in S^\beta(R)} \IB_f(\wh f).$ Consider first the case that $B>0.$ Choose $m_*:=\lfloor B^{-1/\beta}\rfloor$ and observe that since $B<2^{-\beta},$ we must have $m_*\geq 2.$ Also $2m_*\geq m_*+1\geq B^{-1/\beta}$ and so $m_*\geq B^{-1/\beta}/2.$ Applying Proposition \ref{prop.L2_reduction_lb_1} and Proposition \ref{prop.spheri_symm}, there exists a spherically symmetric estimator $\wt \theta$ in the Gaussian sequence model with $m=m_*$ satisfying 
\begin{align*}
    \sup_{\theta \in \Theta_{m_*}^\beta(R)} \big\|E_\theta \big[\wt \theta\big] -\theta \big\|_2^2 \leq B \quad \text{and}  \ \sup_{\theta \in \Theta_{m_*}^\beta(R)} \sum_{i=1}^{m_*} \Var_\theta\big(\wt \theta_i\big) 
    \leq \sup_{f\in S^\beta(R)} \IVar_f\big(\wh f\big).
\end{align*}
Below we will construct a $\theta_0$ for which 
\begin{align}
    \sum_{i=1}^{m_*} \Var_{\theta_0}\big(\wt \theta_i\big) B^{1/\beta} \geq \frac{1}{8n}.
    \label{eq.to_show_L2_lb}
\end{align}
This proves then the result. 

Since by assumption $R\geq 2\Gamma_\beta,$ for any $\theta$ with $\|\theta\|_2= R/(\Gamma_\beta m_*^\beta),$ we have that $\|\theta\|_2\geq 2B$ and combined with \eqref{eq.IBias_formula}, this gives
\begin{align*}
    \big\|E_\theta \big[\wt \theta\big] -\theta \big\|_2^2
    = \|\theta\|_2^2 \big(t\big(\|\theta\|_2\big)-1\big)^2
    \geq 4B^2 \big(t\big(\|\theta\|_2\big)-1\big)^2.
\end{align*}
As $B$ is an upper bound for the bias, $|t(\|\theta\|_2)-1|\leq 1/2$ and thus $t(\|\theta\|_2) \geq 1/2.$ 

Let $0<\Delta\leq 1/2$ and set $A:=R/(\Gamma_\beta m_*^{\beta+1/2}).$ Consider $\theta_0:=(A,\dots,A)^\top$ and $\theta_i=(\theta_{ij})_{j=1,\dots,m_*}^\top$ with $\theta_{ii}:=\sqrt{1+\Delta}\, A$ and $\theta_{ij}:=\sqrt{1-\Delta/(m_*-1)}\, A$ for $j\neq i.$ By construction $\|\theta_i\|_2=R/(\Gamma_\beta m_*^\beta)$ and in particular $\theta_i \in \Theta_{m_*}^\beta(R)$ for all $i=0,1,\dots,m_*.$ Using that $\sqrt{1+u}-1=u/2+O(u^2)$ for $u \to 0,$ we have for $i=1,\dots,m_*$ and $\Delta \to 0,$
\begin{align*}
    \frac{\|\theta_i-\theta_0\|_2^2}{\Delta^2}
    &=\frac{A^2}{\Delta^2}
    \bigg[
    (m_*-1) \Big( \sqrt{1-\frac{\Delta}{m_*-1}}-1\Big)^2+\Big(\sqrt{1+\Delta}-1\Big)^2\bigg] \\
    &= \frac{A^2}{4}\Big(1+\frac{1}{m_*-1}\Big) +O(\Delta).
\end{align*}
Similarly, for $i,j=1,\dots,m_*,$ $i\neq j$ and $\Delta \to 0,$
\begin{align*}
    \frac{\big\langle \theta_i-\theta_0, \theta_j-\theta_0\big \rangle }{\Delta^2}
    &=
    \frac{A^2}{\Delta^2}
    \bigg[
    (m_*-2) \Big( \sqrt{1-\frac{\Delta}{m_*-1}}-1\Big)^2 \\
    &\quad \quad \quad +2\Big(\sqrt{1+\Delta}-1\Big)\Big(\sqrt{1-\frac{\Delta}{m_*-1}}-1\Big)\bigg] \\
    &=
    -\frac{A^2}{4(m_*-1)}\Big(1+\frac{1}{m_*-1}\Big)+O(\Delta).
\end{align*}
Recall that $\|\theta_i\|_2=\|\theta_j\|_2$ by construction. Applying \eqref{eq.row_sum_norm_nd} to the random variables $\wt \theta_1,\dots,\wt \theta_{m_*}$ and using \eqref{eq.chi2_normal_distr} yields
\begin{align*}
    \frac 12 \sum_{i=1}^{m_*} \|\theta_i-\theta_0\|_2^2
    &\leq 
    \sum_{i=1}^{m_*} t(\|\theta_i\|_2) \|\theta_i-\theta_0\|_2^2
    =
    \sum_{i=1}^{m_*} \big\|E_{\theta_i}\big[\wt \theta \big]- E_{\theta_0}\big[\wt \theta \big]\big\|_2^2 \\
    &\leq \max_{i=1,\ldots,m_*} \sum_{j=1}^{m_*} \big |e^{n \langle \theta_i-\theta_0, \theta_j-\theta_0\rangle }-1 \big| \sum_{j=1}^{m_*} \Var_{\theta_0}\big(\wt \theta_j\big).
\end{align*}
Multiplying both sides of the inequality with $\Delta^{-2},$ using the expressions for $\Delta^{-2}\|\theta_i-\theta_0\|_2^2$ and $\Delta^{-2}\langle \theta_i-\theta_0, \theta_j-\theta_0\rangle,$ and letting $\Delta$ tend to zero yields $\sum_{i=1}^{m_*} \Var_{\theta_0}(\wt \theta_i)\geq m_*/(4n).$ As remarked above, $m_*\geq B^{-1/\beta}/2$ and this shows finally \eqref{eq.to_show_L2_lb} proving the theorem for $B>0.$ 

If $B=0$ we consider the estimator $\wh f_\delta := \wh f+\delta$ for an arbitrary deterministic $\delta>0$ that is sufficiently small such that $\wh f_\delta \in T.$ Observe that $\IVar_f(\wh f_\delta)=\IVar_f(\wh f).$ We can now apply the result from the first part and let $\delta$ tend to zero to verify that $\IVar_f(\wh f)$ must be unbounded in this case. The result follows since $0 \cdot (+\infty)$ is interpreted as $+\infty.$   
\end{proof}

The proof strategy carries over to the nonparametric regression model with fixed uniform design on $[0,1].$ The discretization effects result in a slight heteroscedasticity of the noise in the Gaussian sequence model which make the computations considerably more technical.

\section{Proofs for Section \ref{sec.highdimensional}}
\label{sec:proofs_highdimensional}

\textbf{About Equation~\eqref{eq.minimax_rate_fctal}:} To show that
\begin{align*}
    \inf_{\wh{\|\theta\|_2^2}} \ \  \sup_{\theta\in \Theta_n^2(s)} \, \big(E_\theta\big[\wh{\|\theta\|_2^2}\big] -\|\theta\|_2^2\big)^2
    \asymp s^2\log^2\Big(1+\frac{\sqrt{n}}s\Big)\asymp s^2\log^2\Big(\frac n{s^2}\Big)\vee n
\end{align*}
is the minimax estimation rate for $\|\theta\|_2^2,$ observe that the parameter space $\Theta_n^2(s)$ can be rewritten as
$\Theta_n^2(s) = \Theta(s) \cap B_2(\kappa)$,
where $B_2(\kappa) := \{\theta \in \Rb^n: \| \theta\|_2 \leq \kappa \}$ and 
$\kappa^2 := 2s\log(1+\sqrt{n}/s)$.
With this choice of $\kappa,$ Theorem 4 and Theorem 5 of \cite{MR3662444} can be rewritten in our notation as
\begin{align*}
    \inf_{\wh{\|\theta\|_2^2}} \ \sup_{\theta\in \Theta_n^2(s)}\big(E_\theta\big[\wh{\|\theta\|_2^2}\big] -\|\theta\|_2^2\big)^2
    \asymp \psi^Q_1 (s, n, \kappa),
\end{align*}
where, for any $\sigma, \kappa \geq 0$,
$\psi^Q_\sigma (s, n, \kappa) := \min \big(\kappa^4, \max(\sigma^2 \kappa^2, \bar{\psi}_\sigma(s,n) \big)$ and $\bar{\psi}_\sigma(s,n) := \sigma^4 s^2 \log^2(1+n/s^2)$ if $s < \sqrt{n}$ and $\bar{\psi}_\sigma(s,n) := \sigma^4 n$ if $s \geq \sqrt{n}$.

As we consider the case $\sigma=1$ and $\kappa^2=2s\log(1+\sqrt{n}/s)$, it follows that
\begin{align*}
    &\psi^Q_1 (s, n, \kappa)
    = \min \big(\kappa^4, \max(\kappa^2, \bar{\psi}_1(s,n)) \big) \\
    % &\asymp \min \Big( s^2 \log^2 \big(1 + \sqrt{n}/s \big) \, , \, \max \Big(s \log\big(1 + \sqrt{n}/s \big) \,, \\
    % &\hspace{2cm} s^2 \log^2\big(1+n/s^2\big) \mathbf{1}(s \leq \sqrt{n})
    % + n \mathbf{1}(s > \sqrt{n}) \Big) \\
    %
    &\asymp \min \Big( s^2 \log^2 \big(1 + \sqrt{n}/s \big) , \max \Big(s \log\big(1 + \sqrt{n}/s \big) , s^2 \log^2\big(1+n/s^2\big) \Big) \mathbf{1}(s \leq \sqrt{n}) \\
    & + \min \Big( s^2 \log^2 \big(1 + \sqrt{n}/s \big) , \max \Big(s \log\big(1 + \sqrt{n}/s \big) , n \Big) \mathbf{1}(s > \sqrt{n}) \\
    &\asymp s^2 \log^2 \big(1+\sqrt{n}/s \big)
    \mathbf{1}(s \leq \sqrt{n})
    + n \mathbf{1}(s > \sqrt{n}) \\
    &\asymp s^2 \log^2 \big(1 + \sqrt{n}/s \big),
\end{align*}
as claimed.

\begin{proof}[Proof of Theorem \ref{thm.sparsity_lb}]
For each $i=1,\dots,n,$ we derive a lower bound for $\Var_0(\wh \theta_i)$ applying \eqref{eq.row_sum_norm_nd}. Denote by $P_\theta$ the distribution of the data in the Gaussian sequence model for the parameter vector $\theta=(\theta_1,\dots,\theta_n).$ Fix an integer $i \in \{1, \dots, n\}.$ There are $M:=\binom{n-1}{s-1}$ distinct vectors $\theta_1^{(i)},\dots,\theta_M^{(i)}$ with exactly $s$ non-zero entries, having a non-zero entry at the $i$-th position and all non-zero entries equal to $ \sqrt{\alpha\log(n/s^2)},$ where  $\alpha := 4\gamma + 1/\log(n/s^2).$ To indicate also the dependence on $i,$ for each $j \in \{1, \dots, M\}$ write $P_{ji}:=P_{\theta_j^{(i)}}$ and $P_0=P_{(0,\dots,0)}.$ 

% \ref{tab.1} (see also \eqref{eq.chi2_matrix_for_normal}),
By \eqref{eq.chi2_normal_distr}, we have that $\chi^2(P_0,\dots,P_M)_{j,k}
= \exp(\langle \theta_j^{(i)}, \theta_k^{(i)} \rangle)-1.$ For fixed $j,$ there are $b(n,s,r):=\binom{s-1}{r-1}\binom{n-s}{s-r}$ among the $M$ vectors $\theta_1^{(i)},\dots,\theta_M^{(i)}$ with exactly $r$ non-zero components with $\theta_j^{(i)}$ in common, that is, $\langle \theta_j^{(i)}, \theta_k^{(i)}\rangle= \alpha r \log(n/s^2).$ Hence,
\begin{align}
    \big\| \chi^2(P_0,P_{1i}\dots,P_{Mi}) \big\|_{1,\infty}
    = \sum_{r=1}^s b(n,s,r)  \Big[\Big(\frac{n}{s^2}\Big)^{r\alpha}-1\Big].
    \label{eq.chisq_row_sum_norm_hd}
\end{align}
Since $s\leq \sqrt{n}/2,$ we have for $r=1,\dots,s-1,$
\begin{align*}
    b(n,s,r+1)
    &= \binom{s-1}{r}\binom{n-s}{s-r-1} \\
    &= \frac{(s-1)!(n-s)!}{r!(s-r-1)!(n-2s+r+1)!(s-r-1)!} \\
    &= \frac{(s-r)^2}{r(n-2s+r+1)} \frac{(s-1)!(n-s)!}{(r-1)!(s-r)!(n-2s+r)!(s-r)!} \\
    &= \frac{(s-r)^2}{r(n-2s+r+1)}b(n,s,r) \\
    &\leq \frac{s^2}{n(1-n^{-1/2})}b(n,s,r).
\end{align*}
Recall that $\alpha = 4\gamma + 1/\log(n/s^2)\leq 0.99.$ Thus, for all sufficiently large $n,$ $(1-n^{-1/2})^{-1}(s^2/n)^{1-\alpha}\leq (1-n^{-1/2})^{-1}(1/4)^{0.01}\leq (1/2)^{0.01}.$ Combined with the recursion formula for $b(n,s,r)$ and the formula for the geometric sum, we obtain
\begin{align}
\begin{split}
    \big\| \chi^2(P_0,P_{1i}\dots,P_{Mi}) \big\|_{1,\infty}
    &\leq b(n,s,1) \Big(\frac{n}{s^2}\Big)^{\alpha} \, \sum_{q=0}^{s-1} \frac{1}{(1-n^{-1/2})^q}\Big(\frac{s^2}{n}\Big)^{q(1-\alpha)} \\
    &\leq b(n,s,1) \frac{1}{1-(1/2)^{0.01}} \Big(\frac n{s^2}\Big)^{\alpha},
\end{split}
\label{eq.row_sum_norm_estimate_hd}
\end{align}    
where the last inequality holds for all sufficiently large $n.$ We must have that $M=\sum_{r=1}^s b(n,s,r)$ and so $b(n,s,1)\leq M.$ Let $\wh \theta=(\wh \theta_1,\dots, \wh \theta_n)$ be an arbitrary estimator for $\theta.$ Applying the inequality~\eqref{eq.row_sum_norm_nd}
% Theorem \ref{thm.multiple_lb}(i)
to the random variable $\wh \theta_i$ yields
\begin{align}
    \sum_{j=1}^M \big(E_{P_{ji}}[\wh \theta_i]-E_{P_0}[\wh \theta_i]\big)^2 \leq M \frac{1}{1-(1/2)^{0.01}} \Big(\frac n{s^2}\Big)^{\alpha} \Var_0\big(\wh \theta_i\big). 
    \label{eq.tech311}
\end{align}
Let $\Mc$ be the set of all $\binom{n}{s}$ distributions $P\sim \Nc(\theta,I_n),$ where the mean vector $\theta$ has exactly $s$ non-zero entries and all non-zero entries equal to $\sqrt{\alpha \log(n/s^2)}.$ For a $P\in \Mc$ denote by $S(P)$ the support (the location of the non-zero entries) of the corresponding mean vector $\theta.$ For $S\subset \{1,\dots, n\},$ define moreover $\wh \theta_S:=(\wh \theta_j)_{j\in S}.$ Summing over $i$ in \eqref{eq.tech311} yields then, 
\begin{align}
    \sum_{j=1}^M \sum_{i=1}^n \big(E_{P_{ji}}[\wh \theta_i]-E_{P_0}[\wh \theta_i]\big)^2
    \leq M \frac{1}{1-(1/2)^{0.01}} \Big(\frac n{s^2}\Big)^{\alpha} \sum_{i=1}^n \Var_0\big(\wh \theta_i\big). 
    \label{eq.tech311_2}
\end{align}
For fixed $P \in \Mc$ and for each $i \in S(P)$, we can find $j \in \{1, \dots, M\}$ such that $P_{ji} = P$, so that the previous equation can be rewritten as
\begin{align}
    \sum_{P\in \Mc} \big\| E_{P}\big[\wh \theta_{S(P)}\big]-E_{P_0}\big[\wh \theta_{S(P)}\big]\big\|_2^2
    &= \sum_{P\in \Mc} \sum_{i \in S(P)}
    \big( E_{P}\big[\wh \theta_i\big]-E_{P_0}\big[\wh \theta_i\big]\big)^2 \nonumber \\
    &\leq M \frac{1}{1-(1/2)^{0.01}} \Big(\frac n{s^2}\Big)^{\alpha} \sum_{i=1}^n \Var_0\big(\wh \theta_i\big).
    \label{eq.tech624}
\end{align}
% \begin{align}
%     \sum_{P\in \Mc} \big\| E_{P}\big[\wh \theta_{S(P)}\big]-E_{P_0}\big[\wh \theta_{S(P)}\big]\big\|_2^2 \leq M \frac{1}{1-(1/2)^{0.01}} \Big(\frac n{s^2}\Big)^{\alpha} \sum_{i=1}^n \Var_0\big(\wh \theta_i\big). 
%     \label{eq.tech624}
% \end{align}
For any $P\in \Mc$ with $P=\Nc(\theta,I_d),$ we obtain using the triangle inequality, $\theta_0=0,$ $\|\theta\|_2=\|\theta_{S(P)}\|_2,$ the bound on the bias, and $\alpha = 4\gamma + 1/\log(n/s^2)\leq 1$ combined with $\sqrt{\alpha}-2\sqrt{\gamma}=(\alpha-4\gamma)/(\sqrt{\alpha}+2\sqrt{\gamma})\geq (\alpha-4\gamma)/5=1/(5 \log(n/s^2)),$
\begin{align*}
  \big\| E_{P}\big[\wh \theta_{S(P)}\big]-E_{P_0}\big[\wh \theta_{S(P)}\big]\big\|_2
  &\geq \|\theta\|_2 - \big\| E_{P}[\wh \theta] -\theta\big\|_2-\big\| E_{P_0}[\wh \theta]\big\|_2 \\
  &\geq \sqrt{s\alpha \log(n/s^2)}- 2\sqrt{\gamma s \log(n/s^2)} \\
  &\geq \sqrt{\frac{s}{25\log(n/s^2)}}.
\end{align*}
Observe that $(n/s^2)^{\alpha}=(n/s^2)^{4\gamma}e$ and that the cardinality of $\Mc$ is $\binom{n}{s}=\tfrac ns \binom{n-1}{s-1}= \tfrac ns M.$ Combining this with \eqref{eq.tech624} yields
\begin{align*}
    \sum_{i=1}^n \Var_0\big(\wh \theta_i\big)
    \geq \frac{(1-(1/2)^{0.01})}{25 e \log(n/s^2)} n \Big(\frac{s^2}{n}\Big)^{4\gamma},
\end{align*}
completing the proof.
\end{proof}

\begin{proof}[Proof of Theorem \ref{thm.sparsity_sq_functional}]
We argue in a similar way as in the proof of Theorem \ref{thm.sparsity_lb}. Again, let $P_0\sim \Nc(0,I_n)$ and denote by $\Mc$ the set of distributions $P\sim \Nc(\theta,I_n)$ with $s$-sparse mean vector $\theta$ and all non-zero entries equal to $\sqrt{\alpha' \log(n/s^2)}$ for $\alpha' = 2\gamma + 1/\log(n/s^2)\leq 0.99.$ Thus, $\sum_{i=1}^n\theta_i^2 \leq s\log(n/s^2)\leq 2s\log(1+\sqrt{n}/s)$ and $\theta\in \Theta_n^2(s).$ Write $P_1,\dots,P_{|\Mc|}$ for an arbitrary enumeration of these $|\Mc|=\binom{n}{s}$ probability measures. Using the same arguments as for \eqref{eq.chisq_row_sum_norm_hd}, $\binom{s}{r}=\tfrac sr \binom{s-1}{r-1},$ $b(n,s,r)= \binom{s-1}{r-1}\binom{n-s}{s-r}$ and arguing as for \eqref{eq.row_sum_norm_estimate_hd}, we find that 
\begin{align*}
    \big\| \chi^2\big(P_0,P_1,\dots,P_{|\Mc|}\big) \big\|_{1,\infty}
    &= \sum_{r=1}^s \binom{s}{r}\binom{n-s}{s-r} \Big[\Big(\frac{n}{s^2}\Big)^{r\alpha'}-1\Big] \\
    &= \sum_{r=1}^s \frac{s}{r} \binom{s-1}{r-1}\binom{n-s}{s-r} \Big[\Big(\frac{n}{s^2}\Big)^{r\alpha'}-1\Big] \\
    &\leq s \sum_{r=1}^s b(n,s,r) \Big[\Big(\frac{n}{s^2}\Big)^{r\alpha'}-1\Big] \\
    &\leq  \frac{s b(n,s,1)}{1-(1/2)^{0.01}} \Big(\frac n{s^2}\Big)^{\alpha'}.
\end{align*}
Recall that also $b(n,s,1) = \binom{n-s}{s-1}
\leq \binom{n-1}{s-1}=M.$ Equation~\eqref{eq.row_sum_norm_nd} applied to the random variable $\wh{\|\theta\|_2^2}$ gives then
\begin{align}
    \sum_{P \in \Mc} \bigg(
    E_P \Big[\wh{\|\theta\|_2^2}\Big] - E_{P_0}\Big[\wh{\|\theta\|_2^2}\Big]\bigg)^2
    \leq  \frac{s M}{1-(1/2)^{0.01}} \Big(\frac n{s^2}\Big)^{\alpha'} \Var_0\big( \wh{\|\theta\|_2^2} \big). 
    \label{eq.quadfunct_3}
\end{align}
For any $P \in \Mc$, we obtain using the triangle inequality, $\theta_0=0$ and the bound on the bias
\begin{align*}
  \bigg| E_P \Big[\wh{\|\theta\|_2^2}\Big] - E_{P_0}\Big[\wh{\|\theta\|_2^2}\Big] \bigg|
  &\geq \|\theta\|_2^2
  - \bigg| E_P \Big[\wh{\|\theta\|_2^2}\Big] - \|\theta\|_2^2 \bigg|
  - \bigg| E_{P_0}\Big[\wh{\|\theta\|_2^2}\Big] \bigg| \\
  &\geq s\alpha' \log(n/s^2) - 2 \gamma s \log(n/s^2) \\
  &\geq s.
\end{align*}
Observe that $(n/s^2)^{\alpha'}=(n/s^2)^{2\gamma}e$. Combining this with \eqref{eq.quadfunct_3} yields
\begin{align*}
    |\Mc|s^2
    \leq   \frac{sM}{1-(1/2)^{0.01}} \Big(\frac n{s^2}\Big)^{2\gamma}e \Var_0\Big( \wh{\|\theta\|_2^2} \Big),
\end{align*}
and together with $|\Mc|=\binom{n}{s}=\tfrac ns \binom{n-1}{s-1}=\tfrac ns M,$ we finally obtain
\begin{align*}
    \Var_0 \Big( \wh{\|\theta\|_2^2} \Big)
    \geq \frac{1-(1/2)^{0.01}}{e} n
    \Big(\frac{s^2}{n}\Big)^{2\gamma},
\end{align*}
completing the proof.

\end{proof}

\begin{proof}[Proof of Theorem \ref{thm.functional_GSM}]
We show that the result already holds if $\Theta(s)$ is replaced by $\Theta_{[s]}:=\{(\theta_1,\dots,\theta_s,0,\dots,0)^\top:\theta_i \in \mathbb{R}, i=1,\ldots,s\},$ that is, the space of $s$-sparse vectors with support on the first $s$ components. To simplify the problem, we apply two reductions. In a first step, we prove that for any estimator $\wh \theta$ in the $s$-sparse Gaussian sequence model, there exists a (non-randomized) estimator $\wt \theta$ in the Gaussian sequence model $Z_i=\theta_i +\eps_i,$ $i=1,\dots,s$ with independent $\eps_i \sim \Nc(0,1)$ and parameter space $\Theta=\mathbb{R}^s,$ such that 
\begin{align}
    \sup_{\theta \in \mathbb{R}^s} \, \sum_{i=1}^s \Var_\theta\big(\wt \theta_i\big) 
    \leq \sup_{\theta \in \Theta_{[s]}} \, \sum_{i=1}^n \Var_\theta\big(\wh \theta_i\big)
    \label{eq.red1_var}
\end{align}
and
\begin{align}
    \sup_{\theta \in \mathbb{R}^s} \,
    \big\|E_\theta\big[ \wt \theta \big] - \theta\big\|_2^2
    \leq 
    \sup_{\theta \in \Theta_{[s]}} \,
    \big\|E_\theta\big[ \wh \theta \big]-\theta\big\|_2^2.
    \label{eq.red1_expec}
\end{align}
In the last two inequalities, the expectation and variance of an estimator is always taken with respect to the distributions induced by the corresponding models. 

To prove these inequalities, fix an estimator $\wh \theta$ and observe that $X_1,\dots,X_s$ is a sufficient statistic for the parameter $\theta \in \Theta_{[s]}.$ By the definition of sufficiency, $E_{(\theta_1,\dots,\theta_s,0)^\top}[\wt \theta|X_1,\dots,X_s]$ does not depend on $ (\theta_1,\dots,\theta_s,0,\dots,0)^\top \in \Theta_{[s]}.$ We now interpret the Rao-Blackwell type estimator $$\wt \theta := E_{(\theta_1,\dots,\theta_s,0,\dots,0)^\top}[(\wh \theta_1,\dots,\wh \theta_s)^\top |X_1,\dots,X_s]$$ as an estimator for $(\theta_1,\dots,\theta_s)^\top \in \mathbb{R}^s$ in the Gaussian sequence model $X_i=\theta_i+\eps_i,$ $i=1,\dots,s.$ The inequality in \eqref{eq.red1_expec} follows from $$E_{(\theta_1,\dots,\theta_s)^\top}[\wt \theta]=E_{(\theta_1,\dots,\theta_s,0,\dots,0)^\top}[(\wh \theta_1,\dots,\wh \theta_s)^\top].$$
For $i=1,\dots,s,$ the law of total variance yields 
\begin{align*}
    \Var_{(\theta_1,\dots,\theta_s)^\top}(\wt \theta_i)
    &=\Var_{(\theta_1,\dots,\theta_s, 0,\dots,0)^\top}\big(E_{(\theta_1,\dots,\theta_s,0,\dots,0)^\top}[\wh \theta_i |X_1,\dots,X_s]\big) \\
    &\leq 
    \Var_{(\theta_1,\dots,\theta_s,0,\dots,0)^\top}(\wh \theta_i)
\end{align*}
and \eqref{eq.red1_var} follows.

Recall the definition of a spherically symmetric estimators in Section \ref{sec.reduction}. Arguing as in the proof of Proposition \ref{prop.spheri_symm}, we can find for any estimator $\wt \theta,$ a spherically symmetric estimator $\wt \theta^{(1)}$ with 
\begin{align}
    \sup_{\theta \in \mathbb{R}^s} \, \big\|E_\theta\big[\wt \theta^{(1)}\big]-\theta \big\|_2
    \leq \sup_{\theta \in \mathbb{R}^s} \, \big\|E_\theta\big[\wt \theta\big]-\theta \big\|_2
    \label{eq.red2_expec}
\end{align}
and 
\begin{align}
    \sup_{\theta \in \mathbb{R}^s} \, \sum_{i=1}^s \Var_\theta\big(\wt \theta_i^{(1)} \big)
    \leq 
    \sup_{\theta \in \mathbb{R}^s} \, \sum_{i=1}^s \Var_\theta\big(\wt \theta_i \big).
    \label{eq.red2_var}
\end{align}
Moreover, the same reasoning as for \eqref{eq.IBias_formula} shows that there exists a function $t,$ such that for any $\theta \in \mathbb{R}^s,$ {}{}
\begin{align}
    E_\theta\big[\wt \theta^{(1)}\big] = t(\|\theta\|_2) \theta 
    \ \ 
    \text{and} \ \ 
    \big\|E_\theta\big[\wt \theta^{(1)}\big]-\theta \big\|_2^2=\|\theta\|_2^2 \big(t\big(\|\theta\|_2\big)-1\big)^2.
    \label{eq.exp_red2_seq_model}
\end{align}

The last two reduction steps combined show that if there exists an estimator $\wh \theta$ satisfying $ \sup_{\theta \in \Theta(s)} \, \sum_{i=1}^n \Var_\theta(\wh \theta_i) < s/2,$ then, there exists a spherically symmetric estimator $\wt \theta^{(1)}$ satisfying $\sup_{\theta \in \mathbb{R}^s} \, \sum_{i=1}^s \Var_\theta(\wt \theta_i^{(1)}) < s/2.$ We now show that for such an estimator 
\begin{align}
    \sup_{\theta \in \mathbb{R}^s}
    \, \big\| E_\theta\big[\wt \theta^{(1)}\big]-\theta\big\|_2^2 =\infty.
    \label{eq.exp=infty_to_show}
\end{align}
Because of \eqref{eq.red1_expec} and \eqref{eq.red2_expec}, it then follows that $\sup_{\theta \in \mathbb{R}^s} \, \| E_\theta[\wh \theta]-\theta\|_2^2 =\infty,$ completing the proof.

Let $\rho := (1/2 + s^{-1} \sup_{\theta \in \Theta(s)} \, \sum_{i=1}^n \Var_\theta\big(\wh \theta_i\big))/2.$ Because $\rho$ is the average of $1/2$ and of a quantity that is strictly smaller than $1/2$,
\begin{align*}
    s^{-1} \sup_{\theta \in \Theta(s)} \, \sum_{i=1}^n \Var_\theta\big(\wh \theta_i\big) < \rho < 1/2.
\end{align*}
In order to show \eqref{eq.exp=infty_to_show} we need to treat the cases $s=1$ and $s>1$ separately. If $s=1,$ then, Example \ref{ex.normal_change_of_expec_cts} shows that $|E_\theta[\wt \theta^{(1)}]-E_0[\wt \theta^{(1)}]|\leq |\theta|\sqrt{\rho}.$ Using triangle inequality, this implies that $|\theta-E_{\theta}[\wt \theta^{(1)}]|\geq |\theta|(1-\sqrt{\rho})-|E_0[\wt \theta^{(1)}]|.$ Since $\rho<1$ and by assumption $|E_0[\wt \theta^{(1)}]|<\infty,$ we must have $\lim_{\theta \to \infty} |\theta-E_{\theta}[\wt \theta^{(1)}]|=\infty,$ proving \eqref{eq.exp=infty_to_show} for $s=1.$

We now prove \eqref{eq.exp=infty_to_show} for $s>1.$ Let $a>0$ and $\theta_0=(a,\dots,a)^\top \in \mathbb{R}^s.$ Suppose we can show that
\begin{align}
    \frac s2 t\big(a\sqrt{s}\big) 
    \leq 
    \sum_{i=1}^s \Var_{\theta_0} \big(\wt \theta_i^{(1)}\big).
    \label{eq.var_lb_red2_seq_model}
\end{align}
Together with \eqref{eq.exp_red2_seq_model}, $\sum_{i=1}^s \Var_{\theta_0}(\wt \theta_i^{(1)})\leq s\rho$, and $\rho < 1/2,$ this implies that
$t\big(a\sqrt{s}\big) \leq 2 \rho \leq 1$ and
\begin{align*}
    \big\|E_{\theta_0}\big[\wt \theta^{(1)}\big]
    -\theta\big\|_2^2
    &= \|\theta_0\|_2^2 \big(t\big(\|\theta\|_2\big)-1\big)^2 \\
    &= a^2 s \big(t ( a\sqrt{s} )-1\big)^2
    \geq a^2 s (1-2\rho)^2.
\end{align*}
Taking $a \to \infty$ yields then \eqref{eq.exp=infty_to_show}.

Thus, all what remains is to establish \eqref{eq.var_lb_red2_seq_model}. To show this inequality, we adapt the arguments in the proof of Theorem \ref{thm.LB_L2}. Let $0<\Delta\leq 1/2$ and consider $\theta_i=(\theta_{ij})_{j=1,\dots,s}^\top \in \mathbb{R}^s,$ where we choose $\theta_{ii}:=\sqrt{1+\Delta} \, a$ and $\theta_{ij}:=\sqrt{1-\Delta/(s-1)} \, a$ for $1 \leq j\neq i \leq s.$ Because of $s>1,$ the latter is well-defined. By construction, $\|\theta_i\|_2=a\sqrt{s}$ for all $i=0,1,\dots,s.$ Using that $\sqrt{1+u}-1=u/2+O(u^2)$ for $u \to 0,$ we have for $i=1,\dots,s$ and $\Delta \to 0,$
\begin{align*}
    \frac{\|\theta_i-\theta_0\|_2^2}{\Delta^2} 
    &=\frac{a^2}{\Delta^2}
    \bigg[
    (s-1) \Big( \sqrt{1-\frac{\Delta}{s-1}}-1\Big)^2+\Big(\sqrt{1+\Delta}-1\Big)^2\bigg] \\
    &= \frac{a^2}{4}\Big(1+\frac{1}{s-1}\Big) +O(\Delta).
\end{align*}
Similarly, for $j,\ell=1,\dots,s,$ $j\neq \ell$ and $\Delta \to 0,$
\begin{align*}
    \frac{\big\langle \theta_\ell-\theta_0, \theta_j-\theta_0\big \rangle }{\Delta^2}
    &=
    \frac{a^2}{\Delta^2}
    \bigg[
    (s-2) \Big( \sqrt{1-\frac{\Delta}{s-1}}-1\Big)^2 \\
    &\quad +2\Big(\sqrt{1+\Delta}-1\Big)\Big(\sqrt{1-\frac{\Delta}{s-1}}-1\Big)\bigg] \\
    &=
    -\frac{a^2}{4(s-1)}\Big(1+\frac{1}{s-1}\Big)+O(\Delta).
\end{align*}
% \begin{align*}
%     &\frac{a^2}{\Delta^2}
%     \bigg[
%     (s-2) \Big( \sqrt{1-\frac{\Delta}{s-1}}-1\Big)^2
%     +2\Big(\sqrt{1+\Delta}-1\Big)
%     \Big(\sqrt{1-\frac{\Delta}{s-1}}-1\Big)\bigg] \\
%     &= \frac{a^2}{\Delta^2}
%     \bigg[ (s-2) \frac{\Delta^2}{4 (s-1)^2}
%     + \frac{-2\Delta^2}{4 (s-1)}\bigg] \\
%     &= \frac{a^2}{4 (s-1)}
%     \bigg[ \frac{s-2}{s-1} - 2 \bigg] \\
%     &= \frac{-a^2}{4 (s-1)}
%     \bigg[ 1 + \frac{1}{s-1} \bigg]
% \end{align*}
Recall that $\|\theta_i\|_2=\sqrt{s}a$ by construction. Applying \eqref{eq.row_sum_norm_nd} to the random variables $\wt \theta_i^{(1)}$ and using \eqref{eq.exp_red2_seq_model} and \eqref{eq.chi2_normal_distr}
% \eqref{eq.chi2_matrix_for_normal}
yields
\begin{align*}
    t(\sqrt{s}a) \sum_{j=1}^{s}  \|\theta_j-\theta_0\|_2^2
    &=
    \sum_{j=1}^{s} \big\|E_{\theta_j}\big[\wt \theta^{(1)} \big]- E_{\theta_0}\big[\wt \theta^{(1)} \big]\big\|_2^2 \\
    &\leq \max_{\ell=1,\ldots,s} \sum_{j=1}^{s} \big |e^{\langle \theta_\ell-\theta_0, \theta_j-\theta_0\rangle }-1 \big| \sum_{i=1}^s \Var_{\theta_0}\big(\wt \theta_i^{(1)}\big).
\end{align*}
Multiplying both sides of the inequality with $\Delta^{-2},$ using the expressions for $\Delta^{-2}\|\theta_i-\theta_0\|_2^2$ and $\Delta^{-2}\langle \theta_i-\theta_0, \theta_j-\theta_0\rangle,$ and letting $\Delta$ tend to zero yields \eqref{eq.var_lb_red2_seq_model}.

This completes the proof.
\end{proof}

\begin{proof}[Proof of Lemma \ref{lem.soft_thresh_ub}]
Let $T:=\sqrt{\gamma \log(n/s^2)}$ denote the truncation value. Using that $\wh \theta$ is unbiased under $\theta=(0,\dots,0)^\top$ and applying substitution twice, we find for any $i=1,\dots,n$ 
\begin{align}
\begin{split}
    \Var_0\big(\wh \theta_i\big)
    &= \sqrt{\frac{2}{\pi}} \int_{T}^\infty (x-T)^2 e^{-\frac{x^2}{2}} \, dx
    = \sqrt{\frac{2}{\pi}} \int_0^\infty x^2 e^{-\frac{(x+T)^2}{2}} \, dx \\
    &\leq \sqrt{\frac{2}{\pi}} e^{-\frac{T^2}2} \int_0^\infty x^2 e^{-xT} \, dx 
    = \frac{\sqrt{2}}{\sqrt{\pi} T^3} e^{-\frac{T^2}2} \int_0^\infty y^2 e^{-y} \, dy \\
    &= \frac{\sqrt{2}}{\sqrt{\pi} T^3} e^{-\frac{T^2}2}.
\end{split}    
    \label{eq.var_thetai=0}
\end{align}
Summing over $i$ and inserting the expression for $T$ yields \eqref{eq.soft_thresh_ub1}.

We now prove \eqref{eq.soft_thresh_ub2}. Using Cauchy-Schwarz and the inequality $2\sqrt{ab}\leq a+b,$ $a,b \geq 0,$ we find that for any random variables $Y,Z,$ $\Var(Y+Z)=\Var(Y)+2\Cov(Y,Z)+\Var(Z) \leq \Var(Y)+2\sqrt{\Var(Y)\Var(Z)}+\Var(Z)\leq 2\Var(Y)+2\Var(Z).$ 

Moreover, for any random variable $U,$ we have $\Var(U)\geq \Var(U \wedge 0)+\Var(U\vee 0).$ To see this observe that $E[U\vee 0]\geq 0$ and $E[U \wedge 0]\leq 0$. Hence, $E[U\vee 0]E[U \wedge 0]\leq 0.$ Since also $E[(U\vee 0)(U\wedge 0)]=0,$ we have $\Cov(U\vee 0,U\wedge 0)\geq 0$ and $\Var(U)=\Var((U\vee 0)+(U\wedge 0))\geq \Var(U \wedge 0)+\Var(U\vee 0).$

As a last ingredient of the proof, observe that for a random variable $X,$ we have  $\sign(X)(|X|-T)_+
= (X-T)\mathbf{1}(X\geq T) + (X+T)\mathbf{1}(X\leq -T)
= (X-T)\vee 0 +(X+T) \wedge 0.$ Thus, with $T=\sqrt{\gamma \log(n/s^2)},$
\begin{align}
\begin{split}
    \Var_\theta(\wh \theta_i) 
    &\leq 2\Var_\theta\big( (X_i-T)\vee 0\big)
    + 2\Var_\theta\big( (X_i+T)\wedge 0\big) \\
    &\leq 2\Var_\theta(X_i-T)+2\Var_\theta(X_i+T) \\
    &= 4\Var(X_i) \\
    &=4.
\end{split}\label{eq.min_max_decomp_var}
\end{align}
Applying this inequality for all $i$ with $\theta_i \neq 0,$ and \eqref{eq.var_thetai=0} for all $i$ with $\theta_i=0,$ \eqref{eq.soft_thresh_ub2} follows.

We now study the estimator $\wh{\|\theta\|_2^2}$ for the functional $\|\theta\|_2^2.$ By assumption, $\gamma \log(n/s^2) \geq 2,$ and therefore, $|(X_i^2-\gamma \log(n/s^2))_+ -(X_i^2-1)|\leq \gamma \log(n/s^2)-1.$ Moreover, $E[(\xi^2 - \gamma\log(n/s^2))_+]\leq E[\xi^2]=1.$ With $S=\{i:\theta_i\neq 0\}$ the support of $\theta,$ and since $E_{\theta_i}[X_i^2-1]=\theta_i^2,$ we find for suitable numbers $\eta_i$ satisfying $|\eta_i|\leq \gamma \log(n/s^2),$ that
\begin{align*}
    E_\theta\big[\wh{\|\theta\|_2^2}\big] 
    &= 
    \sum_{i \in S} 
    \Big( E_{\theta_i}\big[(X_i^2 - \gamma\log(n/s^2)\big)_+\big]- E\big[\big(\xi^2 - \gamma\log(n/s^2)\big)_+\big] \Big) \\
    &= \sum_{i \in S}  E_{\theta_i}[X_i^2-1] + \eta_i \\
    &= \|\theta\|_2^2 + \sum_{i \in S} \eta_i.
\end{align*}
Thus, $\sup_{\theta \in \Theta(s)}|\Bias_\theta(\wh{\|\theta\|_2^2})|\leq \gamma s \log(n/s^2),$ proving \eqref{eq.soft_thresh_ub3}. 

Let $\eps \sim \Nc(0,1).$ Then, $\Var((\eps^2 - \gamma\log(n/s^2))_+)\leq  E[(\eps^2 - \gamma\log(n/s^2))_+^2].$ Let $u\geq 0.$ Substituting $y=(x^2-u)/2,$ and using that $dx=dy/\sqrt{2y+u}\leq dy/\sqrt{u}$ for $y\geq 0,$ we have
\begin{align*}
    \int_{\sqrt{u}}^\infty \big(x^2-u\big)^2 e^{-x^2/2} \, dx
    &\leq 
    \frac{4}{\sqrt{u}} e^{-u/2} \int_0^\infty y^2 e^{-y} \, dy
    = \frac{8}{\sqrt{u}} e^{-u/2},
\end{align*}
where the last step follows from the fact that the second moment of a standard exponential distribution is $2.$ Thus, with $u=\gamma \log(n/s^2),$
\begin{align}\label{eq.var_fctal_zero_components}
    \Var\Big(\big(\eps^2 - \gamma\log(n/s^2)\big)_+\Big)
    \leq \frac{8}{\sqrt{\gamma \log(n/s^2)}} \Big(\frac{s^2}{n}\Big)^{\gamma/2}.
\end{align}
Arguing as in \eqref{eq.min_max_decomp_var},
\begin{align}
\begin{split}
    \Var\Big(\big(X_i^2 - \gamma\log(n/s^2)\big)_+\Big)
    &\leq 4\Var(X_i^2)=4\Var\big(2\theta_i\eps_i+\eps_i^2\big) \\
    &\leq E\big[\big(2\theta_i\eps_i+\eps_i^2\big)^2\big] \\
    &=4\theta_i^2+3.
\end{split}\label{eq.var_fctal_nonzero_components}    
\end{align}
Recall that $S$ denotes the support of the vector $\theta.$ Combining \eqref{eq.var_fctal_zero_components} and \eqref{eq.var_fctal_nonzero_components}, we obtain 
\begin{align*}
    \Var_\theta\big(\wh{\|\theta\|_2^2}\big)
    &= \sum_{i\in S} \Var\Big(\big(X_i^2 - \gamma\log(n/s^2)\big)_+\Big)
    +\sum_{i\in S} \Var_0\Big(\big(\eps_i^2 - \gamma\log(n/s^2)\big)_+\Big) \\
    &\leq \|\theta\|_2^2+3s+
    \frac{8n}{\sqrt{\gamma \log(n/s^2)}} \Big(\frac{s^2}{n}\Big)^{\gamma/2}.
\end{align*}
This completes the proof for \eqref{eq.soft_thresh_ub5}. Finally \eqref{eq.soft_thresh_ub4} can be proved along the same lines using that the set $S$ is empty.
\end{proof}

\begin{proof}[Proof of Theorem \ref{thm.main_minimax_est_vector_hd}]
    Let $\wh \theta$ be an estimator attaining the minimax rate of convergence $s \log(n)$ with respect to the worst case risk $\sup_{\theta \in \Theta(s)} E_\theta[\|\wh \theta-\theta\|_2^2]$.
    Using decomposition \eqref{eq.BV_vectors}, we obtain
    \begin{align}
        \sup_{\theta\in \Theta(s)}\big\|E_\theta\big[\wh \theta\big] - \theta\big\|_2^2 = O \big(s \log(n) \big).
        \label{eq:bound_bias_gaussSeq}
    \end{align}
    Arguing by contradiction, assume that $\sup_{\theta\in \Theta(s)}\big\|E_\theta\big[\wh \theta\big] - \theta\big\|_2^2 \leq s \log(n/s^2)/40$.
    By Theorem~\ref{thm.sparsity_lb}, for $n$ large enough,
    $\sum_{i=1}^n \Var_0\big(\wh \theta_i\big)
    \gtrsim n^{9/10} s^{2/10} / \log(n/s^2),$ contradicting the assumption that $\wh \theta$ attains the rate $s \log(n)$ in the regime $s=o(\sqrt{n})$ (using decomposition \eqref{eq.BV_vectors}).
    This combined with \eqref{eq:bound_bias_gaussSeq} proves that $\sup_{\theta\in \Theta(s)}\big\|E_\theta\big[\wh \theta\big] - \theta\big\|_2^2 \asymp s\log(n)$.
    
    Arguing by contradiction, assume now that $\sup_{\theta\in \Theta(s)}\sum_{i=1}^n \Var_\theta\big(\wh \theta_i\big) < \frac s2$ for some values of $n$. Then, for the same values of $n$, by Theorem~\ref{thm.functional_GSM}, we have
    \begin{align*}
        \sup_{\theta \in \Theta(s)} \,
        \big\| E_\theta\big[\wh \theta\big]-\theta\big\|_2 = \infty,
    \end{align*}
    contradicting \eqref{eq:bound_bias_gaussSeq}. This shows the second part of Theorem \ref{thm.main_minimax_est_vector_hd}.
    
    To prove the last assertion of the theorem, recall that we additionally assume for this part $s\leq n^{1/2-\delta}$ for some $0<\delta<1/2$. Let $\gamma := (2 \delta + 1) / (2 \delta)$. For this choice of $\gamma$, the soft-thresholding estimator defined in Equation~\eqref{eq.soft_thresh_est} satisfies the required upper bound on the variance. Indeed, by Lemma~\ref{lem.soft_thresh_ub}, we have
    \begin{align*}
        \sup_{\theta \in \Theta(s)}
        \sum_{i=1}^n\Var_\theta\big(\wh \theta_i\big)
        \lesssim s + \frac{1}{\sqrt{\log^3(n/s^2)}}
        n \Big(\frac{s^2}{n}\Big)^{\frac{\gamma}{2}}
        \leq s + \frac{s}{\sqrt{\log^3(n/s^2)}}
        \lesssim s,
    \end{align*}
    where the second inequality results from the fact that $n(s^2/n)^{\gamma/2} \leq s$ if and only if $s \leq n^{(\gamma-2)/(2\gamma-2)}$, which is the case here since $(\gamma - 2) / (2 \gamma - 2) = 1/2 - \delta$ (by construction of $\gamma$) and because we assume $s\leq n^{1/2-\delta}$.
\end{proof}

\begin{proof}[Proof of Theorem~\ref{thm.main_minimax_est_fctal_hd}]
    \textit{Proof of (i):} It is proved in Appendix \ref{sec:proofs_highdimensional} that the minimax estimation rate is $s^2\log^2(n/s^2)$ in this framework where $s \ll \sqrt{n}$.
    
    Let $\wh{\|\theta\|_2^2}$ be an estimator attaining the minimax optimal estimation rate $s^2\log(n/s^2)$ in the regime $s \ll \sqrt{n}$. By Theorem~\ref{thm.sparsity_sq_functional}, if \begin{align*}
        \sup_{\theta \in \Theta_n^2(s)} \big|\Bias_\theta\big(\wh{\|\theta\|_2^2}\big)\big| \leq \frac{1}{20} s \log\Big(\frac{n}{s^2}\Big),
    \end{align*}
    then for $n$ large enough, $\Var_0 ( \wh{\|\theta\|_2^2} )\gtrsim n^{9/10} s^{2/10},$ contradicting the assumption that $\wh{\|\theta\|_2^2}$ attains the rate $s^2\log(n/s^2)$ in the regime $s \ll \sqrt{n}$.
    Therefore, $\wh{\|\theta\|_2^2}$ has to satisfy
    \begin{align*}
        \sup_{\theta\in \Theta_n^2(s)}\big(E_\theta\big[\wh{\|\theta\|_2^2}\big] -\|\theta\|_2^2\big)^2\asymp s^2\log^2\Big(\frac n{s^2}\Big).
    \end{align*}
    
    We now assume that $s\leq n^{1/2-\delta}$ for some $0<\delta<1/2.$ As an estimator for the functional $\|\theta\|_2^2$, we choose the estimator $\wh{\|\theta\|_2^2}$ defined in Equation~\eqref{eq.soft_thresh_est_func} with the choice $\gamma := (2 \delta + 1) / (2 \delta)$.
    By Lemma~\ref{lem.soft_thresh_ub}, we have
    \begin{align*}
        \sup_{\theta \in \Theta_n(s)}|\Bias_\theta(\wh{\|\theta\|_2^2})|\leq \gamma s \log(n/s^2),
    \end{align*}
    and 
    \begin{align*}
        \sup_{\theta\in \Theta_n^2(s)} \Var_\theta\big(\wh{\|\theta\|_2^2}\big)
        &\lesssim \|\theta\|_2^2 + s +
        \frac{1}{\sqrt{\log(n/s^2)}}
        n \Big(\frac{s^2}{n}\Big)^{\gamma/2} \\
        &\leq 2s\log\Big(1+\frac{\sqrt{n}}s\Big) + s +
        \frac{s}{\sqrt{\log(n/s^2)}} \\
        &\lesssim s\log(n/s^2),
    \end{align*}
    where the second inequality results from the fact that $n(s^2/n)^{\gamma/2} \leq s$ if and only if $s \leq n^{(\gamma-2)/(2\gamma-2)}$, which is the case here since $(\gamma - 2) / (2 \gamma - 2) = 1/2 - \delta$ (by construction of $\gamma$) and because we assume that $s\leq n^{1/2-\delta}$. Combining both bounds on worst-case squared bias and variance, we conclude that our estimator attains the minimax rate $s^2 \log^2(n/s^2)$, as claimed.

    \textit{Proof of (ii):} Consider the estimator  $\wh{\|\theta\|_2^2} := \sum_{i=1}^n X_i^2-1$. Then $E_\theta \wh{\|\theta\|_2^2}]= \sum_{i=1}^n E_\theta[X_i^2] - n= \sum_{i=1}^n \theta_i^2,$ and therefore $\wh{\|\theta\|_2^2}$ is unbiased.
    It remains to show that the variance of this estimator is of the order $n$.
    Note that $\wh{\|\theta\|_2^2}$ follows a non-central chi-square distribution with $n$ degrees of freedom and non-centrality parameter $\sum_{i=1}^n \theta_i^2$. This gives for the variance
    \begin{align*}
        \Var_\theta(\wh{\|\theta\|_2^2})
        = 2 \Big(n + 2 \sum_{i=1}^n \theta_i^2 \Big)
        \lesssim n + s \log \Big(1 + \frac{\sqrt{n}}s\Big)
        \lesssim n,
    \end{align*}
    uniformly over $\Theta_n$ in the regime $s \gtrsim \sqrt{n}$.
\end{proof}

\begin{proof}[Proof of Theorem \ref{thm.sparseReg_lb}]
The proof is a variation of the proof of Theorem \ref{thm.sparsity_lb} with $n$ replaced by $p.$ To comply with standard notation, the parameter vectors are denoted by $\beta$ and therefore all the symbols $\theta$ in the proof of Theorem \ref{thm.sparsity_lb} have to be replaced by $\beta.$ In particular the vectors $\theta_j$ are now denoted by $\beta_j.$ Because of the standardization of the diagonal entries in the Gram matrix, we need to choose the non-zero components of $\beta_j$ as $\sqrt{\alpha \log(p/s^2)/n}.$ Compared with the proof of Theorem \ref{thm.sparsity_lb}, the main difference is that the entries of the $\chi^2$-divergence matrix are bounded as follows
\begin{align*}
    \chi^2(P_0,\dots,P_M)_{j,k}
    &=\exp(\beta_j^\top X^\top X\beta_k)-1 \\
    &\leq \exp(n \beta_j^\top \beta_k +n s^2 \mc(X)\|\beta_j\|_\infty\|\beta_k\|_\infty) \\
    &\leq \exp(n \beta_j^\top \beta_k +\alpha) \\
    &\leq \exp(n \beta_j^\top \beta_k +1),    
\end{align*}
where the first inequality follows from separating the diagonal from the off-diagonal entries and exploiting that the vectors are $s$-sparse and the second inequality uses that the maximum entry norm $\|\cdot\|_\infty$ is bounded by construction of the vectors $\beta_j,\beta_k$ by $\sqrt{\alpha \log(p/s^2)/n}.$ Thus, following exactly the same steps as in the proof of Theorem \ref{thm.sparsity_lb}, we can derive that in analogy with  \eqref{eq.tech311}, 
\begin{align*}
        \sum_{j=1}^M \big(E_{P_{ji}}[\wh \beta_i]-E_{P_0}[\wh \beta_i]\big)^2 \leq M \frac{e}{1-(1/2)^{0.01}} \Big(\frac p{s^2}\Big)^{\alpha} \Var_0\big(\wh \beta_i\big).
\end{align*}
The remainder of the proof is also nearly the same as the one for Theorem \ref{thm.sparsity_lb}. The only real difference is that $\|\beta_j\|_2^2$ and the upper bound on the bias are smaller by a factor $1/n,$ which consequently also occurs in the lower bound on the variance. 
\end{proof}

\bigskip

% \bibliographystyle{abbrv}
% \bibliography{biblio}{}

% \bigskip

\end{document}